\documentclass[11pt]{article}%{ctexart}
\usepackage{amssymb}
\usepackage{amsmath}
\usepackage{amsthm}
\usepackage{graphicx}
\usepackage[margin=0.5in]{geometry}
\usepackage{bm}
\usepackage{amsthm}
\usepackage{amscd}
\usepackage{amsbsy}
\usepackage{amsmath}
\usepackage{amsfonts}
\usepackage{amssymb}
\usepackage{calligra}
\usepackage{xcolor}
\usepackage{float}
\usepackage[T1]{fontenc}
\usepackage{multirow}
\usepackage{mathrsfs}
\usepackage{mathtools}
\usepackage{graphicx}
\usepackage{epstopdf}
\usepackage{subfigure}
\usepackage{booktabs}
\usepackage{listings}
\usepackage{longtable}
\usepackage{latexsym}
\usepackage{arydshln}
\usepackage{enumerate}
\usepackage{epsfig}
\usepackage{cite}
\usepackage{verbatim}
\usepackage{overpic}
\usepackage{abstract}
\allowdisplaybreaks[4]

\newtheorem {theorem*}{Theorem}
\newtheorem {theorem} {Theorem}

\newtheorem{lemma}{Lemma}
\newtheorem{corollary}{Corollary}

\numberwithin{equation}{section}
\numberwithin{lemma}{section}
\numberwithin{theorem}{section}
\numberwithin{proposition}{section}
\numberwithin{corollary}{section}

\usepackage[colorlinks,
            CJKbookmarks=true,
            linkcolor=blue,
            anchorcolor=red,
             urlcolor=blue,
            citecolor=blue
            ]{hyperref}

\begin{document}
\arraycolsep=1pt

\title{\Large\bf Bounded, compact and Schatten class Hankel operators on Fock-type spaces
\footnotetext{\hspace{-0.35cm}
			\endgraf{\it E-mail: zhichengzeng@e.gzhu.edu.cn (Zhicheng Zeng)}
			\endgraf \hspace{1.1cm} {\it wxf@gzhu.edu.cn (Xiaofeng Wang)}
			\endgraf \hspace{1.1cm} {\it huzj@zjhu.edu.cn (Zhangjian Hu)}
			\endgraf X. Wang was supported in part by the National Natural Science Foundation of China (11971125) and Z. Hu was supported in part by the National Natural Science Foundation of China (12071130, 12171150).
	}}
	\author{Zhicheng Zeng$^{1}$, Xiaofeng Wang$^{1, }$\thanks{Corresponding author}, Zhangjian Hu$^{2}$\\
		\small \em 1. School of Mathematics and Information Science and Key Laboratory of Mathematics and Interdisciplinary\\ \small \em Sciences of the Guangdong Higher Education Institute, Guangzhou University, Guangzhou, China\\
		\small \em 2. Department of Mathematics, Huzhou University, Huzhou, Zhejiang, China}

\date{ }
\maketitle

\vspace{-0.8cm}

\begin{center}
\begin{minipage}{16cm}\small
{\noindent{\bf Abstract} \quad In this paper, we consider Hankel operators, with locally integrable symbols, densely defined on a family of Fock-type spaces whose weights are $C^3$-logarithmic growth functions with mild smoothness conditions. It is shown that a Hankel operator is bounded on such a Fock space if and only if its symbol function has bounded distance to analytic functions BDA which is initiated by Luecking(J. Funct. Anal. 110:247-271, 1992). We also characterize the compactness and Schatten class membership of Hankel operators. Besides, we give characterizations of the Schatten class membership of Toeplitz operators with positive measure symbols for the small exponent $0<p<1$. Our proofs depend strongly on the technique of H\"{o}mander's $L^2$ estimates for the $\overline{\partial}$ operator and the decomposition theory of BDA spaces as well as integral estimates involving the reproducing kernel.
\endgraf{\bf Mathematics Subject Classification (2010).}\quad Primary 47B35, 30H20; Secondary 32A36, 32A37.
\endgraf{\bf Keywords.}\quad Fock-type space, Hankel operator, Schatten class.}
\end{minipage}
\end{center}

\section{Introduction}\label{s1}
Let $\Psi:[0,\infty)\to[0,\infty)$ be a $C^3$-weight function such that
\begin{align}\label{cdt1}
  \Psi '(x)>0,\quad \Psi ''(x)\geq 0\textup{ and } \Psi '''(x)\geq 0 \textup{ for } x\in [0,\infty).
\end{align}
We refer to such a function as a logarithmic growth function. Suppose that there exists a real number $\eta<1/2$ satisfying
\begin{align}\label{cdt2}
  \Phi''(x)=O(x^{-\frac{1}{2}}[\Phi'(x)]^{1+\eta})\textup{ as } x\to \infty,
\end{align}
where $\Phi(x):=x\Psi '(x)$. Let $n\geq 1$ be an integer. When $n>1$, we also assume that
\begin{align}\label{cdt3}
  \Psi''(x)=O(x^{-\frac{1}{2}}[\Psi'(x)]^{1+\eta})\textup{ as } x\to \infty,
\end{align}
for some $\eta<1/2.$
 Let $\mathbb{C}^n$ be the $n$-dimensional complex Euclidean space and $dV$ be the usual Lebesgue measure on $\mathbb{C}^n$. Set $d\mu_\Psi(z)=e^{-\Psi(|z|^2)}\,dV(z)$. The weighted Lebesgue space $L^2_\Psi$ consists of all Lebesgue measurable functions $f$ on $\mathbb{C}^n$ with the norm
$$\|f\|_\Psi:=\left\{\int_{\mathbb{C}^n}|f(z)|^2\,d\mu_\Psi(z)\right\}^{\frac{1}{2}}<\infty.$$
The inner product in $L^2_\Psi$ is given by
$$\langle f,g\rangle=\int_{\mathbb{C}^n}f(z)\overline{g(z)}d\mu_\Psi(z).$$
The Fock-type space $F^2_\Psi$ consists of all holomorphic functions in $L^2_\Psi$ with inherited norm and inner product. The space $F_\Psi^2$ was first studied in \cite{SY}. Let $K(z,w)$ be the reproducing kernel. For simplicity, we write $K_w(z):=K(z,w),$ and denote by $k_w$ the normalized kernel of $K_w$. Let $P$ be the orthogonal projection from $L^2_\Psi$ onto $F^2_\Psi$, then
$$Pg(z)=\int_{\mathbb{C}^n}g(w)K(z,w)\,d\mu_\Psi(w)\textup{ for } g\in L^2_\Psi,\,z\in \mathbb{C}^n.$$
%The above formula extends $P$ to the set of all functions $g$ satisfying $gK(z,\cdot)\in L^1(\mathbb{C}^n,d\mu_\Psi)$ for every $z\in\mathbb{C}^n$.
Let $\Gamma=\textup{span}\{K_z:z\in \mathbb{C}^n\}$, then $\Gamma$ is dense in $F^2_\Psi$ \cite{WTH}. Consider the class of symbols
$$\mathcal{S}=\{f\textup{ measurable on }\mathbb{C}^n: fg\in L^2_\Psi \textup{ for } g\in \Gamma\}.$$
%Let $\mathcal{S}$ be the set of all functions $f\in L^1_{\mathrm{loc}}(\mathbb{C}^n)$ such that
%$$\int_{\mathbb{C}^n}|g(w)f(w)K(z,w)|\,d\mu_\Psi(w)<\infty,$$
%for all $g\in \Gamma$ and $z\in \mathbb{C}^n$.
Given $f\in \mathcal{S}$, define the operator $H_f$ with domain $\Gamma$ as follows:
$$H_f(g)(z):=\int_{\mathbb{C}^n}K(z,w)g(w)[f(z)-f(w)]\,d\mu_\Psi(w).$$
This densely defined operator from $F^2_\Psi$ to $L^2_\Psi$ is called the Hankel operator $H_f$ with symbol $f$. Obviously, the symbol class $\mathcal{S}$ contains all essentially bounded measurable functions.

During the past few decades much effort has been devoted to the study of Hankel operators. See \cite{B,B2,HL,HW,BC,CO,HV,HV2,I,SY,TW,WCZ,WTH,XZ,Z} and the references therein in the setting of Fock spaces. See for instance \cite{AFJ,AFP,BCZ,PPR,PZZ,ZWH,FX,L,LL,LR,M,P,X,Zhu2,Zhu} in the context of Bergman spaces. In this paper, we focus on the problem of characterizing the boundedness of Hankel operators with locally integrable symbols acting on Fock-type spaces. More explicitly, we characterize the symbol function $f\in \mathcal{S}$ such that the Hankel operator $H_f:F^2_\Psi \to L^2_\Psi$ is bounded. We also study the compactness and  membership in the Schatten classes of Hankel operators on $F^2_\Psi$. As a byproduct, we give characterizations of the Schatten class membership of Toeplitz operators with positive measure symbols for the small exponent $0<p<1$.

We need to introduce some notations in order to state our results precisely. For every $z \neq 0$, let $P_z$ denote the orthogonal projection in $\mathbb{C}^n$ onto the complex line $\{\xi z:\xi\in \mathbb{C}\}$. We write $P_0$ for the identity map.
For any $0<r<\infty$ and $z\in\mathbb{C}^n$, the set $D(z,r)$ is defined by
	\begin{align*}
		D(z,r):=\left\{w:|z-P_z(w)|\leq r\Phi'(|z|^2)^{-\frac{1}{2}},|w-P_zw|\leq r\Psi'(|z|^2)^{-\frac{1}{2}}\right\}.
	\end{align*}
Let $L^2_{\mathrm{loc}}(\mathbb{C}^n)$ be the collection of square locally Lebesgue integrable functions on $\mathbb{C}^n$. For $f \in L^2_{\mathrm{loc}}(\mathbb{C}^n)$ and $z \in \mathbb{C}^n$, we define the function $G_r(f)(z)$ as follows:
	$$G_r(f)(z)=\inf \left\{\left(\frac{1}{|D(z,r)|}\int_{D(z,r)}|f-h|^2dV\right)^{1/2}: h \in H(D(z,r))\right\},$$
	where $H(D(z,r))$ is the set of all holomorphic functions on $D(z,r)$ and $|D(z,r)|$ is the Euclidean volume of $D(z,r)$. For any $0<r<\infty$, the space $\textup{BDA}_{\Psi,r}$ of bounded distance to analytic functions consists of all $f\in L^2_{\mathrm{loc}}(\mathbb{C}^n)$ such that
$$\|f\|_{\textup{BDA}_{\Psi,r}}=\sup_{z\in\mathbb{C}^n}G_r(f)(z)<+\infty.$$
The space $\textup{VDA}_{\Psi,r}$ consists of all $f$ in $ \textup{BDA}_{\Psi,r}$ such that $G_r(f)(z)\to 0$ as $z\to \infty.$

We now state our main results on boundedness and compactness of Hankel operators.
\begin{theorem}\label{1p1}
		 Suppose $f\in \mathcal{S}$. Then $H_f:F^2_\Psi \to L^2_\Psi$ is bounded if and only if $f \in \textup{BDA}_{\Psi,r}$ for some $0<r<\infty.$ Furthermore,
\begin{align}\label{1e4}
  \|H_f\|\simeq \|f\|_{\textup{BDA}_{\Psi,r}}.
\end{align}
\end{theorem}
\begin{theorem}\label{1p2}
		 Suppose $f\in \mathcal{S}$. Then $H_f:F^2_\Psi \to L^2_\Psi$ is compact if and only if $f \in \textup{VDA}_{\Psi,r}$ for some $0<r<\infty.$
\end{theorem}
The notion of bounded distance to analytic functions BDA was introduced by Luecking \cite{L} in the context of the Bergman space. Hu and Virtanen \cite{HV} further generalized it to the notion of integral distance to analytic functions $\text{IDA}$. Note that $\text{BDA}$ in \cite{L} and $\text{IDA}$ in \cite{HV} are essentially defined by Bergman metric balls. In our case, however, it is more convenient to use the balls $D(z,r)$ to give the definition of $\text{BDA}_{\Psi,r}$. Indeed, all of them are equivalent to Bergman metric balls. We see that, if $f\in \text{BDA}_{\Psi,r_0}$ for some $r_0>0$, then $f\in \text{BDA}_{\Psi,r}$ for all $0<r\leq r_0.$ It is interesting to ask whether the space $\text{BDA}_{\Psi,r}$ is independent of the choice of $r$. We refer to \cite{HV} for this problem. We do not address it in our present work.

In the case when $f$ is holomorphic, Seip and Youssfi \cite{SY} characterized the boundedness and compactness of the Hankel operator $H_{\overline{f}}:F^2_\Psi \to L^2_\Psi$ in terms of the mean oscillation of the symbol function $f$. More specifically, for any function $f$ and $z\in \mathbb{C}^n$, the mean oscillation of $f$ at $z$ is defined by
$$(\textup{MO }f)(z):=\sqrt{\widetilde{|f|^2}(z)-|\widetilde{f}(z)|^2},$$
where $\widetilde{f}(z):=\langle fk_z,k_z\rangle$ is the Berezin transform of $f$. The space $\textup{BMO}$ is the set of all functions $f$ on $\mathbb{C}^n$ such that $\widetilde{|f|^2}(z)$ is finite for every $z\in \mathbb{C}^n$ and
$\textup{MO }f$ is bounded on $\mathbb{C}^n.$ The space $\textup{VMO}$ consists of all $f\in \textup{BMO}$ with the property $(\textup{MO }f)(z)\to 0$ as $z\to \infty.$
It was shown in \cite{SY} that, for $f$ holomorphic, $H_{\overline{f}}$ is bounded if and only if $f\in \textup{BMO}$; $H_{\overline{f}}$ is compact if and only if $f\in \textup{VMO}.$

Wang et al. \cite{WCZ} and Tu \cite{TW} considered more general symbol functions that are not necessary to be anti-holomorphic and extended some results of \cite{SY} by using the tool of local mean oscillation. Given $f\in L^2_{\mathrm{loc}}(\mathbb{C}^n,dV)$, the local mean oscillation of $f$ at $z$ is defined by
$$MO_r(f)(z)=\left\{\frac{1}{|D(z,r)|}\int_{D(z,r)}|f(w)-\widehat{f}_r(z)|^2\,dV(w)\right\}^{1/2},$$
where $r$ is a positive number and
$$\widehat{f}_r(z)=\frac{1}{|D(z,r)|}\int_{D(z,r)}f(w)\,dV(w).$$
We say that $f$ in $\textup{BMO}_r$ if $MO_r(f)$ is bounded; $f$ in $\textup{VMO}_r$ if $MO_r(f)\to 0$ as $z\to \infty$. It was proved in \cite{WCZ} that

 (A)$H_f$ and $H_{\overline{f}}$ are both bounded$\Leftrightarrow f\in\text{BMO}\Leftrightarrow f \in \textup{BMO}_r$ for all $0<r<\infty$;

 (B)$H_f$ and $H_{\overline{f}}$ are both compact$\Leftrightarrow f\in\text{VMO}\Leftrightarrow f \in \textup{VMO}_r$ for all $0<r<\infty$.

We will show that $f$ in $\text{BMO}$ if and only if $f,\overline{f}\in\textup{BDA}_{\Psi,r}$(see Lemma \ref{5p5}). Therefore, our results cover those in \cite{SY,TW,WCZ}. It should be pointed out that Bommier-Hato and Constantin \cite{BC} considered Hankel operators with anti-holomorphic symbols on vector-valued Fock spaces. %However, the theory in \cite{BC} does not cover our result here because the symbol function in \cite{BC} is anti-holomorphic.

We also study the Schatten class membership of Toeplitz operators and Hankel operators. Given two separable Hilbert spaces $H_1$, $H_2$ and $0<p<\infty$, we say that a compact linear operator $A:H_1\to H_2$ belongs to the Schatten class $S_p:=S_p(H_1,H_2)$ if the sequence of eigenvalues $\{s_k\}_{k=1}^\infty$ of $|A|:=(A^*A)^{\frac{1}{2}}$ satisfies $\|A\|_{S_p}^p:=\sum_{k=1}^{\infty}s_k^p<\infty.$
This defines a norm when $1\leq p<\infty$ and a quasi-norm otherwise.
Given a positive Borel measure $\mu$ on $\mathbb{C}^n$, we define the Toeplitz operator $T_\mu$ with the symbol $\mu$ as
$$T_\mu g(z):=\int_{\mathbb{C}^n}K(z,w)g(w)e^{-\Psi(|w|^2)}\,d\mu(w)\textup{ for }g\in F^2_\Psi,z\in\mathbb{C}^n.$$
In particular, when $d\mu(z)=f(z)\,dV(z)$ and $f$ is a positive measurable function, the induced Toeplitz operator is denoted by $T_f.$
%For a positive $T$ operator on $F^2_\Psi$, the Berezin transform of $T$ is defined by
%$$\widetilde{T}(z)=\langle Tk_z,k_z\rangle,\textup{ for } z\in\mathbb{C}^n.$$
For a positive Borel measure $\mu$ on $\mathbb{C}^n$, we define the Berezin transform of $\mu$ as
$$\widetilde{\mu}(z):=\int_{\mathbb{C}^n}|k_z(w)|e^{-\Psi(|w|^2)}\,d\mu(w)\textup{ for } z\in\mathbb{C}^n.$$
For all $0<r<\infty$ and $z\in \mathbb{C}^n$, the integral mean of $\mu$ over $D(z,r)$ is given by $\widehat{\mu}_r(z):=\mu(D(z,r))/|D(z,r)|.$

Set $d\nu_\Psi(z)=K(z,z)\,d\mu_\Psi(z).$ The following theorem gives characterizations of the membership in the Schatten classes of Toeplitz operators with positive measure symbols.
	\begin{theorem}\label{4p5}
		Let $\mu$ be a positive Borel measure on $\mathbb{C}^n$ and $0< p <\infty$. The following conditions are equivalent.
		\textup{(A)} The Toeplitz operator $T_\mu$ is in $S_p$.\\
		\textup{(B)} The function $\widetilde{\mu}(z)$ is in $L^p(\mathbb{C}^n,d\nu_\Psi)$.\\
		\textup{(C)} The function $\widehat{\mu}_r(z)$ is in $L^p(\mathbb{C}^n,d\nu_\Psi)$ for any $0<r<\infty$.\\
		\textup{(D)} The sequence $\{\widehat{\mu}_r(z_k)\}_{k=1}^\infty$ is in $l^p$ for any $\Psi$-lattice $\left\{z_{k}\right\}_{k=1}^\infty$.	Furthermore,
\begin{align}\label{t4}
  \|T_\mu\|_{S_p}\simeq \|\widetilde{\mu}\|_{L^p(\mathbb{C}^n,d\nu_\Psi)} \simeq \|\widehat{\mu}_r\|_{L^p(\mathbb{C}^n,d\nu_\Psi)}\simeq \|\{\widehat{\mu}_r(z_k)\}_{k=1}^\infty\|_{l^p}.
\end{align}
	\end{theorem}
The equivalence between (A) and (D) was proved in \cite{SY} for $1\leq p<\infty$. Our contribution is on the small exponent $0<p<1$. The characterizations of Schatten class Hankel operators are the following.
		\begin{theorem}\label{1p3}
		Let $0<p<\infty$. Suppose $f\in \mathcal{S}$. Then the following statements are equivalent.\\
		\textup{(A)} The Hankel operator $H_f$ is in $S_{p}$.\\
		\textup{(B)} $G_r(f)\in L^p(\mathbb{C}^n,d\nu_\Psi)$ for some $0<r<\infty.$\\
        \textup{(C)} $\int_{\mathbb{C}^n}\|H_fk_z\|_\Psi^p\,d\nu_\Psi<\infty$. Furthermore,
        \begin{align}\label{1e5}
          \|H_f\|_{S_p}\simeq \|G_r(f)\|_{L^p(\mathbb{C}^n,d\nu_\Psi)}\simeq
          \left\{\int_{\mathbb{C}^n}\|H_fk_z\|_\Psi^p\,d\nu_\Psi(z)\right\}^{\frac{1}{p}}.
        \end{align}
	\end{theorem}
As a consequence of the last theorem, the following corollary extends the result in \cite{SY}. Its proof is similar to Lemma 4.1 of \cite{ZWH}.
		\begin{corollary}\label{1p4}
		Let $0< p<\infty$. Suppose $f\in \mathcal{S}$. Then the Hankel operators $H_f$ and $H_{\overline{f}}$ are both in $S_{p}$ if and only if $MO_r(f)\in L^p(\mathbb{C}^n,d\nu_\Psi)$ for some $0<r<\infty.$ Furthermore,
        \begin{align}\label{1e6}
          \|H_f\|_{S_p}+\|H_{\overline{f}}\|_{S_p}\simeq \|MO_r(f)\|_{L^p(\mathbb{C}^n,d\nu_\Psi)}.
        \end{align}
		\end{corollary}
The key ingredients in the proof of our main results are $\bar{\partial}$-techniques, Carleson measure and the decomposition theory of functions in $\textup{BDA}_{\Psi,r}$. As the canonical solution to $\bar{\partial}v=g\bar{\partial}f$, $H_fg$ is naturally connected with the $\bar{\partial}$-theory. The decomposition theory allows us to decompose $f\in\textup{BDA}_{\Psi,r}$ into two parts. One part is associated with  Carleson measures while we deal with the other part by the technique of
H\"{o}mander's $L^2$ estimates for the $\overline{\partial}$ operator. We mention that \cite{SY} also applies H\"{o}mander's $L^2$ theory when $f$ is holomorphic. However, the estimate in \cite{SY} cannot be straightforward extended to non-holomorphic functions. New estimates and methods must be developed(see Lemma \ref{3p8}).

There are two main difficulties we meet while studying Hankel operators on the Fock-type space, one for which is the non-isotropic behavior of the Bergman metric induced by $F^2_\Psi$ and the other is the lack of decay estimate on  $K(z, w)$ with the Bergman distance $\varrho(z, w)$. In these references \cite{B,HV2,I,IVW,XZ}, the Bergman balls are equivalent to Euclidean balls which makes it easier to do analysis along the complex tangential directions. On the classical Fock space $F^2$, we have an identity that
$$
   |K(z, w)|e^{-|z|^2}e^{-|w|^2}= e^{\frac 1 2 |z-w|^2} \ \ \textrm{for }\ z, w \in {\mathbb C}^n.
$$
On $F^2_\phi$ with $\phi\in C^2(\mathbb C^n)$ satisfying $i\partial \overline {\partial } \phi \simeq  i\partial \overline {\partial } |z|^2$, we have
$$
   |K(z, w)|e^{-\phi(z)}e^{-\phi(w)}\le C  e^{-\theta |z-w|} \ \ \textrm{for }\ z, w \in {\mathbb C}^n,
$$
where $C$  and $\theta$ are two positive constants. Similar estimate also holds on those Fock spaces studied in \cite{AT,CO,HL}. This estimate is crucial to the study on Fock spaces, but unfortunately there is no analogy we know on $F^2_\Psi$. This brings us to the new approach of the present work. More precisely, we use integral estimates for the reproducing kernel established in Section \ref{s4}.

The paper is organized as follows. In Section \ref{s2}, we give some preliminaries and background information. We recall some notions and key lemmas with some estimates of the reproducing kernel as well as some useful inequalities. In Section \ref{s3}, we characterize the boundedness and compactness of Hankel operators. Section \ref{s4} is devoted to give some integral estimates involving the reproducing kernel. In Section \ref{s5}, we prove our main result on the Schatten class membership of Toeplitz operators with positive measure symbols. In Section \ref{s6}, we
characterize the membership in the Schatten classes of Hankel operators.
As an application, we obtain a characterization of the simultaneous membership in $S_p$ of Hankel operators $H_{f}$ and $H_{\overline{f}}$ acting on $F^2_\Psi$.

In what follows, we will use the notation $a\lesssim b$ to indicate that there is a constant $C>0$ with $a\leq Cb$; and the expression $a\simeq b$ means that $a\lesssim b$ and $b\lesssim a$. $C$ will denote a positive constant whose value may change from one occasion to another but does not depend on the functions or operators in consideration.

\section{Preliminaries}\label{s2}
    Set $\Lambda_\Psi(z):=\log K(z,z)$ and
    $$\beta^2(z,\xi):=\sum_{j,k=1}^{n}\frac{\partial^2\Lambda_\Psi(z)}{\partial z_j\partial \overline{z}_k}\xi_j\overline{\xi_k},$$
    for arbitrary vectors $z=(z_1,\cdots,z_n)$ and $\xi=(\xi_1,\cdots,\xi_n)$ in $\mathbb{C}^n$. The associated function $\varrho(z,w)$, the Bergman distance, is given by
    $$\varrho(z,w):=\inf_{\gamma}\int_{0}^{1}\beta(\gamma(t),\gamma'(t))\,dt,$$
    where the infimum is taken over all piecewise $C^1$-smooth curves $\gamma:[0,1]\to \mathbb{C}^n$ such that $\gamma(0)=z$ and $\gamma(1)=w$. For any $z\in \mathbb{C}^n$ and $r>0$, denote by
    $B(z,r):=\{w\in\mathbb{C}^n:\varrho(z,w)<r\}$
    the Bergman ball centered at $z$ with radius $r$. Let $|B(z,r)|$ be the Euclidean volume of $B(z,r)$.

    Given some $r>0$, we say that a sequence of distinct points $\{z_k\}_{k=1}^\infty$ is a $\Psi$-lattice if the balls $\{B(z_k,r)\}_{k=1}^\infty$ cover $\mathbb{C}^n$ and $\{B(z_k,r/2)\}_{k=1}^\infty$ are pair-wisely disjoint. We refer to $r$ as the covering radius. From \cite[Lemma 7.4]{SY}, we know that given $r>0$ there exists a $\Psi$-lattice. Moreover, given a $\Psi$-lattice $\{z_k\}_{k=1}^\infty$ there exists an integer $N$ such that each $z\in \mathbb{C}^n$ belongs to at most $N$ balls of $B(z_k,2r)$. More specifically,
\begin{align}\label{2e1}
  \sum_{k=1}^{\infty}\chi_{B(z_k,2r)}(z)\leq N\textup{ for } z\in \mathbb{C}^n.
\end{align}
	In our analysis, it is convenient to use the ball $D(z,r)$. From \cite[Lemma 7.2]{SY}, we know that for every positive number $r$, there exist two positive numbers $r_1$ and $r_2$ such that for every $z\in \mathbb{C}^n$,
\begin{align}\label{2e2}
  D(z,r_1)\subseteq B(z,r)\subseteq D(z,r_2).
\end{align}
It follows that the Euclidean volume of $B(z,r)$ can be estimated as
\begin{align}\label{2e3}
  |B(z,r)|\simeq [\Phi'(|z|^2)]^{-1}[\Psi'(|z|^2)]^{-(n-1)}.
\end{align}
Moreover, given $r>0$ there exists a sequence $\{z_k\}_{k=1}^\infty$ and a positive constant $m$ such that $\cup_k D(z_k,r/2)=\mathbb{C}^n$ and $D(a_j,r/m)\cap D(a_k,r/m)=\emptyset$ for $j\neq k$.

    The following lemma gives an estimate on $K(z,w)$.
    \begin{lemma}\label{2p1}
      There exists a positive number $r_0$ such that
      \begin{align}\label{2e4}
        |K(z,w)|^2\simeq K(z,z)K(w,w),
      \end{align}
      for every $z\in \mathbb{C}^n$ and $w\in D(z,r_0).$
    \end{lemma}
    \begin{proof}
      By the relation \eqref{2e2} and Lemma 8.1 of \cite{SY}, we obtain the desired result.
    \end{proof}
    We also need the following pointwise estimate on the reproducing kernel which was obtained in \cite{SY}.
    \begin{lemma}\label{2p2}
Let $z$ and $w$ be any points in $\mathbb{C}^n$ such that $\langle z,w \rangle \neq 0$, and write $\langle z,w \rangle =re^{i\theta}$, where $0<r<\infty$ and $-\pi <\theta\leq \pi$. Then
\begin{equation}\label{2e5}
  \frac{1}{[\Psi '(r)]^{n-1}}\frac{|K(z,w)|}{e^{\Psi(r)}}\lesssim\left\{
    \begin{aligned}
      \Phi '(r),\qquad\qquad\qquad\quad &|\theta|\leq \theta_0(r),\\
      r^{-3/2}[\Phi '(r)]^{-1/2}|\theta|^{-3},\quad&|\theta|> \theta_0(r),
    \end{aligned}
  \right.
\end{equation}
where $\theta_0(r)=[r\Phi '(r)]^{-1/2}$. Moreover, there exists a positive constant $c$ such that if $\theta<c\theta_0(r)$, then
\begin{align}\label{2e6}
  |K(z,w)|\gtrsim \Phi '(r)[\Psi '(r)]^{n-1}e^{\Psi(r)}.
\end{align}
In particular,
\begin{align}\label{ex2e6}
  \|K_z\|_{\Psi}^2\simeq \Phi '(|z|^2)[\Psi '(|z|^2)]^{n-1}e^{\Psi(|z|^2)}.
\end{align}
\end{lemma}
The following weighted Bergman inequality is useful in our analysis. We can find its proof in \cite{SY} for $p=2$. We observe that it holds for all $0<p<\infty$.
\begin{lemma}\label{2p3}
  Let $0<p<\infty$ and $r_0$ be the constant from Lemma \ref{2p1}. Suppose $0<r\leq r_0$. Then there is a constant $C$ such that
  \begin{align}\label{2e7}
    \left|f(z)e^{-\frac{1}{2}\Psi(|z|^2)}\right|^p\leq \frac{C}{|B(z,r)|}\int_{B(z,r)}\left|f(w)e^{-\frac{1}{2}\Psi(|w|^2)}\right|^p\,dV(w),
  \end{align}
for all $f\in H(\mathbb{C}^n)$ and $z\in \mathbb{C}^n$.
\end{lemma}
\begin{lemma}\label{2p4}
Let $0<p,r<\infty$ and $\mu$ be a positive Borel measure on $\mathbb{C}^n$. Then there exists a positive constant $C$ such that
\begin{align}\label{2e8}
  \int_{\mathbb{C}^n}\left|f(z)e^{-\frac{1}{2}\Psi(|z|^2)}\right|^pd\mu(z)\leq C
\int_{\mathbb{C}^n}\left|f(z)e^{-\frac{1}{2}\Psi(|z|^2)}\right|^p\widehat{\mu}_r(z)dV(z),
\end{align}
 for all $f\in H(\mathbb{C}^n)$ and $z\in \mathbb{C}^n$.
\end{lemma}
\begin{proof}
  See \cite[Lemma 15]{WTH}.
\end{proof}

\section{Boundedness and Compactness of Hankel Operators}\label{s3}
To establish the decomposition theory of the space $\textup{BDA}_{\Psi,r}$, we need to construct a partition of unity subordinated to a certain lattice. We begin with the following lemma.
\begin{lemma}\label{3p1}
Let $0<r\leq 1/8$. For every $z\in \mathbb{C}^n$ and $w\in D(z,r)$, we have
$$\Psi '(|z|^2)\simeq \Psi '(|w|^2)\textup{ and } \Phi '(|z|^2)\simeq \Phi '(|w|^2).$$
\end{lemma}
\begin{proof}
  It follows immediately from the definition of $\Phi$ that $\Psi '(|z|^2)\leq \Phi '(|z|^2)$. Hence, for $w\in D(z,r)$, we have
  $$|z-w|\leq |z-P_zw|+|w-P_zw|\leq 2r[\Psi '(|z|^2)]^{-1/2}.$$
  On the other hand,
  \begin{align*}
    \left||z|^2-|w|^2\right|\leq |z-w|(|z|+|w|)\leq |z-w|(|z|+|z-w|+|z|)
    =|z-w|^2+2|z|\cdot|z-w|.
  \end{align*}
  It follows that for $|z|\geq [\Psi(0)]^{1/2}$ and $r\leq 1/8$,
  $$\left||z|^2-|w|^2\right|\leq \left(4r^2|z|^{-1}[\Psi '(|z|^2)]^{-1/2}+4r \right) |z|[\Psi '(|z|^2)]^{-1/2}\leq |z|[\Psi '(|z|^2)]^{-1/2}.$$
  By applying \cite[Lemma 7.1]{SY} with $\alpha=1/2$, we know that
  $$\Psi '(|w|^2)=(1+o(1))\Psi '(|z|^2)\textup{ as }|z|\to \infty.$$
  Therefore, we have for every $z\in \mathbb{C}^n$ and $w\in D(z,r)$
  \begin{align}\label{3e1}
    \Psi '(|w|^2)\simeq \Psi '(|z|^2).
  \end{align}
  From \eqref{2e2}, we know that there is a positive constant $r_1$ such that $D(z,r)\subseteq B(z,r_1).$
  If $w \in D(z,r)$, then $B(w,r_1)\subseteq B(z,2r_1)$ and $B(z,r_1)\subseteq B(w,2r_1)$.
  This, together with \eqref{2e3}, implies that
  $$[\Phi '(|w|^2)]^{-1}[\Psi '(|w|^2)]^{1-n}\lesssim [\Phi '(|z|^2)]^{-1}[\Psi '(|z|^2)]^{1-n},$$
  and
  $$[\Phi '(|z|^2)]^{-1}[\Psi '(|z|^2)]^{1-n}\lesssim [\Phi '(|w|^2)]^{-1}[\Psi '(|w|^2)]^{1-n}.$$
  Combining the last two inequalities with \eqref{3e1}, we obtain $\Phi '(|w|^2)\simeq \Phi '(|z|^2)$
%  \begin{align}\label{3e2}
%    \Phi '(|w|^2)\simeq \Phi '(|z|^2),
%  \end{align}
   for every $z\in \mathbb{C}^n$ and $w\in D(z,r)$. The proof is completed.
\end{proof}
Recall that for $v\in C^2(\mathbb{C}^n)$, the $\overline{\partial}$ operator is defined as
$$\overline{\partial}v=\sum_{j=1}^{n}\frac{\partial v}{\partial \overline{z}_j}d\overline{z}_j,$$
and the $\overline{\partial}$ operator acting on a (0,1) form $u=\sum_{j=1}^{n}u_jd\overline{z_j}$ is defined by the formula
$$\overline{\partial}u=\sum_{1\leq j<k\leq n}\left(\frac{\partial u_j}{\partial \overline{z}_k}-\frac{\partial u_k}{\partial \overline{z}_j}\right)d\overline{z}_j\land d\overline{z}_k.$$
Set
$$|u|^2=\sum_{j=1}^{n}|u_j|^2\textup{ and } |\overline{\partial}u|^2=\sum_{1\leq j<k\leq n}\left|\frac{\partial u_k}{\partial \overline{z}_j}-\frac{\partial u_j}{\partial \overline{z}_k}\right|^2.$$
It is easy to see that for $v\in C^2(\mathbb{C}^n)$, $\overline{\partial}v$ is $\overline{\partial}$-closed, i.e., $\overline{\partial}^2v=0.$
	\begin{lemma}\label{3p2}
		Let $r>0$ and $z\in \mathbb{C}^n\backslash \{0\}$. Then there exists a real-valued function $\gamma_z \in C^\infty(\mathbb{C}^n)$, satisfying the following properties.\\
\textup{(A)}\,$0\leq \gamma_z \leq 1$, $\gamma_z|_{D(z,r/2)}\equiv 1$, $\textup{supp}\, \gamma_z \subseteq D(z,r)$.\\
\textup{(B)}\,$|\nabla \gamma_z(w)|\leq C_1[\Phi'(|z|^2)]^{\frac{1}{2}}/r$.\\
\textup{(C)}\,$\left|\bar{\partial}\gamma_z(w)\land \overline{\partial}|w| \right|\leq C_2[\Psi'(|z|^2)]^{\frac{1}{2}}/r$, where $C_1$ and $C_2$ are absolute constants.
	\end{lemma}
	\begin{proof}
	Choose a function $T\in C^\infty([0,+\infty))$ satisfying $T|_{[0,\frac{1}{2}]}\equiv 1$, $\textup{supp}\,T\subseteq [0,1]$ and $-3\leq T'(x)\leq 0$. Given $z\in\mathbb{C}^n\backslash \{0\}$, we define the function $\gamma_z$ as
$$\gamma_z(w):=T\left(\frac{|P_zw-z|}{r[\Phi '(|z|^2)]^{-\frac{1}{2}}}\right)\cdot T\left(\frac{|w-P_zw|}{r[\Psi '(|z|^2)]^{-\frac{1}{2}}}\right).$$
It is easy to see that $0\leq \gamma_z(w)\leq 1$, $\gamma_z\in C^\infty(\mathbb{C}^n)$, $\gamma_z|_{D(z,r/2)}\equiv 1$ and $\textup{supp}\, \gamma_z \subseteq D(z,r)$. Thus we have proved (A). To prove (B) and (C), we begin with considering the case $z=(z_1,0,\cdots,0)$. The function $\gamma_z$ is reduced to
$$\gamma_z(w)=T\left(\frac{|w_1-z_1|}{r[\Phi '(|z|^2)]^{-\frac{1}{2}}}\right)\cdot T\left(\frac{|(0,w_2,\cdots,w_n)|}{r[\Psi '(|z|^2)]^{-\frac{1}{2}}}\right).$$
It follows that
\begin{align*}
  \left|\frac{\partial \gamma_z}{\partial w_1}(w)\right|&=\left|T'\left(\frac{|w_1-z_1|}{r[\Phi ' (|z|^2)]^{-\frac{1}{2}}}\right)\cdot \frac{1}{2r[\Phi ' (|z|^2)]^{-\frac{1}{2}}}\right|\cdot T\left(\frac{|(0,w_2,\cdots,w_n)|}{r[\Psi '(|z|^2)]^{-\frac{1}{2}}}\right)
  \leq \frac{3\Phi' (|z|^2)^{\frac{1}{2}}}{2r},
\end{align*}
and
\begin{align*}
  \left|\frac{\partial \gamma_z}{\partial w_j}(w)\right|&=T\left(\frac{|w_1-z_1|}{r[\Phi '(|z|^2)]^{-\frac{1}{2}}}\right)\cdot \left|T'\left(\frac{|(0,w_2,\cdots,w_n)|}{r[\Psi ' (|z|^2)]^{-\frac{1}{2}}}\right)\cdot \frac{\overline{w}_j}{2r|(0,w_2,\cdots,w_n)|[\Psi ' (|z|^2)]^{-\frac{1}{2}}}\right|\\
  &\leq \frac{3\Psi ' (|z|^2)^{\frac{1}{2}}}{2r}\leq \frac{3\Phi ' (|z|^2)^{\frac{1}{2}}}{2r},
\end{align*}
for $j\neq 1$. Similarly,
\begin{align}\label{3e3}
  \left|\frac{\partial \gamma_z}{\partial \overline{w}_1}(w)\right|\leq \frac{3\Phi' (|z|^2)^{\frac{1}{2}}}{2r}, \textup{ and } \left|\frac{\partial \gamma_z}{\partial \overline{w}_j}(w)\right|\leq \frac{3\Psi ' (|z|^2)^{\frac{1}{2}}}{2r}\leq \frac{3\Phi ' (|z|^2)^{\frac{1}{2}}}{2r}\textup{ for }j\neq 1.
\end{align}
Thus, (B) holds with $z=(z_1,0,\cdots,0)$. Notice that
$\overline{\partial}|w|=\frac{1}{2|w|}\sum_{j=1}^{n}w_jd\overline{w}_j.$
We have
\begin{align*}
  \overline{\partial}\gamma_z(w)\land \overline{\partial}|w|
  %&=\left(\sum_{j=1}^{n}\frac{\partial \gamma_z}{\partial \overline{w}_j}(w)d\overline{w}_j\right)\land \left(\frac{1}{2|w|}\sum_{j=1}^{n}w_jd\overline{w}_j\right)\\
  =\sum_{1\leq j<k\leq n}\frac{1}{2|w|}\left(\frac{\partial\gamma_z}{\partial\overline{w}_j}(w)w_k-
  \frac{\partial\gamma_z}{\partial\overline{w}_k}(w)w_j\right)d\overline{w}_j\land d\overline{w}_k.
\end{align*}
Therefore,
\begin{align}\label{3e4}
  \left|\overline{\partial}\gamma_z(w)\land \overline{\partial}|w|\right|^2=\sum_{1\leq j<k\leq n}\frac{1}{4|w|^2}\left|\frac{\partial\gamma_z}{\partial\overline{w}_j}(w)w_k-
  \frac{\partial\gamma_z}{\partial\overline{w}_k}(w)w_j\right|^2.
\end{align}
%On the other hand
%\begin{align*}
%  \frac{\partial \gamma_z}{\partial \overline{w}_1}(w)&=T'\left(\frac{|w_1-z_1|}{r[\Phi ' (|z|^2)]^{-\frac{1}{2}}}\right)\cdot \frac{1}{r[\Phi ' (|z|^2)]^{-\frac{1}{2}}}\cdot \frac{w_1-z_1}{2|w_1-z_1|}\\
%  &\quad \times T\left(\frac{|(0,w_2,\cdots,w_n)|}{r[\Psi '(|w|^2)]^{-\frac{1}{2}}}\right)
%\end{align*}
%and
%\begin{align*}
%  \frac{\partial \gamma_z}{\partial \overline{w}_j}(w)&=T'\left(\frac{|(0,w_2,\cdots,w_n)|}{r[\Psi ' (|z|^2)]^{-\frac{1}{2}}}\right)\cdot \frac{1}{r[\Psi ' (|z|^2)]^{-\frac{1}{2}}}\cdot \frac{w_j}{2|(0,w_2,\cdots,w_n)|}\\
%  &\quad \times T\left(\frac{|w_1-z_1|}{r[\Phi '(|z|^2)]^{-\frac{1}{2}}}\right)
%\end{align*}
%for $j\neq 1$. It follows that
%$$\left|\frac{\partial \gamma_z}{\partial \overline{w}_1}(w)\right|\leq \frac{3[\Phi '(|z|^2)]^{\frac{1}{2}}}{2r},\textup{\quad and \quad}\left|\frac{\partial \gamma_z}{\partial \overline{w}_j}(w)\right|\leq \frac{3[\Psi '(|z|^2)]^{\frac{1}{2}}}{2r}, (j\neq 1).$$
By the fact that $\Phi '(|z|^2)=|z|^2\Psi ''(|z|^2)+\Psi '(|z|^2)$ and inequalities \eqref{cdt1} and \eqref{cdt3}, we deduce that
\begin{align}\label{3e5}
\Phi '(|z|^2)\lesssim |z|^2\Psi '(|z|^2)^{1+\eta}+\Psi '(|z|^2)\lesssim |z|^2\Psi ' (|z|^2)^2.
\end{align}
Obeserve that, if $w\in D((z_1,0,\cdots,0),r)$ and $j\neq 1$, then
$$|w|\geq |w_1|\simeq |z|\textup{ and }|w_j|\lesssim \frac{r}{[\Psi '(|z|^2)]^{\frac{1}{2}}},$$
which, together with \eqref{3e3} and \eqref{3e5}, implies that, for $1<k\leq n$ and $w\in D((z_1,0,\cdots,0),r)$,
\begin{align}\label{3e6}
\frac{1}{|w|}\left|\frac{\partial\gamma_z}{\partial\overline{w}_1}(w)w_k-
  \frac{\partial\gamma_z}{\partial\overline{w}_k}(w)w_1\right|\lesssim \frac{[\Phi '(|z|^2)]^{\frac{1}{2}}}{|z|[\Psi '(|z|^2)]^{\frac{1}{2}}}+
  \frac{[\Psi '(|z|^2)]^{\frac{1}{2}}}{r}\lesssim  \frac{[\Psi '(|z|^2)]^{\frac{1}{2}}}{r}.
\end{align}
For $2\leq j<k\leq n$ and $w\in D((z_1,0,\cdots,0),r)$, by \eqref{3e3}, we have
\begin{align}\label{3e7}
\frac{1}{|w|}\left|\frac{\partial\gamma_z}{\partial\overline{w}_j}(w)w_k-
  \frac{\partial\gamma_z}{\partial\overline{w}_k}(w)w_j\right|\leq \left| \frac{\partial\gamma_z}{\partial\overline{w}_j}(w)\right|+
  \left| \frac{\partial\gamma_z}{\partial\overline{w}_k}(w)\right| \lesssim  \frac{[\Psi '(|z|^2)]^{\frac{1}{2}}}{r}.
\end{align}
Notice that $\gamma_z|_{\mathbb{C}^n\backslash D(z,r)}\equiv 0$. This, together with \eqref{3e4}, \eqref{3e6} and \eqref{3e7}, implies that (C) holds with $z=(z_1,0,\cdots,0)$. We will apply the complex tangent space and unitary transformation to reach the general case. To this end, let $w\in \mathbb{C}^n\backslash \{0\}$ and denote by $T_\mathbb{C}(w)$ the complex tangent space of $\mathbb{C}^n$ at the point $w$, namely,
$$T_\mathbb{C}(w)=\left\{\xi\in\mathbb{C}^n:\langle\xi,w\rangle=0\right\}.$$
Suppose $z=(|z|,0,\cdots,0)\in \mathbb{C}^n\backslash\{0\}$. Notice that $|w|\simeq|w_1|\simeq |z|\neq 0$ when $w=(w_1,w_2,\cdots,w_n)\in D(z,r).$ It is easy to see that $\{(\overline{w}_k,0,\cdots,-\overline{w}_1,0,\cdots,0):2\leq k\leq n \}$(the first component is $\overline{w}_k$ and the $k$-th component is $-\overline{w}_1$) are $n-1$ linearly independent vectors which form a basis of $T_{\mathbb{C}}(w).$ Thus, for each $\xi \in T_{\mathbb{C}}(w)$ with $|\xi|\leq 1$, there are $\alpha_2,\cdots,\alpha_n$ with $|\alpha_k|\lesssim 1$ for $2\leq k\leq n$ such that
$$\xi=\sum_{2\leq k\leq n}\frac{\alpha_k}{|w|}(\overline{w}_k,0,\cdots,-\overline{w}_1,0,\cdots,0).$$
Indeed, we rewrite $\xi$ as
$$\xi
=\left(\sum_{2\leq k\leq n}\frac{\alpha_k\overline{w}_k}{|w|},-\frac{\alpha_2\overline{w}_1}{|w|},\cdots,-\frac{\alpha_n\overline{w}_1}{|w|}\right).$$
Thus,
$|\alpha_k|\cdot |\overline{w}_1|/|w|\leq |\xi|\leq 1$
for $2\leq k\leq n$. Since $|w_1|\simeq |w|$, we have $\alpha_k\lesssim 1$ for $2\leq k\leq n$. Denote by $\nabla_w \gamma_z$ the complex gradient of $\gamma_z$, that is,
$$\nabla_w\gamma_z=\left(\frac{\partial \gamma_z}{\partial w_1},\cdots,\frac{\partial \gamma_z}{\partial w_n}\right).$$
Noting that $\gamma_z$ is real-valued, we also have
$$\overline{\nabla_w\gamma_z}=\left(\frac{\partial \gamma_z}{\partial \overline{w}_1},\cdots,\frac{\partial \gamma_z}{\partial \overline{w}_n}\right).$$
Combining \eqref{3e6} with the fact
$$\langle \xi,\overline{\nabla_w\gamma_z}\rangle=\sum_{2\leq k\leq n}\frac{\alpha_k}{|w|}\langle (\overline{w}_k,0,\cdots,-\overline{w}_1,0,\cdots,0), \overline{\nabla_w\gamma_z} \rangle=\sum_{2\leq k\leq n}
\frac{\alpha_k}{|w|}\left(\frac{\partial\gamma_z}{\partial w_1}(w)\overline{w}_k-
  \frac{\partial\gamma_z}{\partial w_k}(w)\overline{w}_1\right),$$
we see that, for $z=(|z|,0\cdots,0)$,
\begin{align}\label{3e8}
  \sup_{\xi\in T_{\mathbb{C}}(w),|\xi|\leq 1}\left| \langle \xi,\overline{\nabla_w\gamma_z}\rangle \right|\leq \frac{C[\Psi '(|z|^2)]^{1/2}}{r}.
\end{align}
For $z\in \mathbb{C}^n\backslash \{0\}$, let $U_z$ be the unitary transformation on $\mathbb{C}^n$ such that $U_z(z)=(|z|,0,\cdots,0).$ Denote by $w'=U_z(w)$ for $w \in \mathbb{C}^n.$ It is easy to see that $D(z,r)=U_z(D(z',r)),T_\mathbb{C}(w)=U_z(T_\mathbb{C}(w')).$ For $1\leq j <k\leq n$, set
$$\xi_{jk}:=\frac{1}{|w|}(0,\cdots,0,\overline{w}_k,0,\cdots,0,-\overline{w}_j,0\cdots,0)\in T_{\mathbb{C}}(w),$$
where the $j$-th component is $\overline{w}_k/|w|$ and the $k$-th component is $-\overline{w}_j/|w|$. Clearly $|\xi_{jk}|\leq 1$. Moreover, there holds
$$\frac{1}{|w|}\left( w_k\frac{\partial \gamma_z}{\partial \overline{w}_j}(w)-w_j\frac{\partial \gamma_z}{\partial \overline{w}_k}(w)\right)=\langle \overline{\nabla\gamma_z(w)},\xi_{j,k} \rangle, $$
which, together with \eqref{3e8}, further implies that
\begin{align*}
  \frac{1}{|w|}\left| w_k\frac{\partial \gamma_z}{\partial \overline{w}_j}(w)-w_j\frac{\partial \gamma_z}{\partial \overline{w}_k}(w)\right|&\leq \sup_{\xi\in T_{\mathbb{C}}(w),|\xi|\leq 1}\left| \langle \xi,\overline{\nabla_{w}\gamma_z}\rangle \right|
  =\sup_{\xi\in T_{\mathbb{C}}(w'),|\xi|\leq 1}\left| \langle \xi,\overline{\nabla_{w}\gamma_{z'}(w')}\rangle \right|\\
  &\lesssim \frac{[\Psi '(|z'|^2)]^{1/2}}{r}=\frac{[\Psi '(|z|^2)]^{1/2}}{r}.
\end{align*}
Combining this with \eqref{3e4}, we conlcude that (C) holds for all $z\in \mathbb{C}^n\backslash \{0\}.$ The proof is completed.
	\end{proof}
Let $\mathcal{A}$ be a subset of integer numbers. Denote by $\textup{card\,} \mathcal{A}$ the cardinal number of $\mathcal{A}$.
\begin{lemma}\label{3p3}
Let $r$, $m$, $a$ and $b$ be positive constants. Suppose $\{z_k\}_{k=1}^\infty$ is a sequence in $\mathbb{C}^n$ such that $D(z_j,r/m)\cap D(z_k,r/m)=\emptyset$ for $j\neq k$. Then there exists a positive integer $N$ such that for every $z\in \mathbb{C}^n$,
\begin{align}\label{3e9}
  \textup{card}\{k:D(z,ar)\cap D(z_k,br)\neq \emptyset\}\leq N.
\end{align}
\end{lemma}
\begin{proof}
  By the relation \eqref{2e2} and the triangle inequality, we konw that if $D(z,ar)\cap D(z_k,br)\neq \emptyset$, then there is a constant $r_1$(independent of $z$ as well as $\{z_k\}_{k=1}^\infty$) such that $D(z_k,r/m)\subseteq B(z,r_1).$ Since $\{D(z_k,r/m)\}_{k=1}^\infty$ are pair-wisely disjoint, there holds
  \begin{align}\label{3e10}
    \sum_{\{k:D(z,ar)\cap D(z_k,br)\neq \emptyset\}}|D(z_k,r/m)|\leq |B(z,r_1)|\simeq |D(z,r)|.
  \end{align}
  On the other hand, By the relation \eqref{2e2} and the triangle inequality again, we see that if $D(z,ar)\cap D(z_k,br)\neq \emptyset$, then
  $|D(z,r)|\simeq |D(z_k,r/m)|$, which implies that
  \begin{align}\label{3e11}
    \sum_{\{k:D(z,ar)\cap D(z_k,br)\neq \emptyset\}}|D(z_k,r/m)|\simeq
  \sum_{\{k:D(z,ar)\cap D(z_k,br)\neq \emptyset\}}|D(z,r)|.
  \end{align}
  From \eqref{3e10} and \eqref{3e11}, there exists an integer $N>0$ satisfying
%\begin{align*}
  $\textup{card}\{k:D(z,ar)\cap D(z_k,br)\neq \emptyset\}\leq N$
%\end{align*}
for $z\in \mathbb{C}^n$.
The proof is completed.
\end{proof}
We immediately obtain the following lemma.
\begin{lemma}\label{3p4}
Let $r$, $m$ and $b$ be positive constants. Suppose $\{z_k\}_{k=1}^\infty$ is a sequence in $\mathbb{C}^n$ such that $D(z_j,r/m)\cap D(z_k,r/m)=\emptyset$ for $j\neq k$. Then there exists a positive integer $N$ such that for every $z\in \mathbb{C}^n$,
\begin{align}\label{3e12}
  \sum_{k=1}^{\infty}\chi_{D(z_k,br)}(z)\leq N.
\end{align}
\end{lemma}
Given $0<r\leq 1/8$, let $\{z_k\}_{k=1}^\infty$ be a sequence in $\mathbb{C}^n$ such that $\cup_kD(z_k,r/2)=\mathbb{C}^n$ and $D(z_j,r/m)\cap D(z_k,r/m)=\emptyset$ for $j\neq k$ where $m>0$ is a constant. For each $z_k$, let $\gamma_{z_k}$ be the same function as in Lemma \ref{3p2}. Define
$\eta_k:=\gamma_{z_k}/(\sum_{j=1}^{\infty}\gamma_{z_j}).$
Then $\{\eta_k\}_{k=1}^\infty$ is a partition of unity subordinated to $\{D(z_k,r)\}$ satisfying
$0\leq \eta_k\leq 1, \sum_{k=1}^{\infty}\eta_k\equiv 1 \textup{ and } \text{supp }\eta_k\subseteq D(z_k,r).$
Applying Corollary \ref{3p4} with $b=1$, we know that the denominator of $\eta_k$ is always a sum of finite terms. By Lemma \ref{3p1}, we know that for $w\in D(z_j,r)$,
\begin{align}\label{3e13}
  \Psi '(|w|^2)\simeq \Psi '(|z_j|^2)\textup{ and } \Phi '(|w|^2)\simeq \Phi '(|z_j|^2).
\end{align}
This, together with Lemma \ref{3p2}, implies that for $w\in D(z_j,r)$,
\begin{align}\label{3e14}
  \left|\frac{\partial \gamma_{z_j}}{\partial w_i}(w)\right| \lesssim \Phi'(|w|^2)^{1/2}.
\end{align}
Notice that $0\leq \gamma_{z_j}\leq 1$ and $\sum \gamma_{z_j}\geq 1$. It follows from \eqref{3e14} and Lemma \ref{3p4} with $b=1$ that for $1\leq i \leq n$,
		\begin{align}\label{3e15}
		\left|\frac{\partial\eta_k}{\partial w_i}(w)\right|
&=\left|\frac{\frac{\partial \gamma_{z_k}}{\partial w_i}(w)}{\sum_{j=1}^{\infty}\gamma_{z_j}(w)}-\frac{\gamma_{z_k}(w)
\sum_{j=1}^{\infty}\frac{\partial \gamma_{z_j}}{\partial w_i}(w)}{\left[\sum_{j=1}^{\infty}\gamma_{z_j}(w)\right]^2}\right|
		\leq 2\sum_{\{j:w \in D(z_j,r)\}}\left|\frac{\partial \gamma_{z_j}}{\partial w_i}(w)\right|\notag\\
&\lesssim \Phi'(|w|^2)^{1/2}\sum_{j=1}^\infty \chi_{D(z_j,r)}(w)
\leq N\Phi'(|w|^2)^{1/2}.
		\end{align}
Similarly, for  $1\leq i \leq n$,
		\begin{align}\label{3e16}
		\left|\frac{\partial\eta_k}{\partial \overline{w}_i}(w)\right|\lesssim \Phi'(|w|^2)^{1/2}.
		\end{align}
This, together with \eqref{3e15} and \eqref{3e13}, further implies that
		%\begin{align*}
		$\left|\nabla \eta_k(w)\right|\lesssim \Phi'(|w|^2)^{1/2}\simeq \Phi'(|z_k|^2)^{1/2}$
		%\end{align*}
for $w\in D(z_k,r)$. Notice that $\eta_k|_{\mathbb{C}^n\backslash D(z_k,r)}\equiv0$. Hence,
		%\begin{align*}
		$\left|\nabla \eta_k(w)\right|\lesssim  \Phi'(|z_k|^2)^{1/2}$
		%\end{align*}
for every $w\in \mathbb{C}^n$. By the same method, we get
		%\begin{align*}
		$\left|\overline{\partial}\eta_k(w) \land \overline{\partial}|w|\right|\lesssim \Psi'(|z_k|^2)^{1/2}$
		%\end{align*}
for $w\in \mathbb{C}^n$. From the argument above, we immediately obtain the following lemma.
	\begin{lemma}\label{3p5}
		Given $0<r\leq 1/8$, there exist a sequence $\{z_k\}_{k=1}^\infty$ and a partition of unity $\{\eta_k\}_{k=1}^\infty$ with the following property.\\
\textup{(A) supp} $\eta_k\subseteq D(z_k,r),0\leq \eta_k \leq 1,$ $\eta_k\in C^\infty(\mathbb{C}^n),$
 	%\begin{align*}
		$\left|\nabla \eta_k(w)\right|\lesssim \Phi'(|z_k|^2)^{1/2}$ and $ \left|\overline{\partial}\eta_k(w) \land \overline{\partial}|w|\right|\lesssim \Psi'(|z_k|^2)^{1/2}$ for every $ w\in \mathbb{C}^n.$\\
	%\end{align*}
\textup{(B)} $\cup_kD(z_k,r/2)=\mathbb{C}^n,D(z_j,r/m)\cap D(z_k,r/m)=\emptyset (j\neq k)$ where $m>0$ is a constant.
	\end{lemma}
For $z \in \mathbb{C}^n$ and $f \in L^2(D(z,r),dV)$, we define the square mean of $|f|$ over $D(z,r)$ by setting
	$$M_r(f)(z)=\left(\frac{1}{|D(z,r)|}\int_{D(z,r)}|f(w)|^2dV(w)\right)^{1/2}.$$
The following lemma shows the relationship between $G_r(f)$ and $M_r(f)$.
	\begin{lemma}\label{3p6}
		For every $z \in \mathbb{C}^n$, $f \in L^2(D(z,r),dV)$ and $r>0$, there exists a function $h \in H(D(z,r))$ such that
		\begin{align}
			\label{3e17}
			M_r(f-h)(z)=G_r(f)(z).
		\end{align}
	\end{lemma}
\begin{proof}
  This lemma was essentially proved in Lemma 3.3 of \cite{HV}.
\end{proof}
For $z \in \mathbb{C}^n$ and $r>0$, let $A^2(D(z,r),dV):=L^2(D(z,r),dV)\cap H(D(z,r))$
	be the Bergman space over $D(z,r)$. Denote by $D_{z,r}$ the corresponding Bergman projection induced by the Bergman kernel of $A^2(D(z,r),dV)$. It is well-known that $D_{z,r}$ is bounded and $D_{z,r}h=h$ for $h \in A^2(D(z,r),dV)$.
	\begin{lemma}\label{ex3p6}
		For $z \in \mathbb{C}^n$, $f \in L^2(D(z,r),dV)$ and $r>0$, there holds
		%\begin{align*}
			$M_r(f-D_{z,r}(f))(z) \simeq G_r(f)(z).$
		%\end{align*}
	\end{lemma}
\begin{proof}
  See Lemma 3.2 of \cite{ZWH}.
\end{proof}
From \cite[p.191]{SY}, we know that there is a constant $t\geq 1$ such that for every $z\in \mathbb{C}^n$ and $w\in D(z,1)$,
    \begin{align}\label{3e18}
   t^{-1}\varrho(z,w)\leq |z-P_zw|[\Phi '(|z|^2)]^{1/2}+|w-P_zw|[\Psi '(|z|^2)]^{1/2}\leq t\varrho(z,w).
  \end{align}
\begin{lemma}\label{3p7}
		Let $f \in L^2_{\mathrm{loc}}(\mathbb{C}^n)$ and $t$ be the constant from \eqref{3e18}. Suppose $0<r\leq \frac{1}{16t^4}$. Then $f$ admits a decomposition $f=f_1+f_2$ satisfying the following property.\\
\textup{(A)} $f_1\in C^2(\mathbb{C}^n)$, for $z\in \mathbb{C}^n,$
$$
\frac{|\overline{\partial}f_1(z)|}{[\Phi'(|z|^2)]^{1/2}}
+\frac{\left|\overline{\partial}f_1(z)\land \overline{\partial}|z|\right|}{[\Psi'(|z|^2)]^{1/2}}
\lesssim G_{16t^4r}(f)(z),$$
and
$$M_r\left(\frac{|\overline{\partial}f_1(\cdot)|}{[\Phi'(|\cdot|^2)]^{1/2}}
+\frac{\left|\overline{\partial}f_1(\cdot)\land \overline{\partial}|\cdot|\right|}{[\Psi'(|\cdot|^2)]^{1/2}}\right)(z)
\lesssim G_{16t^4r}(f)(z).
$$
\textup{(B)} $f_2\in L^2_{\mathrm{loc}}(\mathbb{C}^n)$ and
$M_r(f_2)(z)\lesssim G_{16t^4r}(f)(z)$ for $ z\in \mathbb{C}^n.$
	\end{lemma}
	\begin{proof}
Given $0<r\leq \frac{1}{16t^4}$, let $\{z_k\}_{k=1}^\infty$ be a sequence in $\mathbb{C}^n$ such that $\cup_kD(z_k,r/2)=\mathbb{C}^n$ and $D(z_j,r/m)\cap D(z_k,r/m)=\emptyset$ for $j\neq k$ where $m>0$ is a constant. By Lemma \ref{3p5}, there is a partition of unity $\{\eta_k\}^\infty_{k=1}$ subordinated to $\{D(z_k,r)\}_{k=1}^\infty$.
	Given $f \in L^2_{\mathrm{loc}}(\mathbb{C}^n)$, for $k=1,2,\dots$, pick $h_k \in H(D(z_k,r))$ as in Lemma \ref{3p6} such that
	\begin{align}\label{3e19}
		M_{r}(f-h_k)(z_k)=G_{r}(f)(z_k).
	\end{align}
	Set
	%\begin{align*}
		$f_1=\sum_{k=1}^{\infty}h_k\eta_k$ and $f_2=f-f_1.$
	%\end{align*}
For each $z\in \mathbb{C}^n$, without loss of generality, we suppose $z\in D(z_1,r/2)$. Notice that the support of $\eta_k$ is contained in $D(z_k,r).$ By applying Lemma \ref{3p4} with $b=1$, we know that $f_1$ is well defined and is a sum of finite terms. Thus, we have
%\begin{align}\label{3e20}
  $\overline{\partial}f_1(z)=\sum_{k=1}^{\infty}h_k(z)\overline{\partial}\eta_k(z).$
%\end{align}
Since $\sum \eta_k\equiv 1$, there holds
%\begin{align}\label{3e21}
  $\sum_{k=1}^{\infty}h_1(z)\overline{\partial}\eta_k(z)=
h_1(z)\overline{\partial}\left(\sum_{k=1}^{\infty}\eta_k\right)(z)=0.$
%\end{align}
It follows % from \eqref{3e20} and \eqref{3e21}
that
$$\overline{\partial}f_1(z)=\sum_{k=1}^{\infty}[h_k(z)-h_1(z)]\overline{\partial}\eta_k(z).$$
%Furthermore,
%$$\overline{\partial}f_1(z)\land \overline{\partial}|z|=\sum_{k=1}^{\infty}[h_k(z)-h_1(z)]\overline{\partial}\eta_k(z)\land \overline{\partial}|z|.$$
By the subharmonicity of $|h_k-h_1|^2$ over $D(z_k,r)$, Lemma \ref{3p5} (A), Lemma \ref{3p1} and \eqref{3e19}, we obtain, for $1\leq j\leq n$,
		\begin{align}\label{3e22}
			\left|\frac{\partial f_1}{\partial w_j}(z)\right|
			&\leq \sum_{k=1}^\infty|h_k(z)-h_1(z)|\cdot\left|\frac{\partial \eta_k}{\partial w_j}(z)\right|\notag\\
			&\leq C\sum_{\{k:z \in D(z_k,r)\}}[\Phi'(|z_k|^2)]^{1/2}M_{r}(h_k-h_1)(z)\notag\\
			&\leq C[\Phi'(|z|^2)]^{1/2}\sum_{\{k:z \in D(z_k,r)\}}[M_{r}(f-h_k)(z)+M_{r}(f-h_1)(z)]\notag\\
&\leq C [\Phi'(|z|^2)]^{1/2}\sum_{\{k:z \in D(z_k,r)\}}G_{r}(f)(z_k).
		\end{align}
By \eqref{3e18} and the triangle inequality, it is easy to see that for $z\in D(z_k,r)$,
\begin{align}\label{ex3e22}
  D(z_k,r)\subseteq B(z,4tr)\subseteq D(z,4t^2r).
\end{align}
Therefore, we have $G_{r}(f)(z_k)\lesssim G_{4t^2r}(f)(z)$ when $z\in D(z_k,r)$. This, together with \eqref{3e22} and Lemma \ref{3p4} with $b=1$, further implies that
\begin{align}\label{3e23}
  \left|\frac{\partial f_1}{\partial w_j}(z)\right|
  \leq C [\Phi'(|z|^2)]^{1/2}\sum_{\{k:z \in D(z_k,r)\}}G_{4t^2r}(f)(z)\leq C [\Phi'(|z|^2)]^{1/2}G_{4t^2r}(f)(z).
\end{align}
Therefore,
\begin{align}\label{3e24}
|\overline{\partial}f_1(z)|[\Phi'(|z|^2)]^{-1/2}\lesssim G_{4t^2r}(f)(z)\textup{ for } z\in \mathbb{C}^n.
\end{align}
Similarly,
$$\frac{1}{|z|}\left|\frac{\partial f_1}{\partial\overline{z}_j}(z)z_l-
  \frac{\partial f_1}{\partial\overline{z}_l}(z)z_j\right|
  \leq \sum_{k=1}^\infty|h_k(z)-h_1(z)|\cdot\frac{1}{|z|}\left|\frac{\partial \eta_k}{\partial\overline{z}_j}(z)z_l-
  \frac{\partial \eta_k}{\partial\overline{z}_l}(z)z_j\right|,$$
  for $1\leq j,l\leq n.$ By the same argument as \eqref{3e23}, we have
 $$\frac{1}{|z|}\left|\frac{\partial f_1}{\partial\overline{z}_j}(z)z_l-
  \frac{\partial f_1}{\partial\overline{z}_l}(z)z_j\right|
  \lesssim  [\Psi'(|z|^2)]^{1/2}G_{4t^2r}(f)(z),$$
 for $1\leq j,l\leq n.$ It follows that
\begin{align}\label{3e25}
\left|\overline{\partial}f_1(z)\land \overline{\partial}|z|\right|[\Psi'(|z|^2)]^{-1/2}\lesssim G_{4t^2r}(f)(z)\textup{ for }  z\in \mathbb{C}^n.
\end{align}
By \eqref{3e18} and the triangle inequality again, we know that for $w\in D(z,r)$
$$ D(w,4t^2r)\subseteq B(z,16t^3r)\subseteq D(z,16t^4r).$$
This, together with \eqref{3e24} and \eqref{3e25}, further implies that
\begin{align}\label{3e26}
 M_r\left(|\overline{\partial}f_1(\cdot)|[\Phi'(|\cdot|^2)]^{-1/2}\right)(z)\lesssim G_{16t^4r}(f)(z),
\end{align}
and
\begin{align}\label{3e27}
 M_r\left(\left|\overline{\partial}f_1(\cdot)\land \overline{\partial}|\cdot|\right|[\Psi'(|\cdot|^2)]^{-1/2}\right)(z)\lesssim G_{16t^4r}(f)(z).
\end{align}
By \eqref{3e24}, \eqref{3e25}, \eqref{3e26}, \eqref{3e27} and the fact that $G_{4t^2r}(f)(z)\lesssim G_{16t^4r}(f)(z)$, we see that (A) holds. We are going to prove the assertion of $f_2$. By Cauthy-Schwarz inequality, we deduce that, for every $z\in \mathbb{C}^n$,
\begin{align*}
  |f_2(z)|^2&=\left|\sum_{k=1}^{\infty}[f(z)-h_k(z)]\eta_k(z)\right|^2
  \leq \left(\sum_{k=1}^{\infty}|f(z)-h_k(z)|^2\eta_k(z)\right)\left(\sum_{k=1}^{\infty}\eta_k(z)\right)\\
  &=\sum_{k=1}^{\infty}|f(z)-h_k(z)|^2\eta_k(z).
\end{align*}
It follows that
		\begin{align}\label{3e28}
			M_r(f_2)(z)^2 &\leq \frac{1}{|D(z,r)|}\int_{D(z,r)}\sum_{k=1}^\infty|f(w)-h_k(w)|^2\eta_k(w)dV(w)\notag\\
			&\leq C \sum_{k=1}^{\infty}\frac{1}{|D(z_k,r)|}\int_{D(z,r)\cap D(z_k,r)}|f(w)-h_k(w)|^2dV(w)\notag\\
			&\leq C \sum_{\{k:D(z,r)\cap D(z_k,r)\ne \emptyset\}}G_{r}(f)(z_k)^2.
		\end{align}
By \eqref{3e18}, we deduce that when $D(z,r)\cap D(z_k,r)\ne \emptyset$,
$$D(z,r)\subseteq B(z,2tr),D(z_k,r)\subseteq B(z_k,2tr) \textup{ and }B(z,2tr)\cap B(z_k,2tr)\ne \emptyset.$$
This, together with the triangle inequality and \eqref{3e18} again, implies that
$$ D(z_k,r)\subseteq B(z_k,2tr)\subseteq B(z,6tr)\subseteq D(z,6t^2r)\subseteq D(z,16t^4r).$$
It follows that $G_{r}(f)(z_k)\lesssim G_{16t^4r}(f)(z)$ when $D(z,r)\cap D(z_k,r)\ne \emptyset$. From this, \eqref{3e28} and Lemma \ref{3p4} with $a=b=1$, we obtain
%\begin{align}\label{3e29}
  $M_r(f_2)(z) \leq C G_{16t^4r}(f)(z).$
%\end{align}
The proof is finished.
	\end{proof}
Let $\varphi:\mathbb{C}^n\to \mathbb{R}$ be a plurisubharmonic function. For a (0,1) form $u=\sum_{j=1}^{n}u_jd\overline{z}_j$, we define
$$|u|^2_{i\partial \overline{\partial}\varphi}:=\sum_{1\leq j,k\leq n}a^{jk}u_j\overline{u}_k,$$
where $(a^{jk})_{n\times n}$ is the inverse of the Hermitian matrix
$$(a_{jk}(z))_{n\times n}=\frac{\partial \varphi}{\partial z_j\partial \overline{z}_k}(z).$$
Given a (0,1) form $u=\sum_{j=1}^{n}u_jd\overline{z}_j$, the original H\"ormander estimate then says that, for $\varphi$ plurisubharmonic, if $u$ is a $\overline{\partial}$-closed (0,1) form with $L^2_{\mathrm{loc}}(\mathbb{C}^n)$ coefficients, then there is a solution $Su$ of $\overline{\partial}v=u$ such that
$$\int_{\mathbb{C}^n}|Su(z)|^2e^{-\varphi(z)}dV(z)\leq \int_{\mathbb{C}^n}|u(z)|^2_{i\partial \overline{\partial}\varphi}e^{-\varphi(z)}dV(z).$$
This estimate is implicit in \cite{H}, where the norm $|u|_{i\partial \overline{\partial}\varphi}$ is replaced by $|u|^2/c$, with $c$ the smallest eigenvalue of the matrix $(a_{jk})_{n\times n}$. The precise formulation can be found in \cite{D}. For $A=(a_{ij})_{n\times n}$ and $\xi=(\xi_1,\cdots,\xi_n)$, set $\xi A=(\sum_{j=1}^{n}\xi_ja_{j1},\cdots, \sum_{j=1}^{n}\xi_ja_{jn})$.
\begin{lemma}\label{n3p9}
  Let $A$ be a $n\times n$ Hermitian matrix and $\mathcal{M}$ be a subspace of $\mathbb{C}^n$. If there exist $\alpha,\beta>0$ such that
  $$\langle \xi A,\xi \rangle=\alpha |P_\mathcal{M}\xi|^2+\beta|(I-P_\mathcal{M})\xi|^2,$$
  for any $\xi \in \mathbb{C}^n$, where $P_\mathcal{M}$ is the projection from $\mathbb{C}^n$ onto $\mathcal{M}$, then $A$ is inversible and
  $$\langle \xi A^{-1},\xi \rangle=\frac{1}{\alpha} |P_\mathcal{M}\xi|^2+\frac{1}{\beta}|(I-P_\mathcal{M})\xi|^2,$$
  for any $\xi \in \mathbb{C}^n$.
\end{lemma}
\begin{proof}
  Define a linear tansformation $\mathscr{A}$ on $\mathbb{C}^n$ by $\mathscr{A}\xi=\xi A\textup{ for }\xi=(\xi_1,\cdots,\xi_n) \in \mathbb{C}^n.$
  Since $A$ is Hermitian, $\mathscr{A}$ is self-adjoint. Define another linear tansformation $\mathscr{B}$ on $\mathbb{C}^n$ by $\mathscr{B}\xi=\alpha P_\mathcal{M}\xi +\beta(I-P_\mathcal{M})\xi\textup{ for }\xi \in \mathbb{C}^n,$ then $\mathscr{B}$ is also self-adjoint. From page 7 of \cite{Zhu}, we know that
  $$\|\mathscr{A}-\mathscr{B}\|=\sup\limits_{\xi \in \mathbb{C}^n,|\xi|=1}|\langle (\mathscr{A}-\mathscr{B})\xi,\xi \rangle|=0.$$
  Therefore, $\mathscr{A}=\mathscr{B}$. It is easy to check that $\mathscr{A}^{-1}=\alpha^{-1}P_\mathcal{M} +\beta^{-1}(I-P_\mathcal{M})$ and the corresponding matrix is $A^{-1}$, namely, $\mathscr{A}^{-1}\xi=\xi A^{-1}$ for $\xi \in \mathbb{C}^n.$ It follows that
  $$\langle \xi A^{-1},\xi \rangle=\langle \mathscr{A}^{-1}\xi,\xi \rangle=\frac{1}{\alpha} |P_\mathcal{M}\xi|^2+\frac{1}{\beta}|(I-P_\mathcal{M})\xi|^2,$$
  for any $\xi \in \mathbb{C}^n$. The proof is finished.
\end{proof}
Now we set $\varphi(z)=\Psi(|z|^2)$. We have the following $L^2$ estimate.
\begin{lemma}\label{3p8}
If $u=\sum_{j=1}^{n}u_jd\overline{z}_j$ is a $\overline{\partial}$-closed (0,1) form with $u_j\in L^2_{\mathrm{loc}}(\mathbb{C}^n)$, then there exists a solution $Su$ of $\overline{\partial}v=u$ such that
$$\int_{\mathbb{C}^n}|Su(z)|^2e^{-\Psi(|z|^2)}dV(z)\lesssim \int_{\mathbb{C}^n}\left(\frac{|u(z)|^2}{\Phi '(|z|^2)}+\frac{\left|u(z)\land\overline{\partial}|z|\right|^2}{\Psi '(|z|^2)}\right)e^{-\Psi(|z|^2)}dV(z).$$
\end{lemma}
\begin{proof}
Set $z=(z_1,\cdots,z_n)$ and $\varphi(z)=\Psi(|z|^2)$. Observe that
$$a_{jk}(z):=\frac{\partial \varphi}{\partial z_j\partial \overline{z}_k}(z)=\Psi '(|z|^2)\delta_{jk}+\Psi ''(|z|^2)\overline{z}_jz_k,$$
where $\delta_{jk}=1$ for $j=k$ and $\delta_{jk}=0$ otherwise. It is easy to check that
$$\sum_{j,k=1}^{n}a_{jk}(z)\xi_j\overline{\xi}_k=|\xi|^2\Psi '(|z|^2)+|\langle z,\xi \rangle |^2\Psi ''(|z|^2)=\Phi '(|z|^2)|P_z\xi|^2+\Psi '(|z|^2)|(I-P_z)\xi|^2,$$
for $\xi=(\xi_1\cdots,\xi_n)\in \mathbb{C}^n$. This shows that $\varphi(z)=\Psi(|z|^2)$ is plurisubharmonic. Set $A=(a_{jk})_{n\times n}$, then
$$\sum_{j,k=1}^{n}a_{jk}(z)\xi_j\overline{\xi}_k=\langle \xi A,\xi\rangle$$
for $\xi=(\xi_1\cdots,\xi_n)\in \mathbb{C}^n$. % By Lemma \ref{n3p9}, we have
%$$\left(\overline{a_{jk}(z)}\right)_{n \times n}=\Phi '(|z|^2)P_z+\Psi '(|z|^2)(I-P_z).$$
%where $P_z$ is identified with the matrix
%$$\left(\frac{z_i\overline{z_j}}{|z|^2}\right)_{n\times n}=$$
%via the standard basis for $\mathbb{C}^n$. It follows that
%$$\left(\overline{a_{jk}(z)}\right)_{n \times n}^{-1}=\Phi '(|z|^2)^{-1}P_z+\Psi '(|z|^2)^{-1}(I-P_z).$$
If we identity $u$ with $(u_1,\cdots,u_n)\in \mathbb{C}^n$, then by Lemma \ref{n3p9} we have
\begin{align}\label{3e30}
  |u(z)|^2_{i\partial \overline{\partial}\varphi}&=\langle u(z)A^{-1},u(z) \rangle=\frac{|P_zu(z)|^2}{\Phi '(|z|^2)}+\frac{|(I-P_z)u(z)|^2}{\Psi '(|z|^2)}
  \leq \frac{|u(z)|^2}{\Phi '(|z|^2)}+\frac{|u(z)|^2-|P_zu(z)|^2}{\Psi '(|z|^2)}.
\end{align}
Notice that
$$u(z)\land \overline{\partial}|z|=u(z)\land\left(\sum_{k=1}^{n}\frac{z_kd\overline{z}_k}{2|z|}\right)=
\sum_{1\leq j<k\leq n}\left(\frac{u_j(z)z_k}{2|z|}-\frac{u_k(z)z_j}{2|z|}\right)d\overline{z}_j\land d\overline{z}_k. $$
It follows that
\begin{align*}
  |u(z)\land\overline{\partial}|z||^2&=\sum_{1\leq j<k\leq n}\left|\frac{u_j(z)z_k}{2|z|}-\frac{u_k(z)z_j}{2|z|}\right|^2\\
  &=\sum_{1\leq j<k\leq n}\left(\frac{|u_j(z)|^2|z_k|^2}{4|z|^2}+\frac{|u_k(z)|^2|z_j|^2}{4|z|^2}\right)+\left(
  -\frac{\overline{u_j(z)}\overline{z}_ku_k(z)z_j}{4|z|^2}-\frac{u_j(z)z_k\overline{u_k(z)}\overline{z}_j}{4|z|^2}\right)\\
  &=\sum_{1\leq j,k\leq n}\frac{|u_j(z)|^2|z_k|^2}{4|z|^2}-\frac{\overline{u_j(z)}\overline{z}_ku_k(z)z_j}{4|z|^2}\\
  &=\frac{1}{4}\left(\sum_{j=1}^{n}|u_j(z)|^2\right)\left(\sum_{k=1}^{n}|z_k|^2/|z|^2\right)
  -\frac{1}{4}\left(\sum_{j=1}^{n}z_j\overline{u_j}/|z|\right)\left(\sum_{k=1}^{n}u_k(z)\overline{z}_k/|z|\right)\\
  &=\frac{1}{4}|u(z)|^2-\frac{1}{4}|P_zu(z)|^2,
\end{align*}
which, together with \eqref{3e30}, further implies that
$$|u(z)|^2_{i\partial \overline{\partial}\varphi}\lesssim
\frac{|u(z)|^2}{\Phi '(|z|^2)}+\frac{\left|u(z)\land\overline{\partial}|z|\right|^2}{\Psi '(|z|^2)}.$$
By applying H\"omander's $L^2$ estimates, we arrive at the desired result.
\end{proof}
\begin{proof}[Proof of Theorem \ref{1p1}]
			(A)$\Rightarrow$(B): %From Lemma \ref{2p2}, we have
%$$K(w,w)\simeq \Phi'(|w|^2)\Psi'(|w|^2)^{n-1}e^{\Psi(|w|^2)}.$$
%Therefore, b
By Lemma \ref{2p1}, \eqref{ex2e6}, and Lemma \ref{3p1}, we know that, there is a constant $r_0<1/8$ such that for every $z\in \mathbb{C}^n$ and $w \in D(z,r)$ with $0<r\leq r_0$,
		\begin{align}
			\label{3e31}
			|k_z(w)|^2e^{-\Psi(|w|^2)}\simeq K(w,w)e^{-\Psi(|w|^2)}\simeq  \Phi'(|w|^2)\Psi'(|w|^2)^{n-1} \simeq  \Phi'(|z|^2)\Psi'(|z|^2)^{n-1}\simeq \frac{1}{|D(z,r)|}\neq 0.
		\end{align}
Moreover, $k_z(\cdot)^{-1}\in H(D(z,r))$. Then for all $z \in \mathbb{C}^n$ and $0<r\leq r_0$, we have
\begin{align*}			
\|H_f\|^2&\geq \|H_fk_z\|_\Psi^2=\int_{\mathbb{C}^n}|fk_z(w)-P(fk_z)(w)|^2e^{-\Psi(|w|^2)}dV(w)\\
			&=\int_{\mathbb{C}^n}|f(w)-P(fk_z)(w)k_z(w)^{-1}|^2|k_z(w)|^2e^{-\Psi(|w|^2)}dV(w)\\
			&\geq \int_{D(z,r)}|f(w)-P(fk_z)(w)k_z(w)^{-1}|^2|k_z(w)|^2e^{-\Psi(|w|^2)}dV(w)\\
			&\simeq \frac{1}{|D(z,r)|}\int_{D(z,r)}\left|f(w)-k_z(w)^{-1}P(fk_z)(w)\right|^2dV(w)\\
			&\geq G_r(f)(z)^2.
\end{align*}
		It follows that $f\in \textup{BDA}_{\Psi,r}$ for all $r\leq r_0$. Moreover,
\begin{align}\label{3e32}
  \|f\|_{\textup{BDA}_{\Psi,r}}\leq C\|H_f\|.
\end{align}
		(B)$\Rightarrow$(A): It is easy to see that if $f\in \textup{BDA}_{\Psi,r_0}$ for some $r_0$, then $f\in \textup{BDA}_{\Psi,r}$ for all $r\leq r_0$. Moreover,
\begin{align}\label{3e33}
  \|f\|_{\textup{BDA}_{\Psi,r}}\leq C\|f\|_{\textup{BDA}_{\Psi,r_0}},
\end{align}
where $C>0$ is independent of $f$. Without loss of generality, we may suppose that $f\in \textup{BDA}_{\Psi,r}$ with $0<r<1/8$. Thus, we can decompose $f=f_1+f_2$ as in Lemma \ref{3p7} with the following property:
\begin{align}\label{3e34}
f_1\in C^2(\mathbb{C}^n),\qquad
\sup_{z\in \mathbb{C}^n }\left\{\frac{|\overline{\partial}f_1(z)|}{\Phi'(|z|^2)^{1/2}}
+\frac{\left|\overline{\partial}f_1(z)\land \overline{\partial}|z|\right|}{\Psi'(|z|^2)^{1/2}}\right\}\leq C\|f\|_{\textup{BDA}_{\Psi,r}}<\infty,
\end{align}
and
\begin{align}\label{3e35}
f_2\in L^2_{\mathrm{loc}}(\mathbb{C}^n),\qquad \sup_{z\in \mathbb{C}^n }\left\{M_{\frac{r}{16t^4}}(f_2)(z)\right\}\leq C\|f\|_{\textup{BDA}_{\Psi,r}}<\infty.
\end{align}
Then $H_{f_1}$ and $H_{f_2}$ are well defined on $\Gamma= \text{span}\{K_z:z\in \mathbb{C}^n\}$. In fact, by Lemma \ref{2p4} with $p=1$(here $\frac{1}{2}\Psi$ is replaced by $\Psi$), for $g\in \Gamma$ and $z\in \mathbb{C}^n$, we have
\begin{align*}
  \int_{\mathbb{C}^n}|f_2(w)g(w)K_z(w)|e^{-\Psi(|w|^2)}dV(w)\leq C\int_{\mathbb{C}^n}|g(w)K_z(w)|e^{-\Psi(|w|^2)}\widehat{|f_2|}_{\frac{r}{16t^4}}(w)dV(w).
\end{align*}
By H\"{o}lder's inequality,
\begin{align*}
  \widehat{|f_2|}_{\frac{r}{16t^4}}(w)\leq M_{\frac{r}{16t^4}}(|f_2|)(w)\leq C\|f\|_{\textup{BDA}_{\Psi,r}}.
\end{align*}
Thus, by applying Cauchy-Schwarz inequality, we have
\begin{align*}
  \int_{\mathbb{C}^n}|f_2(w)g(w)K_z(w)|e^{-\Psi(|w|^2)}dV(w)\leq C\int_{\mathbb{C}^n}|g(w)K_z(w)|e^{-\Psi(|w|^2)}dV(w)\leq \|g\|_\Psi\|K_z\|_\Psi<\infty.
\end{align*}
This implies that $H_{f_2}$ is well defined on $\Gamma.$ Since $H_{f_1}=H_f-H_{f_2}$, we see that $H_{f_1}$ is also well defined on $\Gamma.$
Therefore, it suffices to prove that $H_{f_1}$ and $H_{f_2}$ are bounded. By \cite[Theorem 19]{WTH}, $d\mu=|f_2|^2dV$ is a (2,2) Fock-Carleson measure, i.e.
$$||f_2g||_\Psi\leq C\sup_{z\in \mathbb{C}^n }\left\{M_{\frac{r}{16t^4}}(f_2)(z)\right\} ||g||_\Psi.$$
Therefore,
		$$\|H_{f_2}g\|_\Psi=\|(I-P)(f_2g)||_\Psi\leq C\|f_2g\|_\Psi\leq C\sup_{z\in \mathbb{C}^n }\left\{M_{\frac{r}{16t^4}}(f_2)(z)\right\} \|g\|_\Psi.$$
It follows that $H_{f_2}$ is bounded. Furthermore,
\begin{align}\label{3e39}
  \|H_{f_2}\|\leq C
\sup_{z\in \mathbb{C}^n }\left\{M_{\frac{r}{16t^4}}(f_2)(z)\right\}.
\end{align}
 Now we are going to prove that $H_{f_1}$ is bounded. Let $g \in F^2_\Psi$.  Since $g$ is holomorphic and $f_1\in C^2(\mathbb{C}^n)$, we have
$$\overline{\partial} (g\overline{\partial}f_1)=\overline{\partial}g\land \overline{\partial}f_1+g\land \overline{\partial}^2f_1=0.$$
Therefore, $g\overline{\partial}f_1$ is a $\overline{\partial}$-closed (0,1) form with $L^2_{\mathrm{loc}}$ coefficients.
%On the other hand, we have for each $j=1,\cdots,n$
%		\begin{align*}
%            \left\|g\frac{\partial f_1}{\partial\overline{z}_j}\right\|_\Psi^2&\leq C\int_{\mathbb{C}^n}|g(w)|^2e^{-\Psi(|w|^2)}M_r\left(\left|\frac{\partial f_1}{\partial\overline{z_j}}\right|\right)^2(w)dV(w)\leq C\|g\|_\Psi^2<\infty.
%		\end{align*}
%This shows that the coefficients of $g\overline{\partial}f_1$ are in $L^2_{\Psi}$.
Thus by Theorem \ref{3p8}, we see that for any $ g\in \Gamma,$ there is a solution $S(g\overline{\partial}f_1)$ of the equation $\overline{\partial}v=g\overline{\partial}f_1$ such that
\begin{align}\label{3e36}
  \int_{\mathbb{C}^n}|S(g\overline{\partial}f_1)(z)|^2e^{-\Psi(|z|^2)}dV(z)\lesssim \int_{\mathbb{C}^n}\left(\frac{|g(z)\overline{\partial}f_1(z)|^2}{\Phi '(|z|^2)}+\frac{\left|g(z)\overline{\partial}f_1(z)\land\overline{\partial}|z|\right|^2}{\Psi '(|z|^2)}\right)e^{-\Psi(|z|^2)}dV(z).
\end{align}
Since $\overline{\partial}(H_{f_1}g)-\overline{\partial}(S(g\overline{\partial}f_1))=0$, we see that $H_{f_1}g-S(g\overline{\partial}f_1)$ is holomorphic. Since $H_{f_1}g=H_{f}g-H_{f_2}g\in L^2_\Psi$, we have $H_{f_1}g-S(g\overline{\partial}f_1)\in F^2_\Psi.$ This, together with the fact $H_{f_1}g=f_1g-P(f_1g)\perp F^2_\Psi$, implies that
\begin{align}\label{3e37}
    \|H_{f_1}g\|^2_\Psi&=\|S(g\overline{\partial}f_1)\|^2_\Psi-\|H_{f_1}g-S(g\overline{\partial}f_1)\|^2_\Psi\leq\|S(g\overline{\partial}f_1)\|^2_\Psi.
\end{align}
Combining this with \eqref{3e36}, we obtain
\begin{align*}
    \|H_{f_1}g\|^2_\Psi
    &\lesssim \int_{\mathbb{C}^n}\left(\frac{|g(z)\overline{\partial}f_1(z)|^2}{\Phi '(|z|^2)}+\frac{\left|g(z)\overline{\partial}f_1(z)\land\overline{\partial}|z|\right|^2}{\Psi '(|z|^2)}\right)e^{-\Psi(|z|^2)}dV(z)\\
    &\leq  \sup_{z\in \mathbb{C}^n }\left\{\frac{|\overline{\partial}f_1(z)|^2}{\Phi'(|z|^2)}
+\frac{\left|\overline{\partial}f_1(z)\land \overline{\partial}|z|\right|^2}{\Psi'(|z|^2)}\right\}\|g\|_\Psi^2.
\end{align*}
This shows that $H_{f_1}$ is bounded. Furthermore,
\begin{align}\label{3e38}
  \|H_{f_1}\|\leq C\sup_{z\in \mathbb{C}^n }\left\{\frac{|\overline{\partial}f_1(z)|^2}{\Phi'(|z|^2)}
+\frac{\left|\overline{\partial}f_1(z)\land \overline{\partial}|z|\right|^2}{\Psi'(|z|^2)}\right\}.
\end{align}
Estimate \eqref{1e4} comes from \eqref{3e32}, \eqref{3e34}, \eqref{3e35}, \eqref{3e39} and \eqref{3e38}. The proof is completed.
	\end{proof}
	\begin{proof}[Proof of Theorem \ref{1p2}]
		(A)$\Rightarrow$(B): By the proof of Theorem \ref{1p1}, we know that there is a constant $r_0$ such that for any $z \in \mathbb{C}^n$ and $0<r\leq r_0$,
		\begin{align*}
			||H_fk_z||_\Psi\geq G_r(f)(z).
		\end{align*}
	Since $\{k_z:z\in \mathbb{C}^n\}$ converges weakly to 0 in $F^2_\Psi$ as $z \to \infty$ (see Lemma 12 in \cite{WTH}), we have
	$$\lim_{z\to \infty}||H_fk_z||_\Psi=0,$$
	which implies that
	$$\lim_{z\to \infty}G_r(f)(z)=0.$$
Thus, we have $f\in \textup{VDA}_{\Psi,r}$.

	(B)$\Rightarrow$(A): By the proof of Theorem \ref{1p1}, we can decompose $f=f_1+f_2$ as in Lemma \ref{3p7} with the following property:
\begin{align}\label{3e40}
f_1\in C^2(\mathbb{C}^n),\qquad
\lim_{z\to \infty}\left\{\frac{|\overline{\partial}f_1(z)|^2}{\Phi'(|z|^2)}
+\frac{\left|\overline{\partial}f_1(z)\land \overline{\partial}|z|\right|^2}{\Psi'(|z|^2)}\right\}=0,
\end{align}
and
\begin{align}\label{3e41}
f_2\in L^2_{\mathrm{loc}}(\mathbb{C}^n),\qquad \lim_{z\to \infty}\left\{M_{\frac{r}{16t^4}}(f_2)(z)\right\}=0.
\end{align}
Then $H_{f_1}$ and $H_{f_2}$ are well defined on $\Gamma= \text{span}\{K_z:z\in \mathbb{C}^n\}$. It suffices to prove that $H_{f_1}$ and $H_{f_2}$ are both compact. To this end, let $\{g_j\}_{j=1}^\infty$ be a bounded sequence in $F^2_{\Psi}$ with the property that
$$\lim_{j\to\infty}\sup_{z\in E}|g_j(z)|=0, \textup{ and }\|g_j\|_{\Psi}\leq M_1,$$
where $E$ is an arbitrary compact subset of $\mathbb{C}^n$ and $M_1$ is a positive constant. We are going to prove that for $k=1,2$
$$\lim_{j\to\infty}\|H_{f_k}g_j\|_{\Psi}=0.$$
By checking the proof of Theorem \ref{1p1}, we have
	$$||H_{f_1}g_j||_\Psi^2 \lesssim \int_{\mathbb{C}^n}|g_j(z)|^2\left\{\frac{|\overline{\partial}f_1(z)|}{\Phi'(|z|^2)}+\frac{|\overline{\partial}f_1(z)\land \overline{\partial}|z||^2}{\Psi'(|z|^2)}\right\}e^{-\Psi(|z|^2)}dV(z).$$
	By \eqref{3e40}, for every $\varepsilon>0$, there is a constant $M_2>0$ such that
	$$\sup_{|z|\geq M_2}\left\{\frac{|\bar{\partial}f_1(z)|}{\Phi'(|z|^2)}+\frac{|\overline{\partial}f_1(z)\wedge \overline{\partial}|z||^2}{\Psi'(|z|^2)}\right\}<\varepsilon.$$
	By \eqref{3e41}, for the given $\varepsilon$ and $M_2$ above, there also exists an integer $J$ such that for $j\geq J$,
	$$\sup_{|z|\leq M_2}|g_j(z)|<\varepsilon.$$
	Therefore, for every $\varepsilon>0$, there is an integer $J$ such that for $j\geq J$,
	\begin{align*}
		||H_{f_1}g_j||_\Psi^2 &\lesssim \int_{|z|\leq M_2}|g_j(z)|^2\left\{\frac{|\overline{\partial}f_1(z)|}{\Phi'(|z|^2)}+\frac{|\overline{\partial}f_1(z)\wedge \overline{\partial}|z||^2}{\Psi'(|z|^2)}\right\}e^{-\Psi(|z|^2)}dV(z)\\
&\qquad +\int_{|z|\geq M_2}|g_j(z)|^2\left\{\frac{|\overline{\partial}f_1(z)|}{\Phi'(|z|^2)}+\frac{|\overline{\partial}f_1(z)\wedge \overline{\partial}|z||^2}{\Psi'(|z|^2)}\right\}e^{-\Psi(|z|^2)}dV(z)\\
&\lesssim \varepsilon^2\left\|\frac{|\overline{\partial}f_1(z)|}{\Phi'(|z|^2)}+\frac{|\overline{\partial}f_1(z)\wedge \overline{\partial}|z||^2}{\Psi'(|z|^2)}\right\|_{L^\infty(\mathbb{C}^n,dV)}+\varepsilon M_1^2.
	\end{align*}
Therefore, we have proved that $H_{f_1}$ is compact. It remains to prove
$$\lim_{j\to\infty}\|H_{f_2}g_j\|_{\Psi}=0.$$
 By \cite[Theorem 20]{WTH}, we know that $d\mu=|f_2|^2dV$ is a vanishing (2,2) Fock-Carleson measure, i.e.
$$\lim_{j\to\infty}\|f_2g_j\|_{\Psi}=0.$$
On the other hand, we have
	$$||H_{f_2}g_j||_\Psi=||(I-P)(f_2g_j)||_\Psi\leq \|I-P\|\cdot\|f_2g_j\|_\Psi.$$
It follows that $H_{f_2}$ is also compact. The proof is completed.
	\end{proof}
The following lemma shows the relationship between $\textup{BMO}$(or $\textup{VMO}$) and $\textup{BDA}_{\Psi,r}$(or $\textup{VDA}_{\Psi,r}$ respectively). Its proof is similar to Proposition 6.3 of \cite{HV2}.
\begin{lemma}\label{5p5}
  Let $f\in \mathcal{S}$ and $0<r<\infty$. Then $f\in \textup{BMO}_r$ if and only if $f,\overline{f}\in \textup{BDA}_{\Psi,r}$; $f\in \textup{VMO}_r$ if and only if $f,\overline{f}\in \textup{VDA}_{\Psi,r}$. Furthermore,
  \begin{align}\label{4e17}
    \|f\|_{\textup{BMO}_r} \simeq  \|f\|_{\textup{BDA}_{\Psi,r}}+\|\overline{f}\|_{\textup{BDA}_{\Psi,r}}.
  \end{align}
\end{lemma}
As a consequence of Lemma \ref{5p5}, Theorem \ref{1p1} and Theorem \ref{1p2}, we obtain characterizations of boundedness and compactness of coupled Hankel operators which are the main results of \cite{SY,WCZ,TW}.
\begin{theorem}
  Let $f\in \mathcal{S}$. Then $H_{f},H_{\overline{f}}:F^2_\Psi \to L^2_\Psi$ are both bounded if and only if $f\in \textup{BMO}_r$ for some(or any) $r>0$. Furthermore,
  $$\|H_{f}\|+\|H_{\overline{f}}\|\simeq \|f\|_{\textup{BMO}_r}.$$
\end{theorem}
\begin{theorem}
  Let $f\in \mathcal{S}$. Then $H_{f},H_{\overline{f}}:F^2_\Psi \to L^2_\Psi$ are both compact if and only if $f\in \textup{VMO}_r$ for some(or any) $r>0$.
\end{theorem}
\section{Some Integral Estimates}\label{s4}
This section is devoted to give integral
type estimates about the reproducing kernels. We will deal with integrals of the following type frequently.
\begin{align}\label{q7}
  I_p(r,a,s,b)=\int_{a}^{\infty}y^s\Psi '(r^2+y)^{b}e^{-\frac{p}{2}\Psi(r^2+y)}\,dy,
\end{align}
where $0\leq r,a<\infty$, $-1<s<\infty$, $b\in\mathbb{R}$ and $0<p<\infty$. For simplicity, we set
$$C(r,a,s)=[a\Psi'(r^2)]^se^{-\frac{p}{4}\Psi'(r^2)a}+\int_{\frac{p}{2}\Psi'(r^2)a}^{\infty}u^{s}e^{-u}\,du.$$
In this paper, we are interested in the following two choices:(1) $a=0$ and (2) $a\Psi'(r^2)\to \infty.$ Set
$$C_\Sigma(r,a,s):= C(r,a,0)+C(r,a,1)+C(r,a,s).$$
Since $y^{s'}\leq 1+y+y^s$ for all $y\geq0$ and $0\leq s'\leq s$, we have
$$\sup\{C(r,a,s'):0\leq s'\leq s\}\leq C_\Sigma(r,a,s).$$
%when $a=0$ or $\frac{p}{2}\Psi'(r^2)a\geq 1.$
\begin{lemma}\label{qt1}
  Let $0\leq r,a<\infty$, $b\leq 0$. For fixed $s\geq -1$ and $0<p<\infty$, there holds
  \begin{align}\label{qq1}
    \int_{a}^{\infty}y^s\Psi '(r^2+y)^be^{-\frac{p}{2}\Psi(r^2+y)}\,dy\lesssim C(r,a,s)e^{-\frac{p}{2}\Psi(r^2)}\Psi'(r^2)^{b-s-1}.
  \end{align}
\end{lemma}
\begin{proof}
  Since $\Psi'$ is monotone increasing, we have
  \begin{align}\label{q1}
    \Psi(r^2+y)-\Psi(r^2)\geq \Psi'(r^2)y\textup{ for } y\geq 0,
  \end{align}
  and
  \begin{align}\label{q2}
    \Psi '(r^2+y)^b\leq \Psi '(r^2)^b\textup{ for } y\geq 0.
  \end{align}
  From \eqref{q1}, we see that
  \begin{align}\label{q3}
    \int_{a}^{\infty}y^{s}e^{-\frac{p}{2}\Psi(r^2+y)}\,dy
    &\leq e^{-\frac{p}{2}\Psi(r^2)}\int_{a}^{\infty}y^{s}e^{-\frac{p}{2}\Psi'(r^2)y}\,dy\notag\\
    &\simeq e^{-\frac{p}{2}\Psi(r^2)}\Psi'(r^2)^{-s-1}\int_{\frac{p}{2}\Psi'(r^2)a}^{\infty}u^{s}e^{-u}\,du\notag\\
    &\leq C(r,a,s)e^{-\frac{p}{2}\Psi(r^2)}\Psi'(r^2)^{-s-1}.
  \end{align}
  This and \eqref{q2} yield \eqref{qq1}. The proof is completed.
\end{proof}
\begin{lemma}\label{qt2}
  Let $0\leq a,s<\infty$, $b\in\mathbb{R}$ and $0<p<\infty.$ Then,
  \begin{align}\label{qq2}
    a^s\Psi'(r^2+a)^b e^{-\frac{p}{2}\Psi(r^2+a)}\leq [a\Psi'(r^2)]^se^{-\frac{p}{4}\Psi'(r^2)a}e^{-\frac{p}{2}\Psi(r^2)}\Psi'(r^2)^{b-s},
  \end{align}
   as $r\to \infty$.
\end{lemma}
\begin{proof}
  Set
  $$f_r(x)=\Psi'(r^2+x)^b e^{-\frac{p}{4}\Psi(r^2+x)}\textup{ for } x\geq 0.$$
  From \eqref{cdt3}, we know that there is a constant $C>0$ such that
  \begin{align}\label{q4}
    \Psi''(r^2+x)\leq C(r^2+x)^{-\frac{1}{2}}\Psi'(r^2+x)^{1+\eta}\leq Cr^{-1}\Psi'(r^2+x)^{1+\eta},\textup{ as } r\to \infty.
  \end{align}
  Therefore,
  \begin{align}\label{q5}
    f_r'(x)&=b\Psi'(r^2+x)^{b-1}\Psi''(r^2+x) e^{-\frac{p}{4}\Psi(r^2+x)}-\frac{p}{4}\Psi'(r^2+x)^{b+1}e^{-\frac{p}{4}\Psi(r^2+x)}\notag\\
    &\leq Cr^{-1}\Psi'(r^2+x)^{b+\eta} e^{-\frac{p}{4}\Psi(r^2+x)}-\frac{p}{4}\Psi'(r^2+x)^{b+1}e^{-\frac{p}{4}\Psi(r^2+x)}\notag\\
    &=\Psi'(r^2+x)^{b+\eta}e^{-\frac{p}{4}\Psi(r^2+x)}\left[\frac{C}{r}-\frac{p}{4}\Psi'(r^2+x)^{1-\eta}\right]\notag\\
    &\leq \Psi'(r^2+x)^{b+\eta}e^{-\frac{p}{4}\Psi(r^2+x)}\left[\frac{C}{r}-\frac{p}{4}\Psi'(0)^{1-\eta}\right]\notag\\
    &<0,
  \end{align}
   as $r\to \infty$. It follows that $f_r(x)$ is monotone decreasing. Thus, $f_r(a)\leq f_r(0)$, or equivalently,
  \begin{align}\label{q6}
    \frac{\Psi'(r^2+a)^be^{-\frac{p}{4}\Psi(r^2+a)}}{\Psi'(r^2)^be^{-\frac{p}{4}\Psi(r^2)}}\leq 1\textup{ as } r\to \infty.
  \end{align}
  By this and \eqref{q1}, we obtain
  \begin{align*}
    a^s\Psi'(r^2+a)^b e^{-\frac{p}{2}\Psi(r^2+a)}&\leq \Psi'(r^2+a)^be^{-\frac{p}{4}\Psi(r^2+a)}a^se^{-\frac{p}{4}\Psi'(r^2)a}e^{-\frac{p}{4}\Psi(r^2)}\\
    &=\frac{\Psi'(r^2+a)^be^{-\frac{p}{4}\Psi(r^2+a)}}{\Psi'(r^2)^be^{-\frac{p}{4}\Psi(r^2)}}
    [a\Psi'(r^2)]^se^{-\frac{p}{4}\Psi'(r^2)a}e^{-\frac{p}{2}\Psi(r^2)}\Psi'(r^2)^{b-s}\\
    &\leq [a\Psi'(r^2)]^se^{-\frac{p}{4}\Psi'(r^2)a}e^{-\frac{p}{2}\Psi(r^2)}\Psi'(r^2)^{b-s},
  \end{align*}
  as $r \to \infty$. The proof is finished.
\end{proof}
For $s\geq 0$ and $0<b<\infty$, we define $J(s,b)$ as
$$J(s,b)=J_1(s,b)\cup J_2(s,b)\cup J_3(s,b),$$
where
$$J_1(s,b)=\{(i,j)\in \mathbb{N}\times \mathbb{N}:0\leq s-i-\frac{j}{2}\leq1,b-i+j(\eta-1)>0\},$$
$$J_2(s,b)=\{(i,j)\in \mathbb{N}\times \mathbb{N}:s-i-\frac{j}{2}> 1,b-i+j(\eta-1)\leq 0,j\leq \frac{b-i}{1-\eta}+1\},$$
and
$$J_3(s,b)=\{(i,j)\in \mathbb{N}\times \mathbb{N}: 0\leq s-i-\frac{j}{2}\leq1,b-i+j(\eta-1)\leq 0,j\leq \frac{b-i}{1-\eta}+1\}.$$
\begin{lemma}\label{qt3}
  Let $0\leq r<\infty.$ For fixed $s\geq 0$, $b\in\mathbb{R}$ and $0<p<\infty$, there exists $C=C(s,b,p)>0$ such that
  \begin{align}\label{qq3}
      I_p(r,a,s,b)\leq CC_\Sigma(r,a,s)e^{-\frac{p}{2}\Psi(r^2)}\Psi'(r^2)^{b-s-1},
    \end{align}
    as $r\to \infty$ for all $0\leq a<\infty.$
\end{lemma}
\begin{proof}
    By Lemma \ref{qt1}, inequality \eqref{qq3} holds for $b\leq 0$. We only need to deal with $0<b<\infty$. Integrating by parts, we have
    \begin{align}\label{q7}
    &\int_{a}^{\infty}y^s\Psi '(r^2+y)^{b}e^{-\frac{p}{2}\Psi(r^2+y)}\,dy\notag\\
    &\simeq -\int_{a}^{\infty}y^s\Psi '(r^2+y)^{b-1}\,de^{-\frac{p}{2}\Psi(r^2+y)}\notag\\
    &\lesssim
    a^s\Psi'(r^2+a)^{b-1} e^{-\frac{p}{2}\Psi(r^2+a)}+
    s\int_{a}^{\infty}y^{s-1}\Psi '(r^2+y)^{b-1}e^{-\frac{p}{2}\Psi(r^2+y)}\,dy\notag\\
    &\quad+\int_{a}^{\infty}y^{s}\Psi '(r^2+y)^{b-2}\Psi ''(r^2+y)e^{-\frac{p}{2}\Psi(r^2+y)}\,dy.
    \end{align}
    By condition \eqref{cdt3}, we see that there is a constant $C>0$ such that
    \begin{align}\label{q8}
      \Psi ''(r^2+y)\leq C(r^2+y)^{-\frac{1}{2}}\Psi '(r^2+y)^{1+\eta}\leq C\min\left\{r^{-1}\Psi '(r^2+y)^{1+\eta},y^{-\frac{1}{2}}\Psi '(r^2+y)^{1+\eta}\right\},
    \end{align}
    as $r\to \infty$. By \eqref{q8}, the last term of \eqref{q7} is estimated as
    \begin{align}\label{q9}
      \int_{a}^{\infty}y^{s}\Psi '(r^2+y)^{b-2}\Psi ''(r^2+y)e^{-\frac{p}{2}\Psi(r^2+y)}\,dy\lesssim \min\{I_p(r,a,s,b+\eta-1),I_p(r,a,s-\frac{1}{2},b+\eta-1)\}.
    \end{align}
    We prove \eqref{qq3} for $0\leq s\leq 1$ first. From \eqref{q7},\eqref{q9} and Lemma \ref{qt2}, we deduce that
    \begin{align}\label{q10}
      I_p(r,a,s,b)\lesssim C(r,a,s)e^{-\frac{p}{2}\Psi(r^2)}\Psi'(r^2)^{b-s-1}+I_p(r,a,s-1,b-1)+I_p(r,a,s,b+\eta-1) \textup{ for } s>0,
    \end{align}
    as $r\to\infty$ and
    \begin{align}\label{q11}
      I_p(r,a,0,b)&\lesssim  C(r,a,0)e^{-\frac{p}{2}\Psi(r^2)}\Psi'(r^2)^{b-1}+I_p(r,a,0,b+\eta-1)\notag\\
      &\lesssim C(r,a,0)e^{-\frac{p}{2}\Psi(r^2)}\Psi'(r^2)^{b-1}
      +C(r,a,0)e^{-\frac{p}{2}\Psi(r^2)}\Psi'(r^2)^{b+\eta-2}
      + I_p(r,a,0,b+2(\eta-1))\notag\\
      &\lesssim C(r,a,0)e^{-\frac{p}{2}\Psi(r^2)}\Psi'(r^2)^{b-1}+I_p(r,a,0,b+2(\eta-1))\notag\\
      &\lesssim C(r,a,0)e^{-\frac{p}{2}\Psi(r^2)}\Psi'(r^2)^{b-1}+I_p(r,a,0,b+j(\eta-1)),
    \end{align}
    as $r\to\infty$ where $j$ is the minimum integer satisfying $b+j(\eta-1)\leq 0.$ This and Lemma \ref{qt1} imply that
    \begin{align}\label{q13}
      I_p(r,a,0,b)&\lesssim C(r,a,0)e^{-\frac{p}{2}\Psi(r^2)}\Psi'(r^2)^{b-1}, \textup{ as }r\to \infty.
    \end{align}
    By \eqref{q10} and \eqref{q13}, we see that
    \begin{align}\label{q14}
      I_p(r,a,1,b)&\lesssim C(r,a,1)e^{-\frac{p}{2}\Psi(r^2)}\Psi'(r^2)^{b-2}+I_p(r,a,0,b-1)+I_p(r,a,1,b+\eta-1)\notag\\
      &\lesssim C_\Sigma(r,a,s)e^{-\frac{p}{2}\Psi(r^2)}\Psi'(r^2)^{b-2}+I_p(r,a,1,b+\eta-1)\notag\\
      &\lesssim C_\Sigma(r,a,s)e^{-\frac{p}{2}\Psi(r^2)}\Psi'(r^2)^{b-2}+I_p(r,a,1,b+j(\eta-1)),
    \end{align}
    where $j$ is the minimum integer satisfying $b+j(\eta-1)\leq 0.$ From this and Lemma \ref{qt1}, we obtain
    \begin{align}\label{q15}
      I_p(r,a,1,b)&\lesssim C_\Sigma(r,a,s)e^{-\frac{p}{2}\Psi(r^2)}\Psi'(r^2)^{b-2}
      +C(r,a,1)e^{-\frac{p}{2}\Psi(r^2)}\Psi'(r^2)^{b-2+j(\eta-1)}\notag\\
      &\lesssim C_\Sigma(r,a,s)e^{-\frac{p}{2}\Psi(r^2)}\Psi'(r^2)^{b-2},
    \end{align}
    as $r\to \infty$. For $0<s<1$, H\"{o}lder's inequality with exponents $\frac{1}{s}$ and $\frac{1}{1-s}$, \eqref{q13} and \eqref{q15} yield
    \begin{align}\label{q16}
      I_p(r,a,s,b)\lesssim I_p(r,a,0,b)^{1-s}I_p(r,a,1,b)^s\lesssim C_\Sigma(r,a,s)e^{-\frac{p}{2}\Psi(r^2)}\Psi'(r^2)^{b-s-1}\textup{ as }r\to \infty.
    \end{align}
    Now, we are able to prove \eqref{qq3} for all $s\geq 0$. By \eqref{q7},\eqref{q9} and Lemma \ref{qt2}, we also have
    \begin{align}\label{q17}
      I_p(r,a,s,b)\lesssim C(r,a,s)e^{-\frac{p}{2}\Psi(r^2)}\Psi'(r^2)^{b-s-1}+I_p(r,a,s-1,b-1)+I_p(r,a,s-\frac{1}{2},b+\eta-1).
    \end{align}
    Notice that, for $i,j\geq 0$,
    $$(b-i+j(\eta-1))-\left(s-i-\frac{j}{2}\right)-1=b+j\left(\eta-\frac{1}{2}\right)-s-1\leq b-s-1.$$
    Using formula \eqref{q17} repeatedly, we obtain terms of the following types
    \begin{align}\label{q18}
    C(r,a,s-i-\frac{j}{2})e^{-\frac{p}{2}\Psi(r^2)}\Psi'(r^2)^{b-s-1},
    \end{align}
    and
    \begin{align}\label{q19}
    I_p(r,a,s-i-\frac{j}{2},b-i+j(\eta-1)),
    \end{align}
    where $i,j\geq 0$ are integers. If a certain term of \eqref{q19} satisfies $0\leq s-i-\frac{j}{2}\leq 1$ or $b-i+j(\eta-1)\leq 0$, then we stop iterating this one. After finite steps, we arrive at
    \begin{align}\label{q20}
      I_p(r,a,s,b)&\lesssim C_\Sigma(r,a,s)e^{-\frac{p}{2}\Psi(r^2)}\Psi'(r^2)^{b-s-1}+\sum_{(i,j)\in J(s,b)}I_p(r,a,s-i-\frac{j}{2},b-i+j(\eta-1)).
    \end{align}
From this, \eqref{q13},\eqref{q15},\eqref{q16} and Lemma \ref{qt1}, we see that \eqref{qq3} holds. The proof is completed.
\end{proof}
For $\zeta\in\mathbb{C}^{n-1}$, set $B_{n-1}(\zeta,a)=\{\xi\in\mathbb{C}^{n-1}:|\xi-\zeta|<a\}$ and $B_{n-1}(\zeta,a)^c:=\mathbb{C}^{n-1}\backslash B_{n-1}(\zeta,a).$
\begin{lemma}\label{qt4}
  Let $n\geq 2$, $0<p<\infty$ and $0\leq k<\infty$, then there exists $C=C(n,k,p)>0$ such that
  \begin{align}\label{qq4}
    \int_{B_{n-1}(0,a)^c}e^{-\frac{p}{2}\Psi(r^2+|\xi|^2)}
    [\Psi'(r^2+|\xi|^2)^{n-1}\Phi'(r^2+|\xi|^2)]^k\,dV_{n-1}(\xi)\leq CL(r,a),
  \end{align}
  as $r\to \infty$ for all $0\leq a<\infty,$ where
  \begin{align*}
    L(r,a)=C_\Sigma(r,a,n-2+\frac{k}{2})
   e^{-\frac{p}{2}\Psi(r^2)}\Psi'(r^2)^{(n-1)(k-1)}
   \Phi'(r^2)^k.
  \end{align*}
\end{lemma}
\begin{proof}
  Using spherical coordinates, we obtain
  \begin{align}\label{exq18}
    &\int_{B_{n-1}(0,a)^c}e^{-\frac{p}{2}\Psi(r^2+|\xi|^2)}
    [\Psi'(r^2+|\xi|^2)^{n-1}\Phi'(r^2+|\xi|^2)]^k\,dV_{n-1}(\xi)\notag\\
    &\simeq \int_{a^2}^{\infty}y^{2n-3}e^{-\frac{p}{2}\Psi(r^2+y^2)}
    [\Psi'(r^2+y^2)^{n-1}\Phi'(r^2+y^2)]^k\,dy\notag\\
    &\simeq \int_{a}^{\infty}y^{n-2}e^{-\frac{p}{2}\Psi(r^2+y)}
    [\Psi'(r^2+y)^{n-1}\Phi'(r^2+y)]^k\,dy.
  \end{align}
  By condition \eqref{cdt3}, we have
  \begin{align}\label{exq19}
    \Phi'(r^2+y)\simeq (r^2+y)^{\frac{1}{2}}\Psi '(r^2+y)^{1+\eta}+\Psi '(r^2+y)\simeq (r+y^{\frac{1}{2}})\Psi '(r^2+y)^{1+\eta}+\Psi '(r^2+y),
  \end{align}
  as $r\to \infty.$ This, together with \eqref{exq18} and Lemma \ref{qt3}, implies that
  \begin{align*}
   &\int_{B_{n-1}(0,a)^c}e^{-\frac{p}{2}\Psi(r^2+|\xi|^2)}[\Psi'(r^2+|\xi|^2)^{n-1}\Phi'(r^2+|\xi|^2)]^k\,dV_{n-1}(\xi)\notag\\
   &\simeq r^kI_p(r,a,n-2,(n+\eta)k)+I_p(r,a,n-2+\frac{k}{2},(n+\eta)k)
   +I_p(r,a,n-2,nk)\notag\\
   &\lesssim C_\Sigma(r,a,n-2)r^ke^{-\frac{p}{2}\Psi(r^2)}\Psi'(r^2)^{(n+\eta)k-n+1}
   +C_\Sigma(r,a,n-2+\frac{k}{2})e^{-\frac{p}{2}\Psi(r^2)}\Psi'(r^2)^{(n+\eta)k-n+1-\frac{k}{2}}\notag\\
   &\quad+C_\Sigma(r,a,n-2)e^{-\frac{p}{2}\Psi(r^2)}\Psi'(r^2)^{nk-n+1}\notag\\
   &\lesssim C_\Sigma(r,a,n-2+\frac{k}{2})e^{-\frac{p}{2}\Psi(r^2)}\Psi'(r^2)^{(n-1)k-n+1}
   \left[r^k\Psi'(r^2)^{(1+\eta)k}+
   \Psi'(r^2)^{k}\right]\notag\\
   &\simeq C_\Sigma(r,a,n-2+\frac{k}{2})
   e^{-\frac{p}{2}\Psi(r^2)}\Psi'(r^2)^{(n-1)(k-1)}
   \Phi'(r^2)^k,
  \end{align*}
  as $r\to \infty.$ The proof is finished.
\end{proof}
\begin{lemma}\label{qt5}
    Let $0<p<\infty$ and $0\leq k<\infty$.\\
\textup{(A)} There holds
  \begin{align}\label{qq5}
    \int_{\mathbb{C}^n}|K(z,w)|^pe^{-\frac{p}{2}\Psi (|w|^2)}[\Phi '(|w|^2)\Psi '(|w|^2)^{n-1}]^kdV(w)\lesssim e^{\frac{p}{2}\Psi(|z|^2)}[\Phi '(|z|^2)\Psi '(|z|^2)^{n-1}]^{p-1+k}.
\end{align}
\textup{(B)} There holds
\begin{align}\label{qq6}
    \int_{D(z,R)^c}|K(z,w)|^pe^{-\frac{p}{2}\Psi (|w|^2)}[\Phi '(|w|^2)\Psi '(|w|^2)^{n-1}]^kdV(w)\lesssim o(1)e^{\frac{p}{2}\Psi(|z|^2)}[\Phi '(|z|^2)\Psi '(|z|^2)^{n-1}]^{p-1+k},
\end{align}
as $R\to \infty,$  uniformly in $z\in \mathbb{C}^n$.
\end{lemma}
\begin{proof}
  Without loss of generality, we may assume that $z=(x,0,\cdots,0)$ with $x>0$. Write $w=(w_1,\xi)$ with $\xi\in \mathbb{C}^{n-1}$ and $w_1=re^{i\theta}$ when $n>1$. Define $h:\mathbb{C}\to \mathbb{R}$ as
  \begin{equation*}
  h(te^{i\theta})=\left\{
    \begin{aligned}
      \Phi '(t),\qquad\qquad\qquad\quad &|\theta|\leq \theta_0(t),\\
      t^{-3/2}[\Phi '(t)]^{-1/2}|\theta|^{-3},\quad&|\theta|> \theta_0(t),
    \end{aligned}
  \right.
\end{equation*}
where $\theta_0(t)=[t\Phi'(t)]^{-\frac{1}{2}}$. By Lemma \ref{2p2}, we have
\begin{align}\label{q21}
  H(z,w):&=|K(z,w)|^pe^{-\frac{p}{2}\Psi (|w|^2)}[\Phi '(|w|^2)\Psi '(|w|^2)^{n-1}]^k\notag\\
  &\lesssim
  [\Psi '(xr)]^{p(n-1)}e^{p\Psi(xr)}h(xre^{i\theta})^pe^{-\frac{p}{2}\Psi (r^2+|\xi|^2)}[\Phi '(r^2+|\xi|^2)\Psi '(r^2+|\xi|^2)^{n-1}]^k.
\end{align}
We integrate the left side of the last inequality over $\mathbb{C}^n=\mathbb{C}\times \mathbb{C}^{n-1}$ by Fubini's theorem and then estimate the integral with respect to the vector $w_1$ over $\mathbb{C}$ by considering polar coordinates $\theta$ and $r$. Notice that
\begin{align}\label{exq21}
\int_{-\pi}^{\pi}h(xre^{i\theta})^pd\theta&=\int_{-\theta_0(xr)}^{\theta_0(xr)}[\Phi '(xr)]^pd\theta+2\int_{\theta_0(xr)}^{\pi}(xr)^{-\frac{3p}{2}}[\Phi '(xr)]^{-\frac{p}{2}}\theta^{-3p}d\theta.
\end{align}
We deal with the case $\frac{1}{3}<p<\infty$ first. By elementary calculations, we have
\begin{align}\label{q22}
\int_{-\pi}^{\pi}h(xre^{i\theta})^pd\theta\lesssim
(xr)^{-\frac{1}{2}}[\Phi '(xr)]^{\frac{2p-1}{2}}\textup{ for }\frac{1}{3}<p<\infty.
\end{align}
By Lemma \ref{qt4}, \eqref{q21} and \eqref{q22}, we have
\begin{align}\label{q23}
\int_{\mathbb{C}^n}H(z,w)\,dV(w)\lesssim \int_{0}^{\infty}L(r,0)[\Psi '(xr)]^{p(n-1)}e^{p\Psi(xr)}(xr)^{-\frac{1}{2}}[\Phi '(xr)]^{\frac{2p-1}{2}}r\,dr,
\end{align}
where $L(r,0)$ appears in Lemma \ref{qt4}.
It suffices to show that
\begin{align}\label{q24}
\sup_{x>0}\int_{0}^{\infty}L(r,0)[\Psi '(xr)]^{p(n-1)}e^{p\Psi(xr)}(xr)^{-\frac{1}{2}}[\Phi '(xr)]^{\frac{2p-1}{2}}e^{-\frac{p}{2}\Psi(x^2)}[\Psi '(x^2)^{n-1}\Phi'(x^2)]^{-p+1-k}r\,dr<\infty,
\end{align}
or equivalently,
\begin{align}\label{q25}
\sup_{x>0}\int_{0}^{\infty}S_x(r)e^{-Q_x(r)}r
\,dr<\infty,
\end{align}
where
\begin{align*}
  S_x(r)&=\Phi'(r^2)^k\Psi'(r^2)^{(n-1)(k-1)}
\Psi'(xr)^{p(n-1)}(xr)^{-\frac{1}{2}}\Phi'(xr)^{(p-1+k)-k+\frac{1}{2}}\Psi'(x^2)^{(n-1)(1-p-k)}\Phi'(x^2)^{1-p-k}\\
&= \left[\frac{\Phi'(r^2)}{\Phi'(xr)}\right]^{k}
\left[\frac{\Phi'(xr)}{\Phi'(x^2)}\right]^{p-1+k}
\left[\frac{\Phi'(xr)}{xr}\right]^{\frac{1}{2}}
\left[\frac{\Psi'(xr)}{\Psi'(r^2)}\right]^{(n-1)(1-k)}
\left[\frac{\Psi'(xr)}{\Psi'(x^2)}\right]^{(p-1+k)(n-1)},
\end{align*}
and
$$Q_x(r)=\frac{p}{2}\Psi(r^2)-p\Psi(xr)+\frac{p}{2}\Psi(x^2).$$
By \cite[Lemma 3.6]{SY} and a straightforward argument, we find that $S_x(r)e^{-Q_x(r)}$ satisfies conditions (I), (II) and (III) of \cite[p.177]{SY}(with $x=t$, $Q_x=g_t$, $x_0=x$, $\tau=\Phi'(x^2)^{-\alpha}$, $c=p\Phi'(x^2)$ and $m=0$). Hence,
% and $Q_x(r)$ satisfy the following conditions:
%
%(I) $Q_x(r)$ attains its minimum at the point $r_0=x\to \infty$ with $Q_x ''(r)=(1+o(1))\Phi '(x^2)$ for $|r-x|\leq [\Phi '(x^2)]^{-\alpha}$(where  $\max\{\eta,0\}<\alpha<1/2$), and $[\Phi '(x^2)]^{\alpha}=o(\Phi '(x^2))$ when $x\to \infty.$
%
%(II) For $|r-x|\leq [\Phi '(x^2)]^{-\alpha}$, there holds
%\begin{align}\label{j7}
%  S_x(r)r\leq C|r-x|^N,
%\end{align}
%where $N\geq 0$ is an integer.
%
%(III) When $|r-x|\geq [\Phi '(x^2)]^{-\alpha}$ and $|r-x|$ grows, the function $S_x(r)e^{-Q_x(r)}$ decays so fast that
%\begin{align}\label{q27}
%  \int_{0}^{\infty}S_x(r)e^{-Q_x(r)}rdr=(1+o(1))\int_{|r-x|\leq [\Phi '(x^2)]^{-\alpha}}S_x(r)e^{-Q_x(r)}rdr.
%\end{align}
%By \cite[p.177]{SY}, we deduce that
$$\sup_{x>0}\int_{0}^{\infty}S_x(r)e^{-Q_x(r)}r\,dr\leq C.$$
Therefore, assertion (A) holds. It remains to prove (B). If $\Psi ''(x)\equiv 0$, then the result follows immediately from the theory of the classical Fock space. We suppose $\Psi''(0)>0$. If $w \in D(z,R)^c$, then $|x-re^{i\theta}|>R[\Phi '(x^2)]^{-\frac{1}{2}}$ or $|\xi|>R[\Psi '(x^2)]^{-\frac{1}{2}}.$ Assume that $x$ is sufficiently large. We assert that if $|x-re^{i\theta}|>R[\Phi '(x^2)]^{-\frac{1}{2}}$ and $|x-r|\leq\min\{\frac{R}{2}\Phi '(x^2)^{-\frac{1}{2}},\Phi'(x^2)^{-\alpha}\}$ where $\max\{\eta,0\}<\alpha<\frac{1}{2}$, then
\begin{align}\label{j1}
  |\theta|\geq \frac{R}{8}\theta_0(xr).
\end{align}
In fact, by the triangle inequality, we have
\begin{align}\label{j2}
  |x-xe^{i\theta}|\geq |x-re^{i\theta}|-|x-r|>\frac{R}{2}\Phi '(x^2)^{-\frac{1}{2}}.
\end{align}
It follows that
\begin{align}\label{j3}
|\theta|\geq |1-e^{i\theta}|>\frac{R}{2x}\Phi '(x^2)^{-\frac{1}{2}}.
\end{align}
Notice that
\begin{align}\label{j4}
|xr-x^2|\leq x\Phi '(x^2)^{-\alpha},
\end{align}
which, together with Lemma 3.2 of \cite{SY} implies that
\begin{align}\label{j5}
\Phi '(x^2)^{-\frac{1}{2}}\geq \frac{1}{2}\Phi '(xr)^{-\frac{1}{2}}.
\end{align}
It is easy to see that
\begin{align}\label{k1}
\frac{1}{r+\Phi'(x^2)^{-\alpha}}\geq \frac{1-\Phi'(x^2)^{-\alpha}}{r}&\iff r\geq [1-\Phi'(x^2)^{-\alpha}][r+\Phi'(x^2)^{-\alpha}]\notag\\
&\iff 0\geq -\Phi'(x^2)^{-\alpha}[r-1+\Phi'(x^2)^{-\alpha}]\notag\\
&\iff r\geq 1-\Phi'(x^2)^{-\alpha}.
\end{align}
Notice that
\begin{align}\label{k2}
  r\geq x-\Phi'(x^2)^{-\alpha}\geq 1-\Phi'(x^2)^{-\alpha}\textup{ if } |x-r|\leq \Phi'(x^2)^{-\alpha} \textup{ and }x\geq 1.
\end{align}
We also have
\begin{align}\label{k3}
  \Phi'(x^2)=\Psi'(x^2)+x^2\Psi''(x^2) \to \infty\textup{ as } x\to \infty.
\end{align}
It follows from \eqref{k1}, \eqref{k2} and \eqref{k3} that
\begin{align}\label{j6}
\frac{1}{x}\geq \frac{1}{r+\Phi'(x^2)^{-\alpha}}\geq  \frac{1-\Phi'(x^2)^{-\alpha}}{r}\geq \frac{1}{4r}\textup{ for } |x-r|\leq \Phi'(x^2)^{-\alpha}.
\end{align}
By this, \eqref{j3} and \eqref{j5}, we infer that inequality \eqref{j1} holds. Set
$$E_1(x)=\{w=(w_1,\xi)\in \mathbb{C}^n: w_1=re^{i\theta}\in\mathbb{C}, |x-r|\geq \Phi '(x^2)^{-\alpha}\},$$
$$E_2(x,R)=\{w=(w_1,\xi)\in \mathbb{C}^n: w_1=re^{i\theta}\in\mathbb{C}, |\theta|\geq\frac{R}{8}\theta_0(xr)\},$$
$$E_3(x,R)=\{w=(w_1,\xi)\in \mathbb{C}^n: w_1=re^{i\theta}\in\mathbb{C}, \frac{R}{2}\Phi'(x^2)^{-\frac{1}{2}}\leq |x-r|\leq \Phi '(x^2)^{-\alpha}\},$$
%$$E_4(x,R)=\{w=(w_1,\xi)\in \mathbb{C}^n: w_1=re^{i\theta}\in\mathbb{C}, [\Phi '(x^2)]^{-\alpha}\leq |x-r|\leq \frac{R}{2}\Phi'(x^2)^{-\frac{1}{2}}\},$$
and
$$E_4(x,R)=\{w=(w_1,\xi)\in \mathbb{C}^n: w_1\in\mathbb{C},|\xi|>R\Psi'(x^2)^{-\frac{1}{2}}\}.$$
By \eqref{j1}, we see that
\begin{align}\label{j6a}
  D(z,R)^c\subseteq E_1(x)\cup E_2(x,R)\cup E_3(x,R)\cup E_4(x,R)\textup{ for }R\leq 2\Phi'(x^2)^{\frac{1}{2}-\alpha},
\end{align}
and
\begin{align}\label{j6b}
D(z,R)^c\subseteq E_1(x)\cup E_2(x,R)\cup E_4(x,R)\textup{ for }R > 2\Phi'(x^2)^{\frac{1}{2}-\alpha}.
\end{align}
From condition (III) of \cite[p.177]{SY}, we know that, for every $\varepsilon>0$, there is a constant $M>0$ such that for $R,x\geq M$,
\begin{align}\label{q29}
    \frac{\int_{E_1(x)}H(z,w)dV(w)}{e^{\frac{p}{2}\Psi(|z|^2)}[\Phi '(|z|^2)\Psi '(|z|^2)^{n-1}]^{p-1+k}} <\frac{\varepsilon}{4}.
\end{align}
For $w\in E_2(x,R),$ there holds
\begin{align}\label{exq29}
\int_{\frac{R}{8}\theta_0(xr)}^{\pi}h(xre^{i\theta})^pd\theta\lesssim
 R^{-3p+1}(xr)^{-\frac{1}{2}}\Phi '(xr)^{p-\frac{1}{2}}.
\end{align}
By checking the proof of (A), we conclude that for $x\geq M$, there is a constant $C_2>0$ such that
\begin{align*}
    \frac{\int_{E_2(x,R)}H(z,w)dV(w)}{e^{\frac{p}{2}\Psi(|z|^2)}[\Phi '(|z|^2)\Psi '(|z|^2)^{n-1}]^{p-1+k}} \leq C_2R^{-3p+1},
\end{align*}
or say that there is a constant $$R_2=\left(\frac{\varepsilon}{4C_2}\right)^{\frac{1}{-3p+1}}+1$$
such that for $x\geq M$ and $R\geq R_2$,
\begin{align}\label{q30}
    \frac{\int_{E_2(x,R)}H(z,w)dV(w)}{e^{\frac{p}{2}\Psi(|z|^2)}[\Phi '(|z|^2)\Psi '(|z|^2)^{n-1}]^{p-1+k}} <\frac{\varepsilon}{4}.
\end{align}
By Lemma 3.6 of \cite{SY} and fundamental calculus, we know that, for $r$ with $|x-r|\leq \Phi'(x^2)^{-\alpha}$
\begin{align}\label{j8}
  Q_x(r)=\frac{1+o(1)}{2}p\Phi'(x^2)(r-x)^2\textup{ as }x\to \infty.
\end{align}
This, together with condition (II) of \cite[p.177]{SY}, further implies that there is an integer $N\geq 0$ such that
\begin{align*}
    \frac{\int_{E_3(x,R)}H(z,w)dV(w)}{e^{\frac{p}{2}\Psi(|z|^2)}[\Phi '(|z|^2)\Psi '(|z|^2)^{n-1}]^{p-1+k}} &\lesssim \int\limits_{\frac{R}{2}\Phi'(x^2)^{-\frac{1}{2}}\leq |x-r|\leq \Phi '(x^2)^{-\alpha}}|r-x|^Ne^{-\frac{p}{3}\Phi'(x^2)(r-x)^2}\,dr\notag\\
    &\simeq \int_{\frac{R}{2}\Phi'(x^2)^{-\frac{1}{2}}}^{\Phi '(x^2)^{-\alpha}}t^Ne^{-\frac{p}{3}\Phi'(x^2)t^2}\,dt\notag\\
    &=\int_{\frac{R}{2}}^{\Phi '(x^2)^{\frac{1}{2}-\alpha}}u^N\Phi'(x^2)^{-\frac{N+1}{2}}e^{-\frac{p}{3}u^2}\,du\notag\\
    &\lesssim \int_{\frac{R^2}{4}}^{\infty} r^{\frac{N-1}{2}}e^{-\frac{p}{3}r}\,dr\to 0 \textup{ as }R\to\infty.
\end{align*}
Notice that $E_3(x,R)=\emptyset$ when $R > 2\Phi'(x^2)^{\frac{1}{2}-\alpha}$. There is a constant $R_3>0$
such that for $x\geq M$ and $R\geq R_3$,
\begin{align}\label{q30a}
    \frac{\int_{E_3(x,R)}H(z,w)dV(w)}{e^{\frac{p}{2}\Psi(|z|^2)}[\Phi '(|z|^2)\Psi '(|z|^2)^{n-1}]^{p-1+k}} <\frac{\varepsilon}{4}.
\end{align}
When estimating the integral over $E_4(x,R)$, we will use that
\begin{align}\label{k4}
  \Psi'(r^2)\geq \frac{1}{2}\Psi'(x^2)\textup{ for } x\geq M \textup{ and }|r-x|\leq \Phi'(x^2)^{-\alpha}.
\end{align}
Indeed, it is obvious that, if $r\geq x$, then $\Psi'(r^2)\geq \Psi'(x^2)\geq \frac{1}{2}\Psi'(x^2)$. When $r<x$, by \eqref{cdt3}, we have
\begin{align}\label{k5}
  \Psi'(x^2)-\Psi'(r^2)\leq \Psi''(x^2)(x^2-r^2)\leq Cx^{-1}\Psi'(x^2)^{1+\eta}(x^2-r^2),
\end{align}
or equivalently,
\begin{align}\label{k6}
  \frac{\Psi'(r^2)}{\Psi'(x^2)}\geq 1-Cx^{-1}\Psi'(x^2)^{\eta}(x^2-r^2).
\end{align}
Since $|x-r|\leq \Phi'(x^2)^{-\alpha}$, we have
\begin{align}\label{k6n}
  \left|\frac{r}{x}-1\right|\leq x^{-1}\Phi'(x^2)^{-\alpha}\to 0 \textup{ as }x\to\infty,
\end{align}
which is equivalent to
\begin{align}\label{k7}
 r=(1+o(1))x,  \textup{ as }x\to\infty.
\end{align}
It follows that
\begin{align}\label{k8}
 x^2-r^2=(x-r)(x+r)\leq \Phi'(x^2)^{-\alpha}(2+o(1))x \textup{ as }x\to\infty.
\end{align}
By \eqref{k6} and \eqref{k8}, we see that
\begin{align}\label{k9}
 \frac{\Psi'(r^2)}{\Psi'(x^2)}\geq 1-\frac{C\Psi'(x^2)^{\eta}}{\Phi'(x^2)^{\alpha}}(2+o(1)).
\end{align}
If $\Psi'(x^2)\to \infty $ as $x\to \infty$, then
\begin{align}\label{k10}
 \frac{\Psi'(x^2)^{\eta}}{\Phi'(x^2)^{\alpha}}\leq \Psi'(x^2)^{\eta-\alpha}\to 0 \textup{ as } x\to \infty.
\end{align}
If $\Psi'(x^2)$ is bounded, then we have
\begin{align}\label{k11}
 \frac{\Psi'(x^2)^{\eta}}{\Phi'(x^2)^{\alpha}}&= \frac{\Psi'(x^2)^{\eta}}{[\Psi'(x^2)+x^2\Psi''(x^2)]^{\alpha}}\leq  \frac{\Psi'(x^2)^{\eta}}{[\Psi'(x^2)+Cx\Psi'(x^2)^{1+\eta}]^{\alpha}}\lesssim  x^{-\alpha}\Psi'(x^2)^{-(\alpha-\eta)(1+\eta)-\eta^2}\to 0,
\end{align}
as $x\to \infty$. From \eqref{k9}, \eqref{k10} and \eqref{k11}, we see that inequality \eqref{k4} holds. Therefore, for $x\geq M$,
\begin{align}\label{q31}
  C(r,R[\Psi '(x^2)]^{-\frac{1}{2}},s)&\lesssim
  [R\Psi'(x^2)^{\frac{1}{2}}]^se^{-\frac{p}{8}R\Psi'(x^2)^{1/2}}+
  \int_{\frac{p}{4}\Psi'(x^2)^{\frac{1}{2}}R}^{\infty}u^{s}e^{-u}\,du\notag\\
  &\lesssim
  [R\Psi'(0)^{\frac{1}{2}}]^se^{-\frac{p}{8}R\Psi'(0)^{1/2}}+
  \int_{\frac{p}{4}\Psi'(0)^{\frac{1}{2}}R}^{\infty}u^{s}e^{-u}\,du\notag\\
  &=o(1),\textup{ for all } s\geq 0\textup{ as } R\to \infty.
\end{align}
%and
%\begin{align}\label{q32}
%  C(r,R[\Psi '(x^2)]^{-\frac{1}{2}},0)\leq C(r,R[\Psi '(x^2)]^{-\frac{1}{2}},1)=o(1),\textup{ as } R\to \infty,
%\end{align}
On the other hand, we have
\begin{align}\label{q33}
\int_{E_4(x,R)}H(z,w)\,dV(w)\lesssim \int_{0}^{\infty}L(r,R[\Psi '(x^2)]^{-\frac{1}{2}})[\Psi '(xr)]^{p(n-1)}e^{p\Psi(xr)}(xr)^{-\frac{1}{2}}[\Phi '(xr)]^{\frac{2p-1}{2}}r\,dr,
\end{align}
which is analogous to \eqref{q23}. Using \eqref{q31} and condition (III) of \cite[p.177]{SY}  and considering the additional coefficient $C_\Sigma(r,R[\Psi '(x^2)]^{-\frac{1}{2}},n-2+\frac{k}{2})$ in $L(r,R[\Psi '(x^2)]^{-\frac{1}{2}})$, we see that for $x\geq M$,
\begin{align*}
    \frac{\int_{E_4(x,R)}H(z,w)dV(w)}{e^{\frac{p}{2}\Psi(|z|^2)}[\Phi '(|z|^2)\Psi '(|z|^2)^{n-1}]^{p-1+k}}
    \leq Co(1)\int_{|r-x|\leq \Phi'(x^2)^{-\alpha}}S_x(r)e^{-Q_x(r)}rdr\leq Co(1)\textup{ as }R\to \infty,
\end{align*}
which implies that there is a constant $R_4$ such that for $x\geq M$ and $R\geq R_4$,
\begin{align}\label{q34}
    \frac{\int_{E_4(x,R)}H(z,w)dV(w)}{e^{\frac{p}{2}\Psi(|z|^2)}[\Phi '(|z|^2)\Psi '(|z|^2)^{n-1}]^{p-1+k}}< \frac{\varepsilon}{4}.
\end{align}
From \eqref{j6a},\eqref{q29},\eqref{q30},\eqref{q30a} and \eqref{q34}, we know that for $x\geq M$ and $\max\{M,R_2,R_3,R_4\}\leq R \leq 2\Phi'(x^2)^{\frac{1}{2}-\alpha}$,
\begin{align}\label{q37}
    \frac{\int_{D(z,R)^c}H(z,w)dV(w)}{e^{\frac{p}{2}\Psi(|z|^2)}[\Phi '(|z|^2)\Psi '(|z|^2)^{n-1}]^{p-1+k}}< \varepsilon.
\end{align}
By \eqref{j6b},\eqref{q29},\eqref{q30} and \eqref{q34}, we know that for $x\geq M$ and $R> \max\{M,R_2,R_4,2\Phi'(x^2)^{\frac{1}{2}-\alpha}\}$, inequality \eqref{q37} holds. Therefore, for $x\geq M$ and $R>\max\{M,R_2,R_3,R_4\}$, inequality \eqref{q37} holds.
Now, we consider the situation of $0\leq x\leq M.$ Notice that for $w \in D(z,R)^c$, we have $|x-re^{i\theta}|>R[\Phi '(x^2)]^{-\frac{1}{2}}$ or $|\xi|>R[\Psi '(x^2)]^{-\frac{1}{2}}.$ Thus,
$$r>R\Phi'(M^2)^{-\frac{1}{2}}-M,\textup{ or } |\xi|>R\Psi'(M^2)^{-\frac{1}{2}}.$$
Set
$$E_5(x,t)=\{w=(w_1,\xi)\in \mathbb{C}^n: w_1=re^{i\theta}\in\mathbb{C}, r>t\},$$
and
$$E_6(x,t)=\{w=(w_1,\xi)\in \mathbb{C}^n:\xi\in\mathbb{C}^{n-1}, |\xi|>t\}.$$
Then $D(z,R)^c\subseteq E_5(x,R\Phi'(M^2)^{-\frac{1}{2}}-M)\cup E_6(x,R\Psi'(M^2)^{-\frac{1}{2}}).$
It is easy to see that there is a constant $R_5$ such that for $t>R_5$ and $0\leq x\leq M$,
\begin{align}\label{q35}
    \frac{\int_{E_5(x,t)}H(z,w)dV(w)}{e^{\frac{p}{2}\Psi(|z|^2)}[\Phi '(|z|^2)\Psi '(|z|^2)^{n-1}]^{p-1+k}}< \frac{\varepsilon}{2}.
\end{align}
and
\begin{align}\label{q36}
    \frac{\int_{E_6(x,t)}H(z,w)dV(w)}{e^{\frac{p}{2}\Psi(|z|^2)}[\Phi '(|z|^2)\Psi '(|z|^2)^{n-1}]^{p-1+k}}<\frac{\varepsilon}{2}.
\end{align}
Hence, for $0\leq x\leq M$ and $R>\max\{(R_5+M)\Phi'(M^2)^{\frac{1}{2}},R_5\Psi'(M^2)^{\frac{1}{2}}\}$, inequality \eqref{q37} holds. By taking
$$R_0=\max\left\{M,R_2,R_3,R_4,(R_5+M)\Phi'(M^2)^{\frac{1}{2}},R_5\Psi'(M^2)^{\frac{1}{2}}\right\},$$
which is clearly independent of $x$, % and combining \eqref{q29},\eqref{q30},\eqref{q34},\eqref{q35} and \eqref{q36},
we see that for $0\leq x<\infty$ and $R>R_0$, inequality \eqref{q37} holds.
Thus, we have proved this lemma for the case $\frac{1}{3}<p<\infty.$ As for the case $0<p\leq \frac{1}{3}$, we begin at estimating the second part of \eqref{exq21} as
\begin{align*}
\int_{\theta_0(xr)}^{\pi}(xr)^{-\frac{3p}{2}}[\Phi '(xr)]^{-\frac{p}{2}}\theta^{-3p}d\theta
\leq \pi\int_{\theta_0(xr)}^{\pi}(xr)^{-\frac{3p}{2}}[\Phi '(xr)]^{-\frac{p}{2}}\theta^{-3p-1}d\theta\lesssim \Phi'(xr)^{p}.
\end{align*}
It follows that
\begin{align}\label{exq22}
\int_{-\pi}^{\pi}h(xre^{i\theta})^pd\theta\lesssim
(xr)^{-\frac{1}{2}}[\Phi '(xr)]^{\frac{2p-1}{2}}+\Phi'(xr)^{p}\lesssim \Phi'(xr)^{p}\textup{ for }0<p\leq\frac{1}{3}.
\end{align}
With minor changes to the former case, we see that this lemma holds for $0<p\leq \frac{1}{3}$ while we replace $S_x(r)$ and \eqref{exq29} with
\begin{align*}
S_x(r)=\Phi'(xr)\left[\frac{\Phi'(r^2)}{\Phi'(xr)}\right]^{k}
\left[\frac{\Phi'(xr)}{\Phi'(x^2)}\right]^{p-1+k}
\left[\frac{\Psi'(xr)}{\Psi'(r^2)}\right]^{(n-1)(1-k)}
\left[\frac{\Psi'(xr)}{\Psi'(x^2)}\right]^{(p-1+k)(n-1)},
\end{align*}
and
\begin{align*}
\int_{R\theta_0(xr)}^{\pi}h(xre^{i\theta})^pd\theta\lesssim
 R^{-3p}\Phi '(xr)^{p}.
\end{align*}
To avoid repetitions, we omit the details here. The proof is completed.
\end{proof}
\section{Schatten Class Toeplitz Operators}\label{s5}
Recall that $d\nu_\Psi(z)=K(z,z)\,d\mu_\Psi(z)\simeq \Phi '(|z|^2)\Psi '(|z|^2)^{n-1}dV(z)\simeq |B(z,r)|^{-1}\,dV(z).$
\begin{lemma}\label{4p4}
		Let $0< p<\infty$. The following conditions are equivalent.\\
		\textup{(A)} The function $\widetilde{\mu}(z)$ is in $L^p(\mathbb{C}^n,d\nu_\Psi)$.\\
		\textup{(B)} The function $\widehat{\mu}_r(z)$ is in $L^p(\mathbb{C}^n,d\nu_\Psi)$ for some(or any) $0<r<\infty$.\\
		\textup{(C)} The sequence $\{\widehat{\mu}_r(z_k)\}_{k=1}^\infty$ is in $l^p$ for some(or any) $\Psi$-lattice $\{z_k\}_{k=1}^\infty$. Furthermore,
\begin{align*}
  \|\widetilde{\mu}\|_{L^p(\mathbb{C}^n,d\nu_\Psi)} \simeq \|\widehat{\mu}_r\|_{L^p(\mathbb{C}^n,d\nu_\Psi)}\simeq \|\{\widehat{\mu}_r(z_k)\}_{k=1}^\infty\|_{l^p}.
\end{align*}
	\end{lemma}
	\begin{proof}
See \cite[Lemma 16]{WTH} for $1\leq p<\infty$. For $0<p<1$, the proof is similar to Lemma 2.4 of \cite{IVW} and we omit the details here.
	\end{proof}
\begin{lemma}\label{qt6}
  Let $\left\{z_{k}\right\}_{k=1}^\infty$ be a $\Psi$-lattice with covering radius $r>0$ and $0<R<\infty$. For each $z \in\mathbb{C}^n$, the set $B(z,R)$ contains at most $N$ points of the lattice, where $N$ depends on $R$ and $r$ but not on the point $z$.
\end{lemma}
\begin{proof}
  By an analogous argument as in Lemma \ref{3p3}, we can show that
  $$N:=\textup{card}\{k:B(z,R)\cap B(z_k,r)\neq \emptyset\}<\infty.$$
  Moreover, $N$ is independent of $z$. Since $B(z,R)\cap B(z_k,r)\neq \emptyset$ for $z_k\in B(z,R)$, we have
  $$\textup{card}\{k:z_k\in B(z,R)\}\leq N.$$
  The proof is finished.
\end{proof}
With the help of the last lemma, we obtain the following consequence which is almost identical to Lemma 4.5 of \cite{APP}.
\begin{lemma}\label{qt7}
  Let $\left\{z_{k}\right\}_{k=1}^\infty$ be a $\Psi$-lattice and $0<R<\infty$. Then $\left\{z_{k}\right\}_{k=1}^\infty$ can be partitioned into $M$ subsequences such that, if $z_j$ and $z_k$ are different points in the same subsequence, then $\varrho(z_j,z_k)\geq R.$
\end{lemma}
\begin{lemma}\label{qt8}
  Let $0<R<\infty,$ and $0<p<2.$ Let $r>0$ be fixed. Suppose $\left\{z_{k}\right\}_{k=1}^\infty$ is a sequence in $\mathbb{C}^n$ with the property that $\varrho(z_j,z_k)\geq R$ for every $j\neq k.$ Then we have
  \begin{align*}
    \sum_{\{j:j\neq k\}}\sup_{w\in B(z_k,r)}|k_{z_j}(w)|^pe^{-\frac{p}{2}\Psi(|w|^2)}=o(1)|B(z_k,r)|^{-\frac{p}{2}}\textup{ as } R\to\infty,
  \end{align*}
for every $k$.
\end{lemma}
\begin{proof}
  By Lemma \ref{2p3}, \eqref{2e3} and \eqref{ex2e6}, we obtain
  \begin{align}\label{t1}
    \sup_{w\in B(z_k,r)}|k_{z_j}(w)|^pe^{-\frac{p}{2}\Psi(|w|^2)}
    &\lesssim \frac{1}{|B(z_k,2r)|}\int_{B(z_k,2r)}|k_{z_j}(\zeta)|^pe^{-\frac{p}{2}\Psi(|\zeta|^2)}\,dV(\zeta)\notag\\
    &\simeq \frac{1}{|B(z_k,2r)|}\int_{B(z_k,2r)}|K(z_j,\zeta)|^pe^{-\frac{p}{2}\Psi(|z_j|^2)}|B(z_j,r)|^{\frac{p}{2}}
    e^{-\frac{p}{2}\Psi(|\zeta|^2)}\,dV(\zeta).
  \end{align}
  By Lemma \ref{2p3} again, we get
  \begin{align}\label{t2}
    |K(z_j,\zeta)|^pe^{-\frac{p}{2}\Psi(|z_j|^2)}|B(z_j,r)|^{\frac{p}{2}}
    &\lesssim |B(z_j,r)|^{\frac{p}{2}-1}\int_{B(z_j,r)}|K(\xi,\zeta)|^pe^{-\frac{p}{2}\Psi(|\xi|^2)}\,dV(\xi)\notag\\
    &\simeq \int_{B(z_j,r)}|K(\xi,\zeta)|^pe^{-\frac{p}{2}\Psi(|\xi|^2)}|B(\xi,r)|^{\frac{p}{2}-1}\,dV(\xi).
  \end{align}
  It is easy to see that for $R>12r$ and $\zeta \in B(z_k,2r)$,
  \begin{align}\label{t3}
    \bigcup_{\{j:j\neq k\}}B(z_j,r)\subseteq B(z_k,R/2)^c\subseteq B(\zeta,R/3)^c.
  \end{align}
  By \eqref{t1}, \eqref{t2}, \eqref{t3} and Lemma \ref{qt5}, we have
  \begin{align*}
    &\sum_{\{j:j\neq k\}}\sup_{w\in B(z_k,r)}|k_{z_j}(w)|^pe^{-\frac{p}{2}\Psi(|w|^2)}\\
    &\lesssim |B(z_k,r)|^{-1}\int_{B(z_k,2r)}\left(\sum_{\{j:j\neq k\}}\int_{B(z_j,r)}|K(\xi,\zeta)|^pe^{-\frac{p}{2}\Psi(|\xi|^2)}|B(\xi,r)|^{\frac{p}{2}-1}\,dV(\xi)\right)
    e^{-\frac{p}{2}\Psi(|\zeta|^2)}\,dV(\zeta)\\
    &\lesssim |B(z_k,r)|^{-1}\int_{B(z_k,2r)}e^{-\frac{p}{2}\Psi(|\zeta|^2)}\,dV(\zeta)\int_{B(\zeta,R/3)^c}|K(\xi,\zeta)|^pe^{-\frac{p}{2}\Psi(|\xi|^2)}|B(\xi,r)|^{\frac{p}{2}-1}\,dV(\xi)
    \\
    &\lesssim o(1)|B(z_k,r)|^{-1}\int_{B(z_k,2r)}|B(\zeta,r)|^{-\frac{p}{2}}\,dV(\zeta)\\
    &=o(1)|B(z_k,r)|^{-\frac{p}{2}},
  \end{align*}
  as $R\to\infty$. The proof is finished.
\end{proof}
For an operator $T$ on $F^2_\Psi$, the Berizin transform of $T$ is the function $\widetilde{T}$ on $\mathbb{C}^n$ defined as
$$\widetilde{T}(z):=\langle Tk_z,k_z\rangle \textup{ for } z\in \mathbb{C}^n.$$
\begin{lemma}\label{qt9}
  Let $T$ be a positive operator on $F^2_\Psi$. \\
  \textup{(A)} If $0< p\leq 1$ and $\widetilde{T}\in L^p(\mathbb{C}^n,d\nu_\Psi)$, then $T\in S_p.$ Furthermore,
  $\|T\|_{S_p}\leq \|\widetilde{T}\|_{L^p(\mathbb{C}^n,d\nu_\Psi)}.$\\
  \textup{(B)} If $1\leq p<\infty$ and $T\in S_p$, then $\widetilde{T}\in L^p(\mathbb{C}^n,d\nu_\Psi)$.Furthermore,
  $\|\widetilde{T}\|_{L^p(\mathbb{C}^n,d\nu_\Psi)}\leq \|T\|_{S_p}.$
\end{lemma}
\begin{proof}
The proof is similar to Lemma 2.5 of \cite{IVW}.
%  Clearly,
%  \begin{align}\label{ext4}
 %   \|T\|_{S_p}^p=\|T^p\|_{S_1}.
%  \end{align}
%  By \cite[Lemma 8.4]{SY} and Lemma \ref{2p2}, we know that
%  \begin{align}\label{ext5}
%    \|T^p\|_{S_1}=\int_{\mathbb{C}^n}\widetilde{T^p}(z)K(z,z)e^{-\Psi(|z|^2)}\,dV(z)
%    =\int_{\mathbb{C}^n}\widetilde{T^p}(z)\,d\nu_\Psi(z).
%  \end{align}
%  From \cite[Proposition 1.31]{Zhu}, we have
%  \begin{align}\label{ext6}
%    \widetilde{T^p}(z)=\langle T^pk_z,k_z\rangle\leq \langle Tk_z,k_z\rangle^p=\widetilde{T}(z)^p, \textup{ for } 0<p\leq 1,
%  \end{align}
%  and
%  \begin{align}\label{ext7}
 %   \widetilde{T^p}(z)=\langle T^pk_z,k_z\rangle\geq \langle Tk_z,k_z\rangle^p=\widetilde{T}(z)^p, \textup{ for } p\geq 1.
%  \end{align}
%  \eqref{ext4}, \eqref{ext5} and \eqref{ext6} yield (A) while (B) comes from \eqref{ext4}, \eqref{ext5} and \eqref{ext7}.
\end{proof}

	\begin{proof}[Proof of Theorem \ref{4p5}]
	The equivalence (A)$\Leftrightarrow$(D) for $p\geq 1$ is proved in \cite{SY}. By Lemma \ref{4p4}, it remains to prove that (A)$\Rightarrow$(D) and (B)$\Rightarrow$(A) when $0<p<1$. By direct calculations, we have $\widetilde{T_\mu}=\widetilde{\mu}$. By this fact and Lemma \ref{qt9}, we see that (B)$\Rightarrow$(A) for $0<p<1$. Furthermore,
\begin{align}\label{ext8}
  \|T_\mu\|_{S_p}\leq \|\widetilde{\mu}\|_{L^p(\mathbb{C}^n,d\nu_\Psi)}.
\end{align}
Now, We are going to show (A)$\Rightarrow$(D) for $0<p<1$. Assume that $T_\mu$ is in $S_{p}$. Let $\left\{z_{k}\right\}_{k=1}^\infty$ be a $\Psi$-lattice on $\mathbb{C}^n$. We need to prove that the sequence $\left\{\widehat{\mu}_{r}\left(z_{k}\right)\right\}_{k=1}^\infty$ is in $l^{p}$. For this purpose, we fix a number $R$ and use Lemma \ref{qt7} to divide the lattice $\left\{z_{k}\right\}_{k=1}^\infty$ into $M$ subsequences such that $\varrho(z_{j},z_{k})\geq R$
whenever the two different points $z_{j}$ and $z_{k}$ are in the same subsequence. Let $\left\{z_{k}\right\}_{k=1}^\infty$ be such a subsequence and consider a measure $\nu$ defined by
$$d\nu=\left(\sum_{k=1}^{\infty}\chi_{B(z_k,r)}\right) d\mu.$$
Thus, we have $T_{\nu}\in S_{p}$ and $\left\|T_{\nu}\right\|_{S_{p}}
\leq\left\|T_{\mu}\right\|_{S_{p}}$.

Given an orthonormal basis
$\{e_{k}\}_{k=1}^\infty$ for $F_{\Psi}^{2}$, we define an operator $J$ on $F_{\Psi}^{2}$ by
\begin{align}\label{t5}
Jf=\sum_{k=1}^{\infty}\langle f,e_k\rangle k_{z_{k}}\textup{ for } f\in F_{\Psi}^{2}.
\end{align}%\sum_{n}\lambda_{n}e_{n}
Then $J$ is bounded on $F_{\Psi}^{2}$ by Lemma 8.3 of \cite{SY}. Hence the operator $T=J^{\ast}T_{\nu}J$ is in $S_{p}$ with
\begin{align}\label{t6}
\|T\|_{S_{p}} \lesssim \left\|T_{\mu}\right\|_{S_{p}}.
\end{align}
From the fact that
$$\langle Tf, g\rangle=\langle T_{\nu}Jf, Jg\rangle\textup{ for } f,g \in F_{\Psi}^{2},$$
we deduce that
$$
T f=\sum_{j,k=1}^{\infty}\left\langle T_{\nu} k_{z_{j}}, k_{z_{k}}\right\rangle \left\langle f, e_{j}\right\rangle e_{k}\textup{ for } f \in F_{\Psi}^{2}.
$$
We decompose the operator $T$ as $T=T_{1}+T_{2}$, where $T_{1}$ is the diagonal part of $T$ defined by
$$
T_{1} f=\sum_{k=1}^{\infty}\left\langle T_{\nu} k_{z_{k}}, k_{z_{k}}\right\rangle \left\langle f, e_{k}\right\rangle e_{k}\textup{ for } f \in F_{\Psi}^{2},
$$
and $T_{2}$ is the non-diagonal part defined by $T_{2}=T-T_{1}$. Then, we have
\begin{align}\label{t7}
\|T\|_{S_{p}}^{p} \geq\|T_{1}\|_{S_{p}}^{p}
-\|T_{2}\|_{S_{p}}^{p}.
\end{align}
By the fact that $T_{1}$ is positive, Lemma \ref{2p1}, Lemma \ref{2p2} and \eqref{2e3}, we have the following estimate
\begin{align}\label{t8}
\|T_{1}\|_{S_{p}}^{p}&=\sum_{k=1}^{\infty}\left\langle T_{\nu} k_{z_{k}}, k_{z_{k}}\right\rangle^{p}
=\sum_{k=1}^{\infty}\left(\int_{\mathbb{C}^n}\left|k_{z_{k}}(z)\right|^{2} e^{-\Psi(|z|^2)} d \nu(z)\right)^{p}\notag\\
&\gtrsim\sum_{k=1}^{\infty}\left(\int_{B(z_{k},r)}\left|k_{z_{k}}(z)\right|^{2} e^{-\Psi(|z|^2)}\,d\mu(z)\right)^{p}\notag\\
&\simeq \sum_{k=1}^{\infty}\left(\int_{B(z_{k},r)}|B(z,r)|^{-1}d\mu(z)\right)^{p}
\simeq\sum_{k=1}^{\infty}\widehat{\mu}_{r}(z_{k})^{p}.
\end{align}
On the other hand, by Proposition 1.29 of \cite{Zhu}, we have
\begin{align}\label{t9}
\|T_{2}\|_{S_{p}}^{p} &\leq \sum_{j=1}^{\infty} \sum_{k=1}^{\infty}|\left\langle T_{2} e_{j}, e_{k}\right\rangle |^p=\sum_{j=1}^{\infty}\sum_{\{k:k \neq j\}}|\langle T_{\nu} k_{z_{j}}, k_{z_{k}}\rangle |^p\notag\\
&\leq\sum_{j=1}^{\infty}\sum_{\{k:k \neq j\}}\left(\int_{\mathbb{C}^n}|k_{z_{j}}(w)k_{z_{k}}(w)| e^{-\Psi(|w|^2)} d \nu(w)\right)^{p}\notag\\
&\leq\sum_{j=1}^{\infty}\sum_{\{k:k \neq j\}}\left(\sum_{i=1}^{\infty}\int_{B(z_{i},r)}|k_{z_{j}}(w)k_{z_{k}}(w)| e^{-\Psi(|w|^2)} d \mu(w)\right)^{p}.
\end{align}
For $w\in B(z_i,r)$, by Lemma \ref{2p3}, we have
\begin{align}\label{t10}
\left|k_{z_{j}}(w)\right|e^{-\frac{1}{2}\Psi(|w|^2)}
&\lesssim\left[\frac{1}{|B(w,r)|}\int_{B(w,r)}
\left|k_{z_{j}}(u)\right|^{p} e^{-\frac{p}{2} \Psi(|u|^2)}dV(u)\right]^{1/p}\notag\\
&\lesssim\left[\frac{1}{|B(z_i,2r)|}\int_{B(z_i,2r)}
\left|k_{z_{j}}(u)\right|^{p} e^{-\frac{p}{2} \Psi(|u|^2)}dV(u)\right]^{1/p}\notag\\
&=|B(z_i,2r)|^{-\frac{1}{p}}S_{j}(z_i)^{\frac{1}{p}},
\end{align}
where
$$S_j(z)=\int_{B(z,2r)}
\left|k_{z_j}(u)\right|^{p} e^{-\frac{p}{2} \Psi(|u|^2)}\,dV(u).$$
Similarly, for $w\in B(z_i,r)$,
\begin{align}\label{t11}
\left|k_{z_{k}}(w)\right|e^{-\frac{1}{2}\Psi(|w|^2)}
\lesssim |B(z_i,2r)|^{-\frac{1}{p}}S_k(z_i)^{\frac{1}{p}}.
\end{align}
Thus, by \eqref{t10} and \eqref{t11}, we have
\begin{align*}
&\left(\sum_{i=1}^{\infty}\int_{B(z_{i},r)}|k_{z_j}(w)k_{z_k}(w)|e^{-\Psi(|w|^2)} d \mu(w)\right)^{p}\\
&\lesssim\left(\sum_{i=1}^{\infty}\mu(B(z_{i},r))|B(z_i,2r)|^{-\frac{2}{p}}S_j(z_i)^{\frac{1}{p}}
S_k(z_i)^{\frac{1}{p}}
\right)^{p}\\
&\lesssim \sum_{i=1}^{\infty} \widehat{\mu}_r(z_i)^p|B(z_i,r)|^{p-2}S_j(z_i)
S_k(z_i),
\end{align*}
which, together with \eqref{t9}, further implies that
\begin{align}\label{t12}
\|T_{2}\|_{S_{p}}^{p}
&\lesssim \sum_{i=1}^{\infty}\widehat{\mu}_r(z_i)^p|B(z_i,r)|^{p-2} \sum_{j=1}^{\infty}\sum_{\{k:k \neq j\}} S_j(z_i)S_k(z_i).
\end{align}
We are going to show that, for each $i$,
\begin{align}\label{t13}
\sum_{j=1}^{\infty}\sum_{\{k:k \neq j\}} S_j(z_i)S_k(z_i)=o(1)|B(z_i,r)|^{2-p} \textup{ as } R\to\infty.
\end{align}
For each fixed $i$, there holds
\begin{align}\label{t14}
&\sum_{j=1}^{\infty}\sum_{\{k:k \neq j\}} S_j(z_i)S_k(z_i)=\sum_{\{k:k \neq i\}} S_i(z_i)S_k(z_i)+\sum_{\{j:j\neq i\}}\sum_{\{k:k \neq j\}} S_j(z_i)S_k(z_i)\notag\\
&=\sum_{\{k:k \neq i\}} S_i(z_i)S_k(z_i)+\sum_{\{j:j\neq i\}}\sum_{\{k:k =i\}} S_j(z_i)S_k(z_i)+\sum_{\{j:j\neq i\}}\sum_{\{k:k\neq i,k\neq j\}} S_j(z_i)S_k(z_i)\notag\\
&\leq S_i(z_i)\sum_{\{k:k \neq i\}} S_k(z_i)+S_i(z_i)\sum_{\{j:j\neq i\}} S_j(z_i)+\sum_{\{j:j\neq i\}}\sum_{\{k:k \neq i\}} S_j(z_i)S_k(z_i)\notag\\
&= S_i(z_i)\sum_{\{k:k \neq i\}} S_k(z_i)+S_i(z_i)\sum_{\{j:j\neq i\}} S_j(z_i)+\left(\sum_{\{j:j\neq i\}}S_j(z_i)\right)\left(\sum_{\{k:k \neq i\}} S_k(z_i)\right).
\end{align}
By Lemma \ref{2p1}, Lemma \ref{2p2} and \eqref{2e3}, we know that
\begin{align}\label{t15}
  S_i(z_i)\simeq \int_{B(z_i,2r)}|K(u,u)|^{\frac{p}{2}}e^{-\frac{p}{2}\Psi(|u|^2)}\,dV(u)\simeq \int_{B(z_i,2r)}|B(u,r)|^{-\frac{p}{2}}\,dV(u)\simeq |B(z_i,r)|^{1-\frac{p}{2}}.
\end{align}
By Lemma \ref{qt8}, we get
\begin{align}\label{t16}
  \sum_{\{j:j\neq i\}}S_j(z_i)\leq \sum_{\{j:j\neq i\}}\sup_{u\in B(z_i,2r)}|k_{z_j}(u)|^{p} e^{-\frac{p}{2} \Psi(|u|^2)}|B(z_i,2r)|\lesssim o(1)|B(z_i,r)|^{1-\frac{p}{2}} \textup{ as } R\to \infty.
\end{align}
From \eqref{t14}, \eqref{t15} and \eqref{t16}, we see that \eqref{t13} holds. Therefore, \eqref{t12} and \eqref{t13} yield
\begin{align}\label{t17}
\|T_{2}\|_{S_{p}}^{p}\lesssim o(1)\sum_{i=1}^{\infty}\widehat{\mu}_r(z_i)^p\textup{ as } R\to \infty.
\end{align}
By choosing $R$ to be sufficiently large and combining \eqref{t6}, \eqref{t7}, \eqref{t8} and \eqref{t17}, we arrive at
\begin{align}\label{t18}
\sum_{i=1}^{\infty}\widehat{\mu}_r(z_i)^p\lesssim \left\|T_{\mu}\right\|_{S_{p}}<\infty .
\end{align}
The norm estimate \eqref{t4} comes from \eqref{ext8}, \eqref{t18} and Lemma \ref{4p4}. The proof is completed.
	\end{proof}
\section{Schatten Class Hankel Operators}\label{s6}
Let $z \in \mathbb{C}^n$ and $r>0$. Consider the space $L^2_\Psi(D(z,r)):=L^2(D(z,r),e^{-\Psi(|z|^2)}\,dV(z))$ and its closed subspace of holomorphic functions $F^2_\Psi(D(z,r))$. Let $P_{z,r}$ be the projection of $L^2_\Psi(D(z,r))$ onto $F^2_\Psi(D(z,r))$. Given a function $f \in L^2_\Psi(D(z,r))$, we extend $P_{z,r}(f)$ to $\mathbb{C}^n$ by setting $P_{z,r}(f)(w)=0$ for $w \in \mathbb{C}^n\backslash D(z,r)$.
	\begin{lemma}\label{qt10}
		Let $f,g \in L^2_\Psi$, $0<r<\infty$ and $z\in \mathbb{C}^n$. Then
		\begin{align*}
			\langle f-P(f),\chi_{D(z,r)}g-P_{z,r}(g)\rangle=\langle \chi_{D(z,r)}f-P_{z,r}(f),\chi_{D(z,r)}g-P_{z,r}(g)\rangle.
		\end{align*}
	\end{lemma}
\begin{proof}
  See \cite[Lemma 3.6]{ZWH}.
\end{proof}
\begin{lemma}\label{qt11}
  Suppose $0<l,r<\infty$ and $0<p\leq 1$. Let $\mu$ be a positive Borel measure on $\mathbb{C}^n$. For $f\in H(\mathbb{C}^n)$, there holds
  \begin{align}\label{b1}
    \left(\int_{\mathbb{C}^n}\left|f(z)e^{-\frac{1}{2}\Psi(|z|^2)}\right|^l\,d\mu(z)\right)^p\lesssim \int_{\mathbb{C}^n}\left|f(z)e^{-\frac{1}{2}\Psi(|z|^2)}\right|^{lp}\widehat{\mu}_r(z)^p|B(z,r)|^{p-1}\,dV(z).
  \end{align}
\end{lemma}
\begin{proof}
  Given $0<r_1<r_2<\infty$, we see that $\widehat{\mu}_{r_1}(z)\leq C\widehat{\mu}_{r_2}(z)$ for all $z\in \mathbb{C}^n$ where $C=C(r_1,r_2)>0$ is independent of $z$. Let $r_0$ be the same constant as in Lemma \ref{2p3}, then it suffices to deal with the case $0<r\leq r_0$. Let $\{z_k\}_{k=1}^\infty$ be a $\Psi$-lattice with covering radius $r/4$. By the fact that $(a+b)^p\leq a^p+b^p$ for $a,b>0$ and Lemma \ref{2p3}, we obtain
  \begin{align}\label{b2}
    \left(\int_{\mathbb{C}^n}\left|f(z)e^{-\frac{1}{2}\Psi(|z|^2)}\right|^l\,d\mu(z)\right)^p
    &\leq \sum_{k=1}^{\infty} \left(\int_{B(z_k,r/4)}\left|f(z)e^{-\frac{1}{2}\Psi(|z|^2)}\right|^l\,d\mu(z)\right)^p\notag\\
    &\leq \sum_{k=1}^{\infty}\sup_{z\in B(z_k,r/4)}\left|f(z)e^{-\frac{1}{2}\Psi(|z|^2)}\right|^{lp}\mu(B(z_k,r/2))^p\notag\\
    &\lesssim \sum_{k=1}^{\infty}\int_{B(z_k,r/2)}\left|f(z)e^{-\frac{1}{2}\Psi(|z|^2)}\right|^{lp}\,dV(z)\mu(B(z_k,r/2))^p|B(z_k,r)|^{-1}.
  \end{align}
  Notice that, if $z\in B(z_k,r/2)$, then $B(z_k,r/2)\subseteq B(z,r)$. By this, \eqref{b2} and \eqref{2e1}, we deduce that
  \begin{align}\label{b3}
    \left(\int_{\mathbb{C}^n}\left|f(z)e^{-\frac{1}{2}\Psi(|z|^2)}\right|^l\,d\mu(z)\right)^p
    &\lesssim \sum_{k=1}^{\infty}\int_{B(z_k,r/2)}\left|f(z)e^{-\frac{1}{2}\Psi(|z|^2)}\right|^{lp}\widehat{\mu}_r(z)^p|B(z,r)|^{p-1}\,dV(z)\notag\\
    &\lesssim \int_{\mathbb{C}^n}\left|f(z)e^{-\frac{1}{2}\Psi(|z|^2)}\right|^{lp}\widehat{\mu}_r(z)^p|B(z,r)|^{p-1}\,dV(z).
  \end{align}
    The proof is completed.
\end{proof}
	\begin{proof}[Proof of Theorem \ref{1p3}]
		(A) $\Rightarrow$ (B): We treat the case $2\leq p<\infty$ first. Assume that $r$ is sufficiently small. Let $\{z_k\}_{k=1}^\infty$ be a sequence in $\mathbb{C}^n$ such that $D(z_j,r/m)\cap D(z_k,r/m)=\emptyset$ for $j\neq k$ where $m>0$ is a constant. Let $t$ be a constant as in \eqref{3e18}.
		We define linear operators $J$ and $B$ by setting
		$$Jg=\sum_{k=1}^\infty \langle g,e_k \rangle k_{z_k}\textup{ for } g\in F^2_\Psi,$$
and
        $$B=\sum_{k=1}^{\infty}H_f^*M_{\chi_{D(z_k,4t^2r)}}H_f:\,F^2_\Psi \to F^2_\Psi.$$
Thus, from Lemma \ref{3p4}, we know that, for arbitrary $g\in F^2_\Psi$,
		\begin{align*}
		\langle Bg,g \rangle\leq \sum_{k=1}^{\infty}\|\chi_{D(z_k,4t^2r)}H_fg\|^2_\Psi\lesssim \|H_fg\|^2_\Psi=\langle H_f^*H_fg,g \rangle,
		\end{align*}
which implies that $B \in S_{\frac{p}{2}}$ and
		\begin{align}\label{4e2}
			\|B\|_{S_{\frac{p}{2}}}^{\frac{p}{2}}\leq C\|H_f^*H_f\|_{S_{\frac{p}{2}}}^{\frac{p}{2}}=\||H_f|\|_{S_{p}}^p=\|H_f\|_{S_{p}}^p.
		\end{align}
By \cite[Lemma 8.3]{SY}, $J:F^2_\Psi \to F^2_\Psi$ is bounded. Combining this fact with \eqref{4e2}, we obtain
\begin{align}\label{4e3}
			\|J^*BJ\|_{S_{\frac{p}{2}}}^{\frac{p}{2}}\leq C\|H_f\|_{S_{p}}^p.
		\end{align}
Thus, by using \cite[Theorem 1.27]{Zhu}, we deduce that
		\begin{align}\label{4e3n}
			\sum_{k=1}^{\infty}G_{4t^2r}(f)(z_k)^p&\leq C\sum_{k=1}^{\infty}\|\chi_{D(z_k,4t^2r)}H_fk_{z_k}\|_\Psi^p\notag\\
&=C\sum_{k=1}^{\infty}\langle M_{\chi_{D(z_k,4t^2r)}}H_fJe_{k},H_fJe_k \rangle^\frac{p}{2}\notag\\
&=C\sum_{k=1}^{\infty}\langle J^*H_f^*M_{\chi_{D(z_k,4t^2r)}}H_fJe_{k},e_k \rangle^\frac{p}{2}\notag\\
&\leq C\sum_{k=1}^{\infty}\langle J^*BJe_{k},e_k \rangle^\frac{p}{2}\notag\\
&\leq C\|J^*BJ\|_{S_{\frac{p}{2}}}^\frac{p}{2}.
		\end{align}
This and \eqref{4e3} imply that
\begin{align}\label{4e4}
			\sum_{k=1}^{\infty}G_{4t^2r}(f)(z_k)^p\leq C\|H_f\|_{S_{p}}^p.
		\end{align}
By \eqref{3e18}, for $z\in D(z_k,r)$, we have
$$D(z,r)\subseteq B(z,2tr)\subseteq B(z_k,4tr)\subseteq D(z_k,4t^2r).$$
By this and \eqref{4e4}, we arrive at
		\begin{align}\label{4e5}
			\int_{\mathbb{C}^n}G_r(f)(z)^pd\nu_\Psi(z)
&\leq \sum_{k=1}^{\infty}\int_{D(z_k,r)}G_r(f)(z)^pd\nu_\Psi(z)\notag\\
&\leq C\sum_{k=1}^{\infty}\sup_{z\in D(z_k,r)}G_r(f)(z)^p\notag\\
&\leq C\sum_{k=1}^{\infty}G_{4t^2r}(f)(z_k)^p\notag\\
&\leq C\|H_f\|_{S_p}^p.
		\end{align}
Now, we deal with the case $0<p<2$. Let $\{z_k\}_{k=1}^\infty$ be a $\Psi$-lattice. Given $R>0$, according to Lemma \ref{qt7}, we can partition $\{z_k\}_{k=1}^\infty$ into $M$ subsequences which satisfy that, if $z_j$ and $z_k$ are two different points in the same subsequence, then $\varrho(z_j,z_k)\geq R.$ It is enough to work with one of these subsequences. Without loss of generality, we denote the subsequence by $\{z_k\}_{k=1}^\infty$. Suppose that $\{e_k\}_{k=1}^\infty$ is an orthonormal basis of $F^2_\Psi$. Set
		$$J(g)=\sum_{k=1}^\infty \langle g,e_k \rangle k_{z_k}\textup{ for } g\in F^2_\Psi.$$
	 According to Lemma 8.3 of \cite{SY}, we know that $J$ is bounded on $F^2_\Psi$. Given $r>0$, as is stated in Section 2, we may assume that $D(z_j,r)\cap D(z_k,r)=\emptyset$ for $j\neq k$ when $R$ is chosen to be sufficiently large. Set
		\begin{align*}
			h_k=\left\{\begin{aligned}
                &\frac{\chi_{D(z_k,r)}fk_{z_k}-P_{z_k,r}(fk_{z_k})}
                {||\chi_{D(z_k,r)}fk_{z_k}-P_{z_k,r}(fk_{z_k})||_\Psi},
                \textup{  \,if } ||\chi_{D(z_k,r)}fk_{z_k}-P_{z_k,r}(fk_{z_k})||_{\Psi} \neq 0,\\
                &0,\quad\qquad\qquad\qquad\qquad\qquad\qquad \textup{if } ||\chi_{D(z_k,r)}fk_{z_k}-P_{z_k,r}(fk_{z_k})||_{\Psi}= 0.
			             \end{aligned}
                \right.	
		\end{align*}
		Clearly $||h_j||^2_\Psi \leq 1$. Notice that $\textup{supp }h_k\subseteq D(z_k,r)$ and $D(z_j,r) \cap D(z_k,r)=\emptyset $ for $j\neq k$. It follows that $\langle h_j,h_k \rangle =0$ for $j \neq k$. Define an operator $B$ on $F^2_\Psi$ as
		$$Bg=\sum_{k=1}^{\infty} \langle g,h_k \rangle e_k\textup{ for }g\in F^2_\Psi.$$
		It is easy to see that $B$ is bounded on $F^2_\Psi$ and $||B|| \leq 1$.
		Moreover,
		\begin{align*}
			BH_fJg=\sum_{k=1}^{\infty} \langle H_fJg, h_k \rangle e_k
			=\sum_{k=1}^\infty \sum_{j=1}^\infty \langle H_fk_{z_j}, h_k \rangle \langle g,e_j \rangle  e_k\textup{ for }g\in F^2_\Psi.
		\end{align*}
		Since $J$ and $B$ are bounded, we see that
		\begin{align}\label{t19}
			||BH_fJ||_{S_p} \leq C||H_f||_{S_p}.
		\end{align}
We decompose the operator $BH_fJ$ to the diagonal part defined by
		$$Yg=\sum_{k=1}^\infty \langle H_fk_{z_k}, h_k \rangle \langle g,e_k \rangle  e_k\textup{ for }g\in F^2_\Psi,$$
		and the non-diagonal part defined by
		$$Zg=\sum_{k=1}^\infty \sum_{\{j:j\neq k\}}\langle H_fk_{z_j}, h_k \rangle \langle g,e_j \rangle  e_k\textup{ for }g\in F^2_\Psi.$$
		Thus
\begin{align}\label{t20}
  \|Y\|^p_{S_p} \lesssim \|BH_fJ\|^p_{S_p}+\|Z\|^p_{S_p}.
\end{align}
		By Lemma \ref{2p1} and Lemma \ref{2p2}, we know that, if $r$ is sufficiently small and $w\in D(z_k,r)$, then
\begin{align}\label{t21}
  k_{z_k}(w)\simeq K(w,w)^{\frac{1}{2}}\simeq |D(w,r)|^{-\frac{1}{2}}e^{\frac{1}{2}\Psi(|w|^2)}>0.
\end{align}
Therefore, $k_{z_k}^{-1} \in H(D(z_k,r))$. From this fact, \eqref{t21} and Lemma \ref{qt10}, we deduce that
		\begin{align}\label{t22}
			||Y||^p_{S_p}&=\sum_{k=1}^\infty |\langle H_fk_{z_k},h_k \rangle|^p=\sum_{k=1}^\infty |\langle fk_{z_k}-P(fk_{z_k}),h_k \rangle|^p\notag\\
			&=\sum_{k=1}^\infty |\langle \chi_{D(z_k,r)}fk_{z_k}-P_{z_k,r}(fk_{z_k}),h_k \rangle|^p\notag\\
			&=\sum_{k=1}^\infty ||\chi_{D(z_k,r)}fk_{z_k}-P_{z_k,r}(fk_{z_k})||_{\Psi}^{p}\notag\\
			&=\sum_{k=1}^\infty \left\{\int_{D(z_k,r)}|f(w)k_{z_k}(w)-P_{z_k,r}(fk_{z_k})(w)|^2e^{-\Psi(|w|^2)}dV(w)\right\}^{p/2}\notag\\
			&\simeq \sum_{k=1}^\infty \left\{\frac{1}{|D(z_k,r)|}\int_{D(z_k,r)}|f(w)-k_{z_k}^{-1}(w)P_{z_k,r}(fk_{z_k})(w)|^2dV(w)\right\}^{p/2}\notag\\
			&\geq \sum_{k=1}^\infty G_r(f)(z_k)^p.
		\end{align}
By Proposition 1.29 of \cite{Zhu}, Lemma \ref{qt10} and Cauchy-Schwarz inequality, we have
		\begin{align*}
			||Z||^p_{S_p} &\leq \sum_{j=1}^\infty\sum_{k=1}^\infty |\langle Ze_j, e_k \rangle|^p=\sum_{k=1}^\infty \sum_{\{j:j\neq k\}} |\langle H_fk_{z_j}, h_k \rangle |^p\\
			&=\sum_{k=1}^\infty \sum_{\{j:j\neq k\}} |\langle \chi_{D(z_k,r)}fk_{z_j}-P_{z_k,r}fk_{z_j}, h_k \rangle |^p\\
			& \leq \sum_{k=1}^\infty \sum_{\{j:j\neq k\}} || \chi_{D(z_k,r)}fk_{z_j}-P_{z_k,r}fk_{z_j}||_{\Psi}^p\\
			&=\sum_{k=1}^\infty \sum_{\{j:j\neq k\}} \left\{\int_{D(z_k,r)}|f(w)k_{z_j}(w)-P_{z_k,r}(fk_{z_j})(w)|^2e^{-\Psi(|w|^2)}dV(w)\right\}^{p/2}\\
			&\leq \sum_{k=1}^\infty \sum_{\{j:j\neq k\}} \left\{\int_{D(z_k,r)}|f(w)k_{z_j}(w)-k_{z_j}(w)D_{z_k,r}(f)(w)|^2e^{-\Psi(|w|^2)}dV(w)\right\}^{p/2}.	
		\end{align*}
		where $D_{z,r}$ is the projection of $L^2(D(z,r),dV)$ onto $A^2(D(z,r),dV)$.
By this, Lemma \ref{ex3p6} and Lemma \ref{qt8}, we obtain
		\begin{align}\label{t23}
			||Z||^p_{S_p}
			&\leq \sum_{k=1}^\infty \sum_{\{j:j\neq k\}} \left\{\int_{D(z_k,r)}|k_{z_j}(w)|^2e^{-\Psi(|w|^2)}|f(w)-D_{z_k,r}(f)(w)|^2dV(w)\right\}^{p/2}\notag\\
			&\lesssim \sum_{k=1}^\infty \sum_{\{j:j\neq k\}} |D(z_k,r)|^{\frac{p}{2}} \left\{\frac{1}{|D(z_k,r)|}\int_{D(z_k,r)}|f(w)-D_{z_k,r}(f)(w)|^2dV(w)\right\}^{p/2}\notag \\
&\qquad\qquad\times \sup_{w\in D(z_k,r)}|k_{z_j}(w)|^pe^{-\frac{p}{2}\Psi(|w|^2)}\notag\\
			&\simeq \sum_{k=1}^\infty G_r(f)(z_k)^p|D(z_k,r)|^{\frac{p}{2}}\sum_{\{j:j\neq k\}}\sup_{w\in D(z_k,r)}|k_{z_j}(w)|^pe^{-\frac{p}{2}\Psi(|w|^2)}\notag\\
&\lesssim o(1)\sum_{k=1}^\infty G_r(f)(z_k)^p,
		\end{align}
as $R\to \infty$. From \eqref{t19}, \eqref{t20}, \eqref{t22} and \eqref{t23}, we see that, if $R$ is chosen to be sufficiently large, then there exists $r>0$ such that
\begin{align}\label{t24}
  \sum_{k=1}^\infty G_r(f)(z_k)^p\lesssim \|H_f\|_{S_p}^p.
\end{align}
From this, with the same approach as in \eqref{4e5}, we get
		\begin{align}\label{t25}
			\int_{\mathbb{C}^n}G_r(f)(z)^pd\nu_\Psi(z)\lesssim C\|H_f\|_{S_p}^p.
		\end{align}
		To prove that (B) implies (A), we decompose $f$ as in Lemma \ref{3p7}. By the assumption $G_r(f)\in L^p(\mathbb{C}^n,d\nu_\Psi)$, we know that
\begin{align}\label{4e6}
  f_1\in C^2(\mathbb{C}^n),\qquad
M_{\frac{r}{16t^4}}\left(\frac{|\overline{\partial}f_1(\cdot)|}{\Phi'(|\cdot|^2)^{\frac{1}{2}}}
+\frac{|\overline{\partial}f_1(\cdot)\land \overline{\partial}|\cdot||}{\Psi'(|\cdot|^2)^{\frac{1}{2}}}\right)(z)\lesssim
G_r(f)(z),
\end{align}
and
\begin{align}\label{4e7}
  f_2\in L^2_{\mathrm{loc}}(\mathbb{C}^n),\qquad M_{\frac{r}{16t^4}}(f_2)(z)\lesssim G_r(f)(z).
\end{align}
Then $H_{f_1}$ and $H_{f_2}$ are well defined on $\Gamma= \text{span}\{K_z:z\in \mathbb{C}^n\}$. In fact, by Lemma \ref{2p4}, for $g\in \Gamma$ and $z\in \mathbb{C}^n$, we have
\begin{align}\label{4e8}
  \int_{\mathbb{C}^n}|f_2(w)g(w)K_z(w)|e^{-\Psi(|w|^2)}dV(w)\leq C\int_{\mathbb{C}^n}|g(w)K_z(w)|e^{-\Psi(|w|^2)}\widehat{|f_2|}_{\frac{r}{16t^4}}(w)dV(w).
\end{align}
H\"{o}lder's inequality yields
\begin{align}\label{4e9}
  \widehat{|f_2|}_{\frac{r}{16t^4}}(w)\leq M_{\frac{r}{16t^4}}(|f_2|)(w).
\end{align}
Notice that
\begin{align}\label{4e10}
  e^{-\Psi(|w|^2)}\leq 1\lesssim \Psi'(|w|^2)\Phi'(|w|^2).
\end{align}
Let $p'$ be the conjugate number of $p$. Then by \eqref{4e8}, \eqref{4e9}, \eqref{4e10} and using H\"{o}lder's inequality twice, we arrive at
\begin{align*}
  &\int_{\mathbb{C}^n}|f_2(w)g(w)K_z(w)|e^{-\Psi(|w|^2)}dV(w)\\
  &\leq C\left\{\int_{\mathbb{C}^n}|M_{\frac{r}{16t^4}}(|f_2|)(w)|^pe^{-\Psi(|w|^2)}dV(w)\right\}^{\frac{1}{p}}
  \left\{\int_{\mathbb{C}^n}|g(w)K_z(w)|^{p'}e^{-\Psi(|w|^2)}dV(w)\right\}^{\frac{1}{p'}}\\
  &\leq C\left\{\int_{\mathbb{C}^n}|M_{\frac{r}{16t^4}}(|f_2|)(w)|^p\,d\nu_\Psi(w)\right\}^{\frac{1}{p}}\\
  &\qquad\times\left\{\int_{\mathbb{C}^n}|g(w)|^{2p'}e^{-\Psi(|w|^2)}dV(w)\right\}^{\frac{1}{2p'}}
  \left\{\int_{\mathbb{C}^n}|K_z(w)|^{2p'}e^{-\Psi(|w|^2)}dV(w)\right\}^{\frac{1}{2p'}}\\
  &<\infty,
\end{align*}
where in the last inequality we used Lemma \ref{qt5}. This implies that $H_{f_2}$, and hence also $H_{f_1}=H_f-H_{f_2}$, are both well defined on $\Gamma.$
We consider multiplication operators on $F^2_\Psi$. Notice that the multiplication operator with symbol $\phi$ is given by
		$$M_\phi g=\phi g\textup{ for } g \in F^2_\Psi.$$
		For any $g,h \in F^2_\Psi$, there holds
		\begin{align*}
			\langle M^*_\phi M_\phi g, h\rangle =\langle  M_\phi g, M_\phi h\rangle =\langle   T_{|\phi|^2}g, h\rangle .
		\end{align*}
		Thus, $M^*_\phi M_\phi=T_{|\phi|^2}$ on $F^2_\Psi$.
		By Theorem 1.26 of \cite{Zhu}, we know that $M_\phi \in S_p$ if and only if  $M^*_\phi M_\phi=T_{|\phi|^2} \in S_{\frac{p}{2}}$.
		We are interested in the following two choices: $\phi(z)=f_2(z)$
		and
\begin{align}\label{4e11}
  \phi(z)=\frac{|\bar{\partial}f_1(z)|}{\Phi'(|z|^2)^{1/2}}+\frac{|\bar{\partial}f_1(z)\wedge \bar{\partial}|z||}{\Psi'(|z|^2)^{1/2}}.
\end{align}
Obviously, $M_{\frac{r}{16t^4}}(\phi)(z)^2=\widehat{|\phi|^2}_{\frac{r}{16t^4}}(z).$ Thus we have
\begin{align}\label{4e11n}
\left\|\widehat{|\phi|^2}_{\frac{r}{16t^4}}\right\|^{\frac{1}{2}}_{L^{\frac{p}{2}}(\mathbb{C}^n,d\nu_\Psi)}=
\left\|M_{\frac{r}{16t^4}}(\phi)\right\|_{L^{p}(\mathbb{C}^n,d\nu_\Psi)}.
\end{align}
		By Theorem \ref{4p5}, we know that $M^*_{f_2} M_{f_2}=T_{|f_2|^2}\in S_{\frac{p}{2}}$ and
		\begin{align}\label{4e12}
\|M^*_{\phi} M_{\phi}\|_{S_{\frac{p}{2}}}\simeq \left\|\widehat{|\phi|^2}_{\frac{r}{16t^4}}\right\|_{L^{\frac{p}{2}}(\mathbb{C}^n,d\nu_\Psi)}.
		\end{align}
If $\phi(z)=f_2(z)$, then by the definition $H_{f_2}=(I-P)M_{f_2}$, we deduce that for any $g\in F^2_\Psi$,
\begin{align}\label{ex4e12}
  \langle H_{f_2}^*H_{f_2}g,g\rangle \lesssim \langle M^*_{f_2} M_{f_2}g,g\rangle.
\end{align}
Combining this with \eqref{4e7}, \eqref{4e11n} and \eqref{4e12}, we deduce that
\begin{align}\label{ex4e12n}
  ||H_{f_2}||_{S_p}^p=||H^*_{f_2}H_{f_2}||_{S_{\frac{p}{2}}}^{\frac{p}{2}}\leq C\|M^*_{f_2} M_{f_2}\|_{S_{\frac{p}{2}}}^{\frac{p}{2}}\leq C \left\|M_{\frac{r}{16t^4}}(f_2)\right\|^p_{L^{p}(\mathbb{C}^n,d\nu_\Psi)}\lesssim \left\|G_r(f)\right\|^p_{L^{p}(\mathbb{C}^n,d\nu_\Psi)}.
\end{align}
		Now we consider the choice of (\ref{4e11}). Clearly for $g\in \Gamma$, $g\overline{\partial}f_1$ is a $\overline{\partial}$-closed (0,1) form with $L^2_{\mathrm{loc}}$ coefficients. It follows from Lemma \ref{3p8} that for any $ g\in \Gamma,$ there is a solution $S(g\overline{\partial}f_1)$ of the equation $\overline{\partial}v=g\overline{\partial}f_1$ such that
\begin{align}\label{4e13}
  \int_{\mathbb{C}^n}|S(g\overline{\partial}f_1)(z)|^2e^{-\Psi(|z|^2)}dV(z)\leq \int_{\mathbb{C}^n}\left(\frac{|g(z)\overline{\partial}f_1(z)|^2}{\Phi '(|z|^2)}+\frac{|g(z)\overline{\partial}f_1(z)\land\overline{\partial}|z||^2}{\Psi '(|z|^2)}\right)e^{-\Psi(|z|^2)}dV(z).
\end{align}
On the other hand, we have
	\begin{align}\label{4e14}
        \|H_{f_1}g\|^2_\Psi=\|S(g\overline{\partial}f_1)\|^2_\Psi-\|H_{f_1}g-S(g\overline{\partial}f_1)\|^2_\Psi\leq \|S(g\overline{\partial}f_1)\|^2_\Psi,
	\end{align}
which, together with \eqref{4e13}, further implies that
\begin{align}\label{ex4e14}
  \langle H_{f_1}^*H_{f_1}g,g\rangle= ||H_{f_1}g||^2_\Psi\lesssim ||M_\phi g||^2_\Psi=\langle M^*_{\phi} M_{\phi}g,g\rangle.
\end{align}
		Thus, we have
$$||H_{f_1}||_{S_p}^p=||H^*_{f_1}H_{f_1}||_{S_{\frac{p}{2}}}^{\frac{p}{2}}\leq C\|M^*_{\phi} M_{\phi}\|_{S_{\frac{p}{2}}}^{\frac{p}{2}}.$$	
This and \eqref{4e6} yield
\begin{align}\label{4e15}
  ||H_{f_1}||_{S_p}^p\leq C \left\|M_{\frac{r}{16t^4}}\left(\frac{|\overline{\partial}f_1(\cdot)|}{\Phi'(|\cdot|^2)^{\frac{1}{2}}}
+\frac{|\overline{\partial}f_1(\cdot)\land \overline{\partial}|\cdot||}{\Psi'(|\cdot|^2)^{\frac{1}{2}}}\right)\right\|^p_{L^{p}(\mathbb{C}^n,d\nu_\Psi)}\lesssim \left\|G_r(f)\right\|^p_{L^{p}(\mathbb{C}^n,d\nu_\Psi)}.
\end{align}
\eqref{ex4e12n} and \eqref{4e15} yield that $H_f=H_{f_1}+H_{f_2}$ is in $S_p$. Furthermore,
\begin{align}\label{4e16}
  ||H_{f}||_{S_p}\lesssim \left\|G_r(f)\right\|_{L^{p}(\mathbb{C}^n,d\nu_\Psi)}.
\end{align}
(B)$\Rightarrow$(C): We decompose $f=f_1+f_2$ as in Lemma \ref{3p7}. From \eqref{ex4e12} and \eqref{ex4e14}, we know that
\begin{align}\label{5e1}
  \|H_{f_1}k_z\|_\Psi^2\lesssim \langle M^*_{\phi} M_{\phi}k_z,k_z \rangle=\langle T_{\phi^2}k_z,k_z \rangle,
\end{align}
where
\begin{align*}
  \phi(z)=\frac{|\bar{\partial}f_1(z)|}{\Phi'(|z|^2)^{1/2}}+\frac{|\bar{\partial}f_1(z)\wedge \bar{\partial}|z||}{\Psi'(|z|^2)^{1/2}},
\end{align*}
and
\begin{align}\label{5e2}
  \|H_{f_2}k_z\|_\Psi^2\lesssim \langle M^*_{f_2} M_{f_2}k_z,k_z \rangle=\langle T_{|f_2|^2}k_z,k_z \rangle.
\end{align}
Assume $0<p<2$. From \eqref{5e1} and Lemma \ref{qt11} with $d\mu(w)=|\phi(w)|^2\,dV(w)$, we know that
\begin{align}\label{b4}
  \int_{\mathbb{C}^n}\|H_{f_1}k_z\|_\Psi^p \,d\nu_\Psi(z)&\lesssim \int_{\mathbb{C}^n}\left(\int_{\mathbb{C}^n}|k_z(w)|^2e^{-\Psi(|w|^2)}|\phi(w)|^2\,dV(w)\right)^{\frac{p}{2}}|B(z,r)|^{-1}\,dV(z)\notag\\
  &\lesssim \int_{\mathbb{C}^n}\int_{\mathbb{C}^n}|k_z(w)|^pe^{-\frac{p}{2}\Psi(|w|^2)}M_r(\phi)(w)^p|B(w,r)|^{\frac{p}{2}-1}\,dV(w)|B(z,r)|^{-1}\,dV(z).
\end{align}
From this, Fubini's theorem, Lemma \ref{2p2}, Lemma \ref{qt5}, \eqref{2e3} and \eqref{4e6}, we have
\begin{align}\label{b5}
  \int_{\mathbb{C}^n}\|H_{f_1}k_z\|_\Psi^p \,d\nu_\Psi(z)&\lesssim \int_{\mathbb{C}^n}e^{-\frac{p}{2}\Psi(|w|^2)}M_r(\phi)(w)^p|B(w,r)|^{\frac{p}{2}-1}\,dV(w)\int_{\mathbb{C}^n}|k_z(w)|^p|B(z,r)|^{-1}\,dV(z)\notag\\
  &\simeq \int_{\mathbb{C}^n}e^{-\frac{p}{2}\Psi(|w|^2)}M_r(\phi)(w)^p|B(w,r)|^{\frac{p}{2}-1}\,dV(w)\int_{\mathbb{C}^n}|K(z,w)|^pe^{-\frac{p}{2}\Psi(|z|^2)}|B(z,r)|^{\frac{p}{2}-1}\,dV(z)\notag\\
  &\lesssim \int_{\mathbb{C}^n}M_r(\phi)(w)^p\,d\nu_\Psi(w)\notag\\
  &\lesssim \int_{\mathbb{C}^n}G_r(f)(w)^p\,d\nu_\Psi(w).
\end{align}
Similarly, we have
\begin{align}\label{b6}
  \int_{\mathbb{C}^n}\|H_{f_2}k_z\|_\Psi^p \,d\nu_\Psi(z)\lesssim  \int_{\mathbb{C}^n}M_r(f_2)(w)^p\,d\nu_\Psi(w)
  \lesssim \int_{\mathbb{C}^n}G_r(f)(w)^p\,d\nu_\Psi(w).
\end{align}

Suppose $2\leq p<\infty.$ \eqref{5e1}, \eqref{5e2} and Proposition 1.31 of \cite{Zhu} yield that
\begin{align}\label{5e3}
  \|H_{f_1}k_z\|_\Psi^p\lesssim \langle T_{\phi^2}^{\frac{p}{2}}k_z,k_z \rangle,
\end{align}
and
\begin{align}\label{5e4}
  \|H_{f_2}k_z\|_\Psi^p\lesssim \langle T_{|f_2|^2}^{\frac{p}{2}}k_z,k_z \rangle.
\end{align}
From Lemma 8.4 of \cite{SY} and Lemma \ref{2p2}, we know that
\begin{align}\label{5e5}
  \left\|T_{\phi^2}^{\frac{p}{2}}\right\|_{S_1}=\int_{\mathbb{C}^n}\langle T_{\phi^2}^{\frac{p}{2}}k_z,k_z \rangle \,d\nu_\Psi(z)\textup{ and } \left\|T_{|f_2|^2}^{\frac{p}{2}}\right\|_{S_1}=\int_{\mathbb{C}^n}\langle T_{|f_2|^2}^{\frac{p}{2}}k_z,k_z \rangle \,d\nu_\Psi(z).
\end{align}
By \eqref{5e3}, \eqref{5e5} and Theorem \ref{4p5}, we have
\begin{align}\label{5e6}
  \int_{\mathbb{C}^n}\|H_{f_1}k_z\|_\Psi^p \,d\nu_\Psi(z)&\lesssim \left\|T_{\phi^2}^{\frac{p}{2}}\right\|_{S_1}=\left\|T_{\phi^2}\right\|^{\frac{p}{2}}_{S_{\frac{p}{2}}}
  \lesssim \left\|\widehat{|\phi|^2}_{\frac{r}{16t^4}}\right\|^{\frac{p}{2}}_{L^{\frac{p}{2}}(\mathbb{C}^n,d\nu_\Psi)}
  \notag\\
  &=
\left\|M_{\frac{r}{16t^4}}(\phi)\right\|^p_{L^{p}(\mathbb{C}^n,d\nu_\Psi)}\lesssim \left\|G_r(f)\right\|^p_{L^{p}(\mathbb{C}^n,d\nu_\Psi)}.
\end{align}
Similarly, by \eqref{5e4}, \eqref{5e5} and Theorem \ref{4p5}, we obtain
\begin{align}\label{5e7}
  \int_{\mathbb{C}^n}\|H_{f_2}k_z\|_\Psi^p \,d\nu_\Psi(z)&\lesssim \left\|T_{|f_2|^2}^{\frac{p}{2}}\right\|_{S_1}=\left\|T_{|f_2|^2}\right\|^{\frac{p}{2}}_{S_{\frac{p}{2}}}
  \lesssim \left\|\widehat{|f_2|^2}_{\frac{r}{16t^4}}\right\|^{\frac{p}{2}}_{L^{\frac{p}{2}}(\mathbb{C}^n,d\nu_\Psi)}
  \notag\\
  &=
\left\|M_{\frac{r}{16t^4}}(f_2)\right\|^p_{L^{p}(\mathbb{C}^n,d\nu_\Psi)}\lesssim \left\|G_r(f)\right\|^p_{L^{p}(\mathbb{C}^n,d\nu_\Psi)}.
\end{align}
By \eqref{5e6} and \eqref{5e7}, we see that
\begin{align}\label{5e8}
  \int_{\mathbb{C}^n}\|H_{f}k_z\|_\Psi^p \,d\nu_\Psi(z)\lesssim \left\|G_r(f)\right\|^p_{L^{p}(\mathbb{C}^n,d\nu_\Psi)}.
\end{align}
(C)$\Rightarrow$(B): From the proof of Theorem \ref{1p1}, we see that there is some $r_0>0$ such that for $0<r\leq r_0$
\begin{align}\label{5e9}
  G_r(f)(z)\lesssim \|H_fk_z\|_\Psi.
\end{align}
This implies that
\begin{align}\label{5e10}
  \int_{\mathbb{C}^n}G_r(f)(z)^p \,d\nu_\Psi(z)\lesssim \int_{\mathbb{C}^n}\|H_fk_z\|_\Psi^p \,d\nu_\Psi(z).
\end{align}
The proof is completed.
\end{proof}
%\begin{proof}[Proof of Corollary \ref{1p4}]
%  Suppose $H_f,H_{\overline{f}} \in S_p$. By Theorem \ref{1p3}, we know that
%  \begin{align}\label{5e12}
%  \int_{\mathbb{C}^n}
%    \|H_f(k_z)\|_\Psi^p+\|H_{\overline{f}}(k_z)\|_\Psi^p \,d\nu_\Psi(z) \simeq \|H_f\|_{S_p}+\|H_{\overline{f}}\|_{S_p}<\infty.
%  \end{align}
%  \eqref{5e11} and \eqref{5e12} yield that $MO_r(f)\in L^p(\mathbb{C}^n, d\nu_\Psi)$ and
%  \begin{align}\label{5e13}
%  \int_{\mathbb{C}^n}MO_r(f)(z)^p \,d\nu_\Psi(z)\lesssim \|H_f\|_{S_p}+\|H_{\overline{f}}\|_{S_p}.
%  \end{align}
%  Now suppose $MO_r(f)\in L^p(\mathbb{C}^n, d\nu_\Psi)$. By the definitions of $MO_r(f)$ and $G_r(f)$, we have
%  $$G_r(f)(z)\leq MO_r(f)(z) \textup{ and } G_r(\overline{f})(z)\leq MO_r(\overline{f})(z).$$
%  This, together with Theorem \ref{1p3} and the fact that $MO_r(f)(z)=MO_r(\overline{f})(z)$, implies that $H_f,H_{\overline{f}} \in S_p$ and
%  \begin{align}\label{5e14}
%  \|H_f\|_{S_p}+\|H_{\overline{f}}\|_{S_p}\lesssim \int_{\mathbb{C}^n}G_r(f)(z)^p+G_r(\overline{f})(z) \,d\nu_\Psi(z) \leq 2 \int_{\mathbb{C}^n}MO_r(f)(z)^p \,d\nu_\Psi(z).
%  \end{align}
%The norm estimate \eqref{1e6} follows from \eqref{5e13} and \eqref{5e14}. The proof is completed.
%\end{proof}

\end{document}